\documentclass{amsart}
\pdfoutput=1
\usepackage{amssymb}
\usepackage{fullpage}
\usepackage[outline]{contour}
\usepackage{tikz}
\usepackage[all]{xy}
\usetikzlibrary{decorations.pathreplacing,decorations.markings}
\usepackage[bookmarks=false]{hyperref}
\hypersetup{colorlinks, linkcolor=blue}

\newtheorem{thm}{Theorem}[section]
\newtheorem{prop}[thm]{Proposition}
\newtheorem{cor}[thm]{Corollary}
\newtheorem{lem}[thm]{Lemma}

\theoremstyle{definition}
\newtheorem{defn}[thm]{Definition}

\newtheorem{exmp}[thm]{Example}
\theoremstyle{remark}
\newtheorem{rmk}[thm]{Remark}

\makeatletter
\let\c@equation\c@thm
\makeatother
\numberwithin{equation}{section}

\newcommand{\conf}{\mathrm{Conf}}

\newcommand{\GL}{\mathrm{GL}}
\newcommand{\PGL}{\mathrm{PGL}}
\newcommand{\SL}{\mathrm{SL}}

\newcommand{\sgn}{\mathrm{sign}}

\newcommand{\DT}{\mathrm{DT}}

\renewcommand{\vec}[1]{\mathbf{#1}}
\newcommand{\inprod}[2]{\left\langle#1,#2\right\rangle}

\tikzset{>=stealth}

\allowdisplaybreaks

\title{Donaldson-Thomas Transformation of Double Bruhat Cells in General Linear Groups}
\author{Daping Weng}
\date{\today}

\begin{document}

\begin{abstract} Kontsevich and Soibelman \cite{KS} defined the Donaldson-Thomas invariants of a 3d Calabi-Yau category with a stability condition. Any cluster variety can produce an example of such a category, whose corresponding Donaldson-Thomas invariants are encoded by a special formal automorphism of the cluster variety, known as the Donaldson-Thomas transformation.

In this paper we prove a conjecture of Goncharov and Shen \cite{GS} in the case of $\GL_n$, which describes the Donaldson-Thomas transformation of the double quotient of the double Bruhat cells $H \backslash \GL_n^{u,v}/H$ where $H$ is a maximal torus, as a certain explicit cluster transformation related to Fomin-Zelevinsky's twist map \cite{FZ}. Our result, combined with the work of Gross, Hacking, Keel, and Kontsevich \cite{GHKK}, proves the duality conjecture of Fock and Goncharov \cite{FG} in the case of $H\backslash \GL_n^{u,v}/H$.
\end{abstract}

\maketitle

\tableofcontents

\section{Introduction}

Donaldson-Thomas invariants were first introduced by Donaldson and Thomas \cite{DT} as geometric invariants on a Calabi-Yau threefold, and then later were generalized to 3d Calabi-Yau categories with stable conditions by Kontsevich and Soibelman \cite{KS}. It is known that examples of such 3d Calabi-Yau categories can be constructed from cluster varieties, whose Donaldson-Thomas invariants are typically encoded by a special formal automorphism of the cluster variety, also known as the Donaldson-Thomas transformation \cite{KS}. Keller \cite{KelDT} gave a combinatorial characterization of a subclass of Donaldson-Thomas transformations using quiver mutations. Goncharov and Shen \cite{GS} gave an equivalent definition using tropical points of the cluster varieties.

Double Bruhat cells were first studied by Fomin and Zelevinsky in \cite{FZ}; Berenstein, Fomin, and Zelevinsky proved that the algebra of regular functions on double Bruhat cells in simply connected semisimple Lie groups are cluster algebras in \cite{BFZ}. On the other hand, Fock and Goncharov showed that double Bruhat cells in adjoint semisimple Lie groups are cluster Poisson varieties in \cite{FGD}; in the same paper they also showed that in the biggest double Bruhat cell, which is a Zariski open subset of the Lie group, the Poisson structure coincides with the Poisson-Lie structure defined by Drinfeld in \cite{D}. 

\subsection{Main Result} Goncharov and Shen \cite{GS} conjectured that for a semisimple Lie group $G$, the Donaldson-Thomas transformation on the double quotient of a double Bruhat cell $H\backslash G^{u,v}/H$ is a slight modification of Fomin and Zelevinsky's twist map which was defined in \cite{FZ}. In this paper, we will focus on such double quotients of double Bruhat cells in $\GL_n$ and prove Goncharov and Shen's conjecture in this case.

Fix a pair of opposite Borel subgroups $B_\pm$ of $\GL_n$. Let $H:=B_+\cap B_-$ be the corresponding maximal torus in $\GL_n$. Then with respect to these two Borel subgroups, $\GL_n$ admits two Bruhat decompositions 
\[
\GL_n=\bigsqcup_{w\in W}B_+wB_+=\bigsqcup_{w\in W}B_-wB_-,
\]
where $W$ denotes the Weyl group with respect to the maximal torus $H$. Given a pair of Weyl group elements $(u,v)$, the \textit{double Bruhat cell} $\GL_n^{u,v}$ is defined to be the intersection
\[
\GL_n^{u,v}:=B_+uB_+\cap B_-vB_-.
\]
Our main result is the following theorem.

\begin{thm}\label{main} Let $(u,v)$ be a pair of Weyl group elements. The Donaldson-Thomas transformation of the cluster Poisson variety $H\backslash\GL_n^{u,v}/H$ is a cluster transformation which can also be described as
\[
\DT:H\backslash x/H \mapsto H\backslash\left(\left[\overline{u}^{-1}x\right]_-^{-1}\overline{u}^{-1} x \overline{v^{-1}}\left[x\overline{v^{-1}}\right]^{-1}_+\right)^t/H;
\]
here $x=[x]_-[x]_0[x]_+$ is the Gaussian decomposition of a general linear matrix $x$ and $\overline{w}$ denotes the lift of a Weyl group element $w$ defined by Equation \eqref{wbar} (see also \cite{FZ}, \cite{BFZ}, and \cite{GS}).
\end{thm}

Note that the cluster Donaldson-Thomas transformation exhibited above, modulo the double quotients, differs from Fomin and Zelevinsky's twist map only by the anti-automorphism $x\mapsto x^\iota$ on $\GL_n$, which is defined on the generators as
\[
e_{\pm i}^\iota=e_{\pm i} \quad \quad \text{and} \quad \quad a^\iota=a^{-1} \quad \forall a\in H.
\]
We show in Subsection \ref{birational} that the anti-automorphism $x\mapsto x^\iota$ coincides with the involution $i_\mathcal{X}$ introduced by Goncharov and Shen in \cite{GS}, and hence Fomin and Zelevinsky's original twist map in turn coincides with Goncharov and Shen's involution $D_\mathcal{X}$, which is defined as $D_\mathcal{X}:=i_\mathcal{X}\circ \DT$.

A key point of our proof is that the cluster Donaldson-Thomas transformation can be obtained as the composition of the following three maps (plus a lift):
\begin{enumerate}
\item Fock and Goncharov's amalgamation map from \cite{FGD}, which gives a map $\chi$ from a cluster Poisson variety $\mathcal{X}$ to $H\backslash \GL_n^{u,v}/H$;
\item the result of Berenstein, Fomin, and Zelevinsky \cite{BFZ} saying that $\mathcal{O}(\GL_n^{u,v})$ is isomorphic to a cluster algebra gives a map $\psi$ from $\GL_n^{u,v}$ to a cluster $\mathcal{A}$-variety $\mathcal{A}$;
\item Fock and Goncharov's $p$ map from cluster ensemble \cite{FG}, which is a map going from the cluster $\mathcal{A}$-variety $\mathcal{A}$ back to the original cluster Poisson variety $\mathcal{X}$.
\end{enumerate}
Such composition can be summarized by the diagram below, and the choice of lifting from $H\backslash\GL_n^{u,v}/H$ to $\GL_n^{u,v}$ does not affect the outcome of the composition.
\[
\xymatrix{ & \GL_n^{u,v} \ar@{-->}[r]^\psi  \ar[d] & \mathcal{A}\ar[d]^p \\
\mathcal{X} \ar[r]^(0.3)\chi & H\backslash \GL_n^{u,v}/H & \mathcal{X}
}
\]

Although our main theorem is stated in terms of the general linear group $\GL_n$, it can be extended to $\SL_n$ and $\PGL_n$ without much effort: since $\SL_n$ is a subgroup of $\GL_n$ and $\PGL_n$ is a quotient of $\GL_n$, a pair of opposite Borel subgroups of $\GL_n$ naturally gives rise to a pair of opposite Borel subgroups in $\SL_n$ and $\PGL_n$; after passing to the double quotient by the corresponding maximal tori on both sides, we see that there is a natural identification
\[
H \backslash \GL_n^{u,v}/H\cong H \backslash \SL_n^{u,v}/H\cong H \backslash \PGL_n^{u,v}/H,
\]
as the center is killed by the double quotient. Thus our main theorem solves Goncharov and Shen's conjecture in the case of classical Lie group of Dynkin type $A_n$.

Here is an important application of our main result. We prove that the Donaldson-Thomas transformation on the cluster Poisson variety $H\backslash \GL_n^{u,v}/H$ is a cluster transformation. Therefore our result combined with the work of Gross, Hacking, Keel, and Kontsevich (Theorem 0.10 of \cite{GHKK}) proves the Fock and Goncharov's duality conjecture \cite{FG} in the case of $H\backslash \GL_n^{u,v}/H$.

\subsection{Structure of the Paper} We divide the rest of the paper into two sections. Section 2 contains all the preliminaries necessary for our proof of the main theorem. Subsection \ref{bruhat} introduces double Bruhat cells $\GL_n^{u,v}$ and related structures, most of which are similar to the original work of Fomin and Zelevinsky \cite{FZ} and Fock and Goncharov \cite{FGD}. Subsection \ref{flag} describes a link between double Bruhat cells and configurations of quadruples of Borel subgroups, which is the key to proving the cluster nature of our candidate map for Donaldson-Thomas transformation; such link and the proof of cluster nature are both credit to Shen. Subsection \ref{bipartite} introduces the bipartite method in the study of double Bruhat cells, variations of which can also be found in many literature such as the work of Postnikov \cite{Pos}, Goncharov and Kenyon \cite{GK}, Goncharov \cite{Gon}, Gektman, Shapiro, Tabachnikov, and Vainshtein \cite{GSTV}, Fock and Marshakov \cite{FM}, Rietsch and Williams \cite{RW}, Williams \cite{HW}, and ourselves \cite{W}. Subsections \ref{cluster} and \ref{1.4} focus on Fock and Goncharov's theory of cluster ensemble and tropicalization, the main source of reference of which is \cite{FG}.

Section 3 is the proof of our main theorem itself. Subsection \ref{formula} gives a formula of our candidate map which resembles Fomin and Zelevinsky's twist map. Subsection \ref{birational} uses the involutiveness of the twist map to give a quick proof of the fact that the double quotient of a double Bruhat cell is birationally equivalent to a cluster Poisson variety. Subsection \ref{positive} gives a description of minors in terms of cluster Poisson variables, which is similar to Postnikov's boundary measurement map for Grassmannian in \cite{Pos}. Subsections \ref{proof1} and \ref{proof2} prove that our candidate map satisfies the two defining properties of the cluster Donaldson-Thomas transformation, the former of which is due to a private conversation with Shen.

\subsection{Acknowledgements} The author is deeply grateful to his advisor Alexander Goncharov for introducing the author to the subject of cluster theory and suggesting this problem on cluster Donaldson-Thomas transformation, as well as all his insightful guidance throughout the process of solving this problem and his help in the revision of the paper. The author also would like to thank Linhui Shen for his help on understanding the cluster Donaldson-Thomas transformation and his inspiring comment on proving the cluster nature of our candidate map using the configuration of quadruples of Borel subgroups. This paper uses many ideas from the work of Goncharov \cite{Gon}, the work of Fock and Goncharov \cite{FG}, \cite{FGD}, \cite{FGI}, the work of Goncharov and Shen \cite{GS}, the work of Fomin and Zelevinsky \cite{FZ}, and the work of Postnikov \cite{Pos}; without the work of all these pioneers in the field this paper would not have been possible otherwise.

\section{Preliminaries} \label{pre}

\subsection{Double Bruhat Cells in General Linear Group}\label{bruhat}

For simplicity, throughout this paper we will fix $B_+$ to be the subgroup of upper triangular matrices and fix $B_-$ to be the subgroup of lower triangular matrices (but any other pair of opposite Borel subgroups will work as fine). Then $H:=B_+\cap B_-$ is the maximal torus consisting of diagonal matrices, and $U_\pm:=[B_\pm,B_\pm]$ form a pair of opposite maximal unipotent matrices. The choice of opposite Borel subgroups $B_\pm$ also determines a Coxeter generating set $S$ for the Weyl group $W:=N_{\GL_n}(H)/H$ of $\GL_n$. The Weyl group action on $H$ is the same as permutation on the entries of diagonal matrices; we will hence identify $W$ with the symmetric group $S_n$, under which elements of $S$ will be identified with adjacent transpositions $s_i$ permuting the $i$th and the $(i+1)$th entries.

It is known that the subset $G_0:=U_-HU_+$ is a Zariski open subset of $\GL_n$, and any element $x\in G_0$ can be uniquely written as
\begin{equation}\label{gauss}
x=[x]_-[x]_0[x]_+
\end{equation}
where $[x]_\pm \in U_\pm$ and $[x]_0\in H$; this is called the \textit{Gaussian decomposition} of the element $x\in G_0$.

With respect to the Borel subgroups $B_\pm$, $\GL_n$ admits the following two Bruhat decompositions:
\[
\GL_n=\bigsqcup_{w\in W} B_+wB_+ =\bigsqcup_{w\in W} B_-wB_-.
\]
Each constituent $B_\pm wB_\pm$ of these Bruhat decompositions is known as a Bruhat cell of $\GL_n$ (with respect to $B_\pm$). 

\begin{defn} For a pair of Weyl group elements $(u,v)$, the \textit{double Bruhat cell} of $\GL_n$ (with respect to the pair of opposite Borel subgroups $B_\pm$) is defined to be
\[
\GL_n^{u,v}:=B_+uB_+\cap B_-vB_-.
\]
\end{defn}

It follows straight from the definition that $\GL_n$ admits a \textit{double Bruhat decomposition} (with respect to $B_\pm$):
\[
\GL_n=\bigsqcup_{u,v\in W} \GL_n^{u,v}.
\]

Next let's consider the Weyl group $W$. Since $S$ is a generating set of $W$, using elements of $S$ as ``letters'' we can spell out ``words'' to represent Weyl group elements.

\begin{defn} A \textit{reduced word} of a Weyl group element $w$ is a finite sequence of positive integers between $1$ and $n-1$
\[
\vec{i}:=(i(1), i(2),\dots, i(l))
\]
that is the shortest among all sequences such that $s_{i(1)}s_{i(2)}\dots s_{i(l)}=w$; in particular, the number $l$ is called the \textit{length} of the Weyl group element $w$ and is denoted as $l(w)$.

We can extend the notion of reduced words to a pair of Weyl group elements. A \textit{reduced word} of the pair of Weyl group elements $(u,v)$ is a finite sequence 
\[
\vec{i}=(i(1),i(2),\dots, i(l))
\]
where $i(k)\in \{\pm 1, \dots, \pm (n-1)\}$, from which we obtain a reduced word of $u$ if we delete all the positive entries and then turn the negative entries to their positive counterparts, and obtain a reduced word of $v$ if we delete all the negative entries. By convention we will also call $l(u,v):=l$ the \textit{length} of the pair $(u,v)$, and it is not hard to see that $l(u,v)=l(u)+l(v)$, which does not depend on the choice of the reduced words.
\end{defn}

Given a reduced word $\vec{i}=(i(1),i(2),\dots, i(l))$ of a pair of Weyl group elements $(u,v)$ (resp. a single Weyl group element $w$), we define its \textit{opposite} reduced word to be $\vec{i}^\circ=(i(l),i(l-1), \dots, i(1))$; note that $\vec{i}^\circ$ is a reduced word $(u^{-1},v^{-1})$ (resp. $w^{-1}$).

The Lie algebra of $\GL_n$ is the space of all $n\times n$ complex matrices $\mathfrak{gl}_n$. Let $\{E_{i,j}\}$ be the standard vector basis of $\mathfrak{gl}_n$ where $E_{i,j}$ denotes the matrix consisting of all zeros except a single 1 at the entry in the $i$th row and $j$th column. Then using matrix exponentiation we can define the following special elements of $\GL_n$:
\[
  e_i:=\exp E_{i,i+1} \quad \quad \text{and}\quad \quad e_{-i}:=\exp E_{i+1,i} \quad \quad \text{(for $1\leq i\leq n-1$);}
\]
\[
 h^i(X):=\exp\left((\log X)\sum_{1\leq k\leq i}E_{k,k}\right) \quad \quad \text{(for $0\leq i\leq n$ and $X\in \mathbb{C}^*$).}
\]
By using these special elements we can define a lift for each Coxeter generator $s_i$ to $\GL_n$:
\[
\overline{s}_i:=e_i^{-1}e_{-i}e_i^{-1}.
\]
Alternatively, such a lift can also be defined as 
\[
\overline{s}_i:=\varphi_i\begin{pmatrix} 0 & -1 \\ 1 & 0\end{pmatrix},
\]
where $\varphi$ is the Lie group homomorphism $\SL_2\rightarrow \GL_n$ induced from the Lie algebra homomorphism $\mathfrak{sl}_2\rightarrow \mathfrak{gl}_n$ defined by
\[
E\mapsto E_{i,i+1} \quad \quad \text{and}\quad \quad F\mapsto E_{i+1,i}.
\]
Since these lifts satisfy the braid relations, it follows that any reduced word $s_{i(1)}s_{i(2)}\dots s_{i(l)}$ of a Weyl group element $w$ defines a lift
\begin{equation}\label{wbar}
\overline{w}:=\overline{s}_{i(1)}\overline{s}_{i(2)}\dots \overline{s}_{i(l)}
\end{equation}
of $w$ to $\GL_n$, and such a lift $\overline{w}$ is independent of the chosen reduced word (see also \cite{FZ}, \cite{BFZ}, and \cite{GS}).

Lastly, we define an automorphism $*$ on $W$ by $w^*:=w_0ww_0$ where $w_0$ is the longest element of $W$ (with respect to the length of reduced words). Note that $s_i^*=s_{n-i}$. We then lift this automorphism to $\GL_n$ by defining
\[
g^*:=\overline{w}_0\left(g^{-1}\right)^t\overline{w}_0^{-1}.
\]
Since $\overline{s}_i^{-1}=\overline{s}_i^t$, it follows that $\overline{s}_i^*=\overline{s_i^*}$. We also observe that this automorphism $*$ on $G$ is an involution and preserves the Borel subgroups $B_\pm$ respectively.

\subsection{Configuration of Quadruples of Borel Subgroups and Double Bruhat Cells}\label{flag}

In this subsection we will look at the double quotients of double Bruhat cells $H\backslash\GL_n^{u,v}/H$ from a different angle; all results within this subsection are stated for $\GL_n$, but they hold in general semisimple Lie groups as well. 

Recall that the space of Borel subgroups in $\GL_n$, which we will denote as $\mathcal{B}$, can be identified with either of the quotient spaces $\GL_n/B_\pm$, via the isomorphisms
\[
xB_+x^{-1}\mapsto xB_+ \quad \quad \text{and} \quad \quad xB_-x^{-1}\mapsto xB_-.
\]
Such isomorphisms can be understood as taking the ``stabilizer subgroup'' in one direction and taking the ``fixed point'' in the opposite direction (with respect to the natural group action on the spaces of cosets). For notation simplicity, we will not distinguish cosets and Borel subgroups in this paper, hence phrases like ``the Borel subgroup $xB_+$'' should make sense tautologically. 

\begin{defn} Let $B_1=x_1B_+=y_1B_-$ and $B_2=x_2B_+=y_2B_-$ be two Borel subgroups; we define two maps
\[
d_\pm: \mathcal{B}\times \mathcal{B} \rightarrow W
\]
with $d_+(B_1,B_2)=u$ if $x_1^{-1}x_2\in B_+uB_+$ and $d_-(B_1,B_2)=v$ if $y_1^{-1}y_2\in B_-vB_-$.
\end{defn} 

Our first observation is that these two maps are anti-symmetry on their arguments: $d_\pm(B_1,B_2)=w$ if and only if $d_\pm (B_2,B_1)=w^{-1}$. But there are more symmetries to these two maps, as we will see in the following several propositions.

\begin{prop} The maps $d_\pm$ are $\GL_n$-equivariant, i.e., $d_\pm(B_1,B_2)=d_\pm (xB_1x^{-1},xB_2x^{-1})$ for any $x\in \GL_n$.
\end{prop}
\begin{proof} Just note that the $x$ in the first argument becomes $x^{-1}$ after inverting and cancels with the $x$ in the second argument.
\end{proof}

Recall that we have an involutive automorphism $*$ defined on $\GL_n$, which can be extended naturally to the space of Borel subgroups $\mathcal{B}$. We then make the following observation.

\begin{prop} $d_\pm(B_1,B_2)=w$ if and only if $d_\pm(B_1^*,B_2^*)=w^*$.
\end{prop}
\begin{proof} It follows from the fact that $*$ is an automorphism and $B_\pm^*=B_\pm$.
\end{proof}

\begin{prop} Let $B_1$ and $B_2$ be two Borel subgroups. Then $d_+(B_1,B_2)=w$ if and only if $d_-(B_1,B_2)=w^*$.
\end{prop}
\begin{proof} Let's show one direction only, for the other direction is completely analogous. Suppose $B_1=xB_+$ and $B_2=yB_+$. Then $d_+(B_1,B_2)=w$ means that $x^{-1}y\in B_+wB_+$. On the other hand since $B_+=w_0B_-$, we know that $B_1=xw_0B_-$ and $B_2=yw_0B_-$; therefore we know that 
\[
\overline{w}_0x^{-1}y\overline{w}_0^{-1}\in w_0B_+wB_+w_0=B_-w^*B_-,
\]
which implies $d_-(B_1,B_2)=w^*$.
\end{proof}

Since $d_+$ already contains the information of $d_-$, we introduce the following more concise notation: we write $\xymatrix{B_1 \ar[r]^u & B_2}$ to mean $d_+(B_1,B_2)=u$; then $d_-(B_1,B_2)=v$ can be denoted as $\xymatrix{B_1 \ar[r]^{v^*} & B_2}$; moreover, since $w_0^*=w_0^{-1}=w_0$, we can simplify $\xymatrix{B_1 \ar[r]^{w_0} & B_2}$ to $\xymatrix{B_1 \ar@{-}[r] & B_2}$ without the arrow or the argument $w_0$.

\begin{prop}\label{opposite flag} Let $B_1$ and $B_2$ be two Borel subgroups. Then the followings are equivalent:
\begin{enumerate}
\item $B_1$ and $B_2$ are opposite Borel subgroups;
\item $\xymatrix{B_1 \ar@{-}[r] & B_2}$;
\item there exists an element $g\in \GL_n$ such that $B_1=gB_+$ and $B_2=gB_-$.
\end{enumerate}
Further the choice of $g$ in (3) is unique up to a right multiple of an element from $H$.
\end{prop}
\begin{proof} (1)$\implies$(2) follows from the fact that conjugation preserves pairs of opposite Borel subgroups.

(2)$\implies$(3): Suppose $B_1=xB_+$ and $B_2=yB_+$. Then by assumption $x^{-1}y\in B_+w_0B_+$. Thus we can find $b$ and $b'$ from $B_+$ such that $x^{-1}y=b\overline{w}_0b'$. Let $g:=xb$; then 
\[
gB_+=xB_+=B_1 \quad \quad \text{and} \quad \quad gB_-=xbw_0B_+=yB_+=B_2.
\]

(3)$\implies$(1) is trivial since $gB_+$ and $gB_-$ are obviously opposite Borel subgroups.

For the remark on the uniqueness of $g$, note that if $gB_+=g'B_+$ and $gB_-=gB_-$, then $g^{-1}g'$ is in both $B_+$ and $B_-$; but then since $B_+\cap B_-=H$, it follows that $g$ and $g'$ can only differ by a right multiple of an element from $H$.
\end{proof}

\begin{prop}\label{2.8} Suppose $l(uv)=l(u)+l(v)$. Then $\xymatrix{B_1\ar[r]^{uv} & B_2}$ if and only if there exists a Borel subgroup $B_3$ such that $\xymatrix{B_1 \ar[r]^u & B_3}$ and $\xymatrix{B_3 \ar[r]^v & B_2}$. In particular, such a Borel subgroup $B_3$ is unique.
\end{prop}
\begin{proof} The existence part follows from the general fact about semisimple Lie groups that 
\[
(B_+uB_+)(B_+vB_+)=B_+uvB_+
\]
whenever $l(uv)=l(u)+l(v)$ (see for example \cite{Hum} Section 29.3 Lemma A). The unique part follows from the following lemma.
\end{proof}

\begin{lem} Let $B$ be a Borel subgroup of $\GL_n$. If $x\in BwB$ where $w=s_{\alpha(1)}s_{\alpha(2)}\dots s_{\alpha(l)}$ is a reduced word for $w$, then there exists $x_k\in Bs_{\alpha(k)}B$ such that $x=x_1x_2\dots x_l$; further if $x=x'_1x'_2\dots x'_l$ is another such factorization then $x_k^{-1}x'_k\in B$ for all $1\leq k\leq l-1$ and $x'_kx_k^{-1}\in B$ for all $2\leq k\leq l$.
\end{lem}
\begin{proof} The existence part is essentially the same as the existence part of the above proposition, so it suffices to show the uniqueness part. We will do an induction on $l$. There is nothing to show for the base case $l=1$. Suppose $l>1$. Let $x=yx_l=y'x'_l$ where both $y$ and $y'$ are in $Bs_{\alpha(1)}\dots s_{\alpha(l-1)}B$ and both $x_l$ and $x'_l$ are in $Bs_{\alpha(l)}B$. Then from the fact that $(Bs_{\alpha(l)}B)^2\subset B\cup Bs_{\alpha(l)}B$ we know that $x'_lx_l^{-1}$ is in either $B$ or $Bs_{\alpha(l)}B$. To rule out the latter possibility, note that if $x'_lx_l^{-1}\in Bs_{\alpha(l)}B$, then $y'x'_lx_l^{-1}=xx_l^{-1}=y$ is in both Bruhat cells $Bs_{\alpha(1)}\dots s_{\alpha(l-1)}B$ and $BwB$, which contradicts the Bruhat decomposition. Thus $x'_lx_l^{-1}\in B$.

The fact that $x'_lx_l^{-1}\in B$ implies that $y^{-1}y'=x_lx^{-1}x{x'_l}^{-1}=x_l{x'_l}^{-1}\in B$. Thus $y$ and $y'$ can only differ by a right multiple of $B$. This difference can be absorbed into the right ambiguity of $x_{l-1}$, and hence without loss of generality one can assume that $y=y'$, and the proof is finished by induction.
\end{proof}

After all the basic facts and notations, we are now ready to link $\mathcal{B}$ to double Bruhat cells $\GL_n^{u,v}$.

\begin{defn} For a pair of Weyl group elements $(u,v)$ we define the \textit{configuration space of Borel subgroups} $\conf^{u,v}(\mathcal{B})$ to be the quotient space of quadruples of Borel subgroups $(B_1,B_2,B_3,B_4)$ satisfying the following \textit{square diagram} relation
\[
\xymatrix{B_1 \ar[r]^u \ar@{-}[d] & B_4 \ar@{-}[d] \\
B_3 \ar[r]_{v^*} & B_2}
\]
modulo the diagonal action by $G$, i.e., $(gB_1g^{-1},gB_2g^{-1},gB_3g^{-1},gB_4g^{-1})\sim (B_1,B_2,B_3,B_4)$.
\end{defn}

\begin{prop}\label{flag-bruhat} $\conf^{u,v}(\mathcal{B})\cong H\backslash \GL_n^{u,v}/H$.
\end{prop}
\begin{proof} By Proposition \ref{opposite flag}, any equivalence class $[B_1,B_2,B_3,B_4]$ in $\conf^{u,v}(\mathcal{B})$ can be represented by 
\[
\xymatrix{B_+ \ar[r]^u \ar@{-}[d] & xB_+ \ar@{-}[d] \\
B_- \ar[r]_{v^*} & xB_-}
\]
for some $x\in \GL_n$, and the choice of $x$ is unique up to a left multiple and a right multiple by elements in $H$. Note that by definition of $\conf^{u,v}(\mathcal{B})$, $x\in B_+uB_+\cap B_-vB_-$. Thus the map $[B_1,B_2,B_3,B_4]\mapsto H\backslash x/H$ is a well-defined map from $\conf^{u,v}(\mathcal{B})$ to $H\backslash \GL_n^{u,v}/H$, and it is not hard to see that it is indeed an isomorphism. 
\end{proof}

Using the isomorphism $\conf^{u,v}(\mathcal{B})\cong H\backslash \GL_n^{u,v}/H$, we can obtain our candidate map $\mathrm{DT}$ on $H\backslash \GL_n^{u,v}/H$ from a natural automorphism $\eta$ on $\conf^{u,v}(\mathcal{B})$, which we will introduce next.

Let $w^c:=w_0w^{-1}$ for any Weyl group element $w$; by computation it is not hard to see that $w_0=w^cw=w^*w^c$ and $l(w_0)=l(w^c)+l(w)=l(w^*)+l(w^c)$. But then Proposition \ref{2.8} tells us that we can find two new Borel subgroups $B_5$ and $B_6$ to put into the middle of the two vertical edges in the above diagram, forming the following hexagon diagram.
\[
\xymatrix{ & B_6\ar[r]^{u^c} \ar@{-}[drr] & B_1 \ar[dr]^u \ar@{-}[dll] & \\
B_3 \ar[ur]^{u^*} \ar[dr]_{v^*} \ar@{-}[drr] & & &  B_4 \ar@{-}[dll] \\
& B_2 \ar[r]_{v^c} & B_5 \ar[ur]_v &}
\]

Note that if we take out the parallelogram with vertices $B_3, B_4, B_5$, and $B_6$, and apply the involution $*$, we get another diagram of a quadruple representing a point in $\conf^{u,v}(\mathcal{B})$.
\[
\xymatrix{B_3^* \ar[r]^u \ar@{-}[d] & B_6^* \ar@{-}[d]\\
B_5^* \ar[r]_{v^*}& B_4^*}
\]
This observation gives rise to a map
\begin{align*}
\eta:\conf^{u,v}(\mathcal{B})&\rightarrow \conf^{u,v}(\mathcal{B})\\
[B_1,B_2,B_3,B_4]&\mapsto [B_3^*, B_4^*,B_5^*,B_6^*].
\end{align*}

\begin{prop}\label{eta} Via the isomorphism $\conf^{u,v}(\mathcal{B})\cong H\backslash \GL_n^{u,v}/H$, the map $\eta$ translates into the map $\mathrm{DT}$ on $H\backslash \GL_n^{u,v}/H$, whose definition is given in our main theorem (Theorem \ref{main}).
\end{prop}
\begin{proof} Recall that $H\backslash x/H$ corresponds to a configuration that can be represented as
\[
\xymatrix{ B_+ \ar@{-}[d] \ar[r]^u & xB_+\ar@{-}[d] \\
B_- \ar[r]_{v^*} & xB_-}
\]
It is not hard to see that
\[
B_5:=x\overline{v^{-1}}B_+ \quad \quad \text{and}\quad \quad B_6:=\overline{u}B_-
\]
will fit into the hexagon diagram
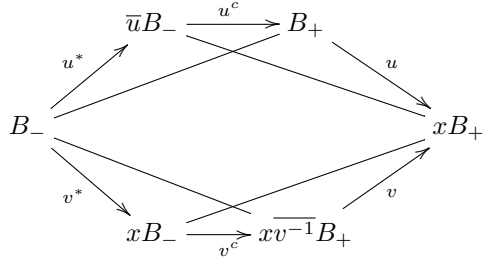
\[
\xymatrix{ & \overline{u}B_- \ar[r]^{u^c} \ar@{-}[drr] & B_+ \ar[dr]^u \ar@{-}[dll] & \\
B_- \ar[ur]^{u^*} \ar[dr]_{v^*} \ar@{-}[drr] & & &  xB_+ \ar@{-}[dll] \\
& xB_- \ar[r]_{v^c} & x\overline{v^{-1}}B_+ \ar[ur]_v &}
\]
Thus by definition the $\eta$ map maps the configuration $[B_+,xB_-,B_-,xB_+]$ to 
\[
\left[\vcenter{\xymatrix{B_- \ar@{-}[d] \ar[r]^{u^*} & \overline{u}B_- \ar@{-}[d] \\
x\overline{v^{-1}}B_+\ar[r]_v & xB_+}}\right]^*
\]
To compute the corresponding image of $\eta$ in $H\backslash \GL_n^{u,v}/H$ we need to rewrite the quadruple of Borel subgroups $\left(B_-,xB_+, x\overline{v^{-1}}B_+,\overline{u}B_-\right)$ as $\left(yB_+,zB_-,yB_-,zB_+\right)$ for some elements $y$ and $z$ in $\GL_n$. Following the guideline we have in the proof of Proposition \ref{opposite flag} we can easily compute
\[
y=\left[x\overline{v^{-1}}\right]_-\overline{w}_0^{-1} \quad \quad \text{and} \quad \quad z=\overline{u}\left[\overline{u}^{-1}x\right]_-\overline{w}_0^{-1}.
\]
Thus the corresponding image of $\eta$ is
\[
H\backslash\left(y^{-1}z\right)^*/H=H\backslash\left(\left[\overline{u}^{-1}x\right]_-^{-1}\overline{u}^{-1}x\overline{v^{-1}}\left[x\overline{v^{-1}}\right]_+^{-1}\right)^t/H,
\]
which is exactly what we had in the definition of our candidate map $\mathrm{DT}$.
\end{proof}

\subsection{Bipartite Graphs}\label{bipartite}

Our next goal is to introduce another set of tools we will use in this paper: bipartite graphs. Unlike last subsection, the bipartite graph technique seems to be applicable only to classical groups of Dynkin type $A_n$ but not other semisimple Lie groups in general.

Let $(u,v)$ be a pair of Weyl group elements of $\GL_n$ and let $\vec{i}=(i(1),\dots, i(l))$ be a reduced word of $(u,v)$. The procedure of producing the associated bipartite graph $\Gamma_\vec{i}$ is the following.
\begin{enumerate}
    \item Draw $n$ horizontal parallel lines with $n-1$ spacings in between.
    \item As we go from left to right along these horizonal parallel lines and go from $i(1)$ to $i(l)$ in the reduced word $\vec{i}$, draw a $
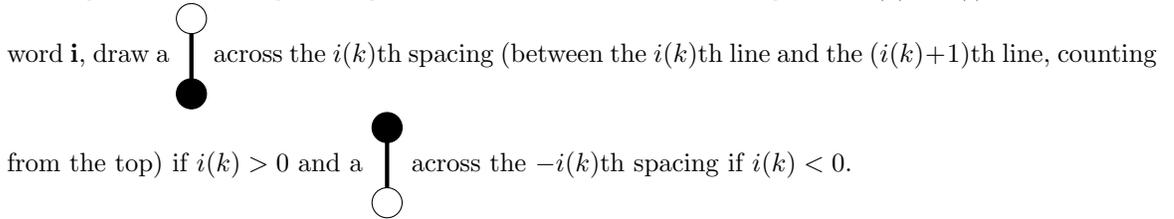
\tikz[baseline=-0.5ex]{\draw[ultra thick] (0,0.5) -- (0,-0.5); \draw[fill=white] (0,0.5) circle [radius=0.2]; \draw[fill=black] (0,-0.5) circle [radius=0.2];}$ across the $i(k)$th spacing (between the $i(k)$th line and the $(i(k)+1)$th line, counting from the top) if $i(k)>0$ and a $\tikz[baseline=-0.5ex]{\draw[ultra thick] (0,0.5) -- (0,-0.5); \draw[fill=black] (0,0.5) circle [radius=0.2]; \draw[fill=white] (0,-0.5) circle [radius=0.2];}$ across the $-i(k)$th spacing if $i(k)<0$.
    \item Put a white vertex between every two neighboring black vertices and draw a black vertex between every two neighboring white vertices.
    \item If either end (or both ends) of a parallel line does not end with a white vertex, then add one to it.
\end{enumerate}

\begin{exmp}\label{2.2} Consider the reduced word $(1,-1)$ of the pair of Weyl group elements $(w_0,w_0)$ for $\GL_2$. The corresponding bipartite graph will look like the following.
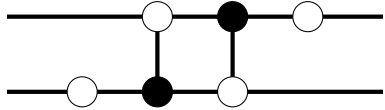
\[
\tikz{
\draw[ultra thick] (0,0) -- (5,0);
\draw[ultra thick] (0,1) -- (5,1);
\draw[ultra thick] (2,0) -- (2,1);
\draw[ultra thick] (3,0) -- (3,1);
\draw[fill=white] (2,1) circle [radius=0.2];
\draw[fill=white] (1,0) circle [radius=0.2];
\draw[fill=white] (3,0) circle [radius=0.2];
\draw[fill=white] (4,1) circle [radius=0.2];
\draw[fill=black] (2,0) circle [radius=0.2];
\draw[fill=black] (3,1) circle [radius=0.2];
}
\]

For a slightly more complicated example, consider the reduced word $(1,-2,2,1,-1,-2)$ of the pair of Weyl group elements $(w_0,w_0)$ for $\GL_3$. The corresponding bicolor graph will look like the following.
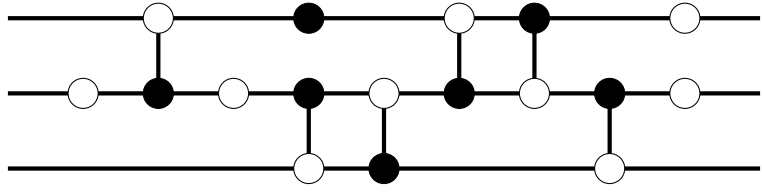
\[
\tikz{
    \draw[ultra thick] (-2,0) -- (8,0);
    \draw[ultra thick] (-2,1) -- (8,1);
    \draw[ultra thick] (-2,2) -- (8,2);
    \draw[ultra thick] (0,1) -- (0,2);
    \draw[ultra thick] (2,1) -- (2,0);
    \draw[ultra thick] (3,1) -- (3,0);
    \draw[ultra thick] (4,1) -- (4,2);
    \draw[ultra thick] (5,1) -- (5,2);
    \draw[ultra thick] (6,1) -- (6,0);
    \draw[fill=white] (0,2) circle [radius=0.2];
    \draw[fill=black] (0,1) circle [radius=0.2];
    \draw[fill=white] (2,0) circle [radius=0.2];
    \draw[fill=black] (2,1) circle [radius=0.2];
    \draw[fill=white] (3,1) circle [radius=0.2];
    \draw[fill=black] (3,0) circle [radius=0.2];
    \draw[fill=white] (4,2) circle [radius=0.2];
    \draw[fill=black] (4,1) circle [radius=0.2];
    \draw[fill=white] (5,1) circle [radius=0.2];
    \draw[fill=black] (5,2) circle [radius=0.2];
    \draw[fill=white] (6,0) circle [radius=0.2];
    \draw[fill=black] (6,1) circle [radius=0.2];
    \draw[fill=white] (-1,1) circle [radius=0.2];
    \draw[fill=white] (1,1) circle [radius=0.2];
    \draw[fill=white] (7,2) circle [radius=0.2];
    \draw[fill=white] (7,1) circle [radius=0.2];
    \draw[fill=black] (2,2) circle [radius=0.2];
}
\]
\end{exmp}

Using bipartite graphs constructed from this procedure Fock and Marshakov \cite{FM} gave a pictorial description of Fock and Goncharov's amalgamation map in \cite{FGD}, and in our case it is the following. Fix a reduced word $\vec{i}$ and its associated bipartite graph $\Gamma$. Associate to each face $f$ of $\Gamma$ a variable $X_f$. Then going from left to right over the bipartite graph $\Gamma$ we multiply the following factors in order:
\begin{enumerate}
    \item $h^i(X_f)$ for every face $f$ lying in the $i$th spacing;
    \item $e_i$ for every vertical edge $\tikz[baseline=-0.5ex]{\draw[ultra thick] (0,0.5) -- (0,-0.5); \draw[fill=white] (0,0.5) circle [radius=0.2]; \draw[fill=black] (0,-0.5) circle [radius=0.2];}$ in the $i$th spacing;
    \item $e_{-i}$ for every vertical edge $\tikz[baseline=-0.5ex]{\draw[ultra thick] (0,0.5) -- (0,-0.5); \draw[fill=black] (0,0.5) circle [radius=0.2]; \draw[fill=white] (0,-0.5) circle [radius=0.2];}$ in the $i$th spacing.
\end{enumerate}
We then treat the resulting expression as a map (the tilde notation will become clear later)
\[
\tilde{\chi}_\vec{i}:(\mathbb{C}^*)^{\text{\# faces}} \rightarrow \GL_n^{u,v}.
\]
The image lands inside $\GL_n^{u,v}$ because $e_{\pm i}\in B_\pm \cap (B_\mp s_iB_\mp)$ and $\prod_{k=1}^l B_\pm s_{i(k)} B_\pm = B_\pm wB\pm$ whenever $(i(1),i(2),\dots, i(l))$ is a reduced word of $w$.

\begin{prop} The map $\tilde{\chi}_\vec{i}$ is well-defined.
\end{prop}
\begin{proof} The only ambiguity occurs when we have a face in the $i$th spacing going across a vertical edge in the $j$th spacing with $i\neq j$; hence the statement reduces to showing the following two identities for $i\neq j$:
\[
h^i(X)e_{\pm j}=e_{\pm j}h^i(X).
\]
These two identities are not hard to verify by direct computation. Proof of similar statements can also be found in \cite{FGD} and \cite{FM}.
\end{proof}

\begin{exmp} To demonstrate, the second bipartite graph of Example \ref{2.2} gives rise to a map $\tilde{\chi}_\vec{i}:(\mathbb{C}^*)^{10}\rightarrow \GL_n^{w_0,w_0}$ as follows.
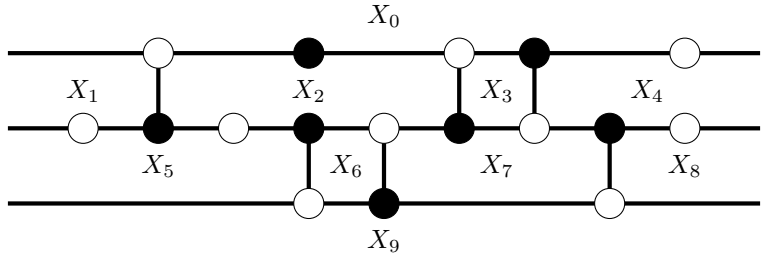
\[
\tikz{
    \draw[ultra thick] (-2,0) -- (8,0);
    \draw[ultra thick] (-2,1) -- (8,1);
    \draw[ultra thick] (-2,2) -- (8,2);
    \draw[ultra thick] (0,1) -- (0,2);
    \draw[ultra thick] (2,1) -- (2,0);
    \draw[ultra thick] (3,1) -- (3,0);
    \draw[ultra thick] (4,1) -- (4,2);
    \draw[ultra thick] (5,1) -- (5,2);
    \draw[ultra thick] (6,1) -- (6,0);
    \draw[fill=white] (0,2) circle [radius=0.2];
    \draw[fill=black] (0,1) circle [radius=0.2];
    \draw[fill=white] (2,0) circle [radius=0.2];
    \draw[fill=black] (2,1) circle [radius=0.2];
    \draw[fill=white] (3,1) circle [radius=0.2];
    \draw[fill=black] (3,0) circle [radius=0.2];
    \draw[fill=white] (4,2) circle [radius=0.2];
    \draw[fill=black] (4,1) circle [radius=0.2];
    \draw[fill=white] (5,1) circle [radius=0.2];
    \draw[fill=black] (5,2) circle [radius=0.2];
    \draw[fill=white] (6,0) circle [radius=0.2];
    \draw[fill=black] (6,1) circle [radius=0.2];
    \draw[fill=white] (-1,1) circle [radius=0.2];
    \draw[fill=white] (1,1) circle [radius=0.2];
    \draw[fill=white] (7,2) circle [radius=0.2];
    \draw[fill=white] (7,1) circle [radius=0.2];
    \draw[fill=black] (2,2) circle [radius=0.2];
    \node at (3,2.5) [] {$X_0$};
    \node at (-1,1.5) [] {$X_1$};
    \node at (2,1.5) [] {$X_2$};
    \node at (4.5,1.5) [] {$X_3$};
    \node at (6.5,1.5) [] {$X_4$};
    \node at (0,0.5) [] {$X_5$};
    \node at (2.5,0.5) [] {$X_6$};
    \node at (4.5,0.5) [] {$X_7$};
    \node at (7,0.5) [] {$X_8$};
    \node at (3,-0.5) [] {$X_9$};
}
\]
\begin{align*}
\tilde{\chi}_\vec{i}(X_0,\dots, X_9)=& h^1(X_1)e_1 h^2(X_5) h^1(X_2) e_{-2}h^2(X_6)e_2h^3(X_9)h^0(X_0)e_1h^1(X_3)h^2(X_7)e_{-1}e_{-2}h^1(X_4)h^2(X_8)\\
=&\begin{pmatrix} X_1(1+X_2+X_2X_3)X_4X_5X_6X_7X_8X_9 & X_1X_5X_6(1+X_7+X_2X_7)X_8X_9 & X_1X_5X_6X_9\\
X_4X_5X_6X_7X_8X_9 & X_5X_6(1+X_7)X_8X_9 & X_5X_6X_9\\
X_4X_6X_7X_8X_9 & (1+X_6+X_6X_7)X_8X_9 & (1+X_6)X_9
\end{pmatrix}.
\end{align*}
Note that the face variable $X_0$ is never involved in the final expression because $h^0(X_0)=1$ for any $X_0\in \mathbb{C}^*$.
\end{exmp}

Now let's turn to the next notion related to bipartite graphs. Bipartite graphs produced from our described procedure will have $n$ external edges on the left and $n$ external edges on the right. Starting from any external edges, one can draw zig-zag strands on such bipartite graphs, following the following rules.
\begin{enumerate}
\item Zig-zag strands leave and enter the free ends of external edges as follows.
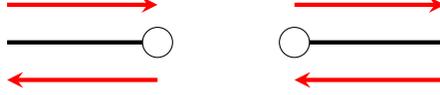
\[
\tikz{
\draw[ultra thick] (0,0) -- (2,0);
\draw[ultra thick, ->, red] (0,0.5) -- (2,0.5);
\draw[ultra thick, ->, red] (2,-0.5) -- (0,-0.5);
\draw[fill=white] (2,0) circle [radius=0.2];
} \quad \quad \quad \quad
\tikz{
\draw[ultra thick] (0,0) -- (2,0);
\draw[ultra thick, ->, red] (0,0.5) -- (2,0.5);
\draw[ultra thick, ->, red] (2,-0.5) -- (0,-0.5);
\draw[fill=white] (0,0) circle [radius=0.2];
}
\]
\item Zig-zag strands travel next to edges.
\item Zig-zag strands turn right at each black vertex and turn left at each white vertex; as a result, zig-zag strands circulate in the counterclockwise direction around a black vertex and circulate in the clockwise direction around a white vertex.
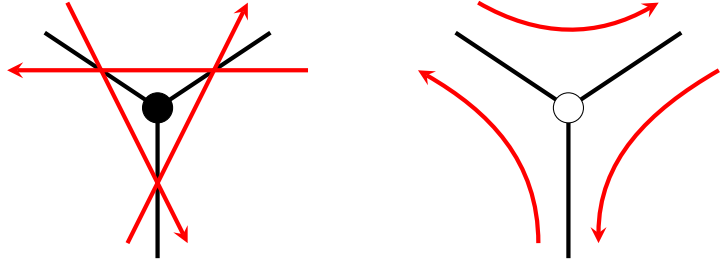
\[
\tikz[baseline=-0.5ex]{
\draw [ultra thick] (-1.5,1) -- (0,0) -- (1.5,1);
\draw [ultra thick] (0,0) -- (0,-2);
\draw [fill=black] (0,0) circle [radius=0.2];
\draw [ultra thick, ->, red] (2,0.5) -- (-2,0.5);
\draw [ultra thick, ->, red] (-1.2,1.4) -- (0.4,-1.8);
\draw [ultra thick, ->, red] (-0.4,-1.8) -- (1.2,1.4);
} \quad \quad  \quad \quad 
\tikz[baseline=-0.5ex]{
\draw [ultra thick] (-1.5,1) -- (0,0) -- (1.5,1);
\draw [ultra thick] (0,0) -- (0,-2);
\draw [fill=white] (0,0) circle [radius=0.2];
\draw [ultra thick, ->, red] (2,0.5) to [out=-150, in=90] (0.4,-1.8);
\draw [ultra thick, ->, red] (-0.4,-1.8) to [out=90, in=-30] (-2,0.5);
\draw [ultra thick, ->, red] (-1.2,1.4) to [out=-30, in=-150] (1.2,1.4);
}
\]
\end{enumerate}

\begin{exmp} To demonstrate, we draw the zig-zag strands on the first bipartite graph of Example \ref{2.2} as follows.
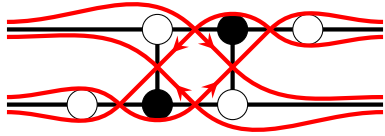
\[
\tikz{
\draw[ultra thick] (0,0) -- (5,0);
\draw[ultra thick] (0,1) -- (5,1);
\draw[ultra thick] (2,0) -- (2,1);
\draw[ultra thick] (3,0) -- (3,1);
\draw[fill=white] (2,1) circle [radius=0.2];
\draw[fill=white] (1,0) circle [radius=0.2];
\draw[fill=white] (3,0) circle [radius=0.2];
\draw[fill=white] (4,1) circle [radius=0.2];
\draw[fill=black] (2,0) circle [radius=0.2];
\draw[fill=black] (3,1) circle [radius=0.2];
\draw [
      decoration={markings, mark=at position 0.55 with {\arrow{stealth}}},
      postaction={decorate}, 
      ultra thick, red
      ] (0,1.1) to [out=0, in=135] (2.5,1) to [out=-45,in=135] (3,0.5) to [out=-45, in=180] (5,0.1);
\draw [
      decoration={markings, mark=at position 0.55 with {\arrow{stealth}}},
      postaction={decorate}, 
      ultra thick, red
      ] (0,0.1) to [out=0, in=135] (1.5,0) to [out=-45,in=-135] (2.5,0) to [out=45, in=-135] (3,0.5) to [out=45, in=-135] (3.5,1) to [out=45, in=180] (5,1.1);
\draw [
      decoration={markings, mark=at position 0.55 with {\arrow{stealth}}},
      postaction={decorate}, 
      ultra thick, red
      ] (5,-0.1) to [out=180, in=-45] (2.5,0) to [out=135,in=-45] (2,0.5) to [out=135, in=0] (0,0.9);
\draw [
      decoration={markings, mark=at position 0.55 with {\arrow{stealth}}},
      postaction={decorate}, 
      ultra thick, red
      ] (5,0.9) to [out=180, in=-45] (3.5,1) to [out=135,in=45] (2.5,1) to [out=-135, in=45] (2,0.5) to [out=-135, in=45] (1.5,0) to [out=-135, in=0] (0,-0.1);
}
\]
\end{exmp}

\begin{rmk} Due to the fact that $\vec{i}$ is a reduced word for the pair $(u,v)$, the bipartite graph we get from this procedure always satisfies the \textit{minimality conditions} of Thurston \cite{Thu}. In particular, the ones that are associated to reduced words of the pair $(w_0,w_0)$ are examples of minimal bipartite graphs defined in \cite{W}.
\end{rmk}

It is not hard to see that among the $2n$ zig-zag strands, half of them are going towards the right and the other half are going towards the left. We call them \textit{right-going} and \textit{left-going} zig-zag strands respectively.

\begin{defn}\label{dominate} A face $f$ is said to be \textit{dominated} by a zig-zag strand $\zeta$ if it lies below $\zeta$. Index the zig-zag strands by the external edge from which they start (counting from the top to the bottom), and collect the indices of the right-going (resp. left-going) zig-zag strands that dominate a face $f$ into the \textit{dominating set} $I(f)$ (resp. $J(f)$). 
\end{defn}

We have associated a variable $X_f$ to each face $f$; now let's associate a function $A_f$ on $\GL_n$ to each face $f$. Observe that for each face $f$, the two dominating sets $I(f)$ and $J(f)$ always have the same size. Thus we can pick out the square submatrix of $x$ formed by rows corresponding to $I(f)$ and columns corresponding to $J(f)$, compute the corresponding minor (the determinant of the submatrix) $\Delta^{I(f),J(f)}(x)$, and define
\[
A_f(x):=\Delta^{I(f),J(f)}(x).
\]
As a convention, we define $\Delta^{\emptyset,\emptyset}(x):=1$, and $\Delta^{I,J}(x):=0$ if $|I|\neq |J|$.

\begin{exmp} To demonstrate, the faces of the second bipartite graph of Example \ref{2.2} are associated to the following list of minors.
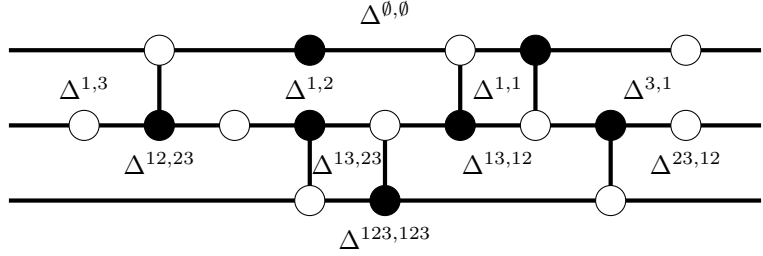
\[
\tikz{
    \draw[ultra thick] (-2,0) -- (8,0);
    \draw[ultra thick] (-2,1) -- (8,1);
    \draw[ultra thick] (-2,2) -- (8,2);
    \draw[ultra thick] (0,1) -- (0,2);
    \draw[ultra thick] (2,1) -- (2,0);
    \draw[ultra thick] (3,1) -- (3,0);
    \draw[ultra thick] (4,1) -- (4,2);
    \draw[ultra thick] (5,1) -- (5,2);
    \draw[ultra thick] (6,1) -- (6,0);
    \draw[fill=white] (0,2) circle [radius=0.2];
    \draw[fill=black] (0,1) circle [radius=0.2];
    \draw[fill=white] (2,0) circle [radius=0.2];
    \draw[fill=black] (2,1) circle [radius=0.2];
    \draw[fill=white] (3,1) circle [radius=0.2];
    \draw[fill=black] (3,0) circle [radius=0.2];
    \draw[fill=white] (4,2) circle [radius=0.2];
    \draw[fill=black] (4,1) circle [radius=0.2];
    \draw[fill=white] (5,1) circle [radius=0.2];
    \draw[fill=black] (5,2) circle [radius=0.2];
    \draw[fill=white] (6,0) circle [radius=0.2];
    \draw[fill=black] (6,1) circle [radius=0.2];
    \draw[fill=white] (-1,1) circle [radius=0.2];
    \draw[fill=white] (1,1) circle [radius=0.2];
    \draw[fill=white] (7,2) circle [radius=0.2];
    \draw[fill=white] (7,1) circle [radius=0.2];
    \draw[fill=black] (2,2) circle [radius=0.2];
    \node at (3,2.5) [] {$\Delta^{\emptyset,\emptyset}$};
    \node at (-1,1.5) [] {$\Delta^{1,3}$};
    \node at (2,1.5) [] {$\Delta^{1,2}$};
    \node at (4.5,1.5) [] {$\Delta^{1,1}$};
    \node at (6.5,1.5) [] {$\Delta^{3,1}$};
    \node at (0,0.5) [] {$\Delta^{12,23}$};
    \node at (2.5,0.5) [] {$\Delta^{13,23}$};
    \node at (4.5,0.5) [] {$\Delta^{13,12}$};
    \node at (7,0.5) [] {$\Delta^{23,12}$};
    \node at (3,-0.5) [] {$\Delta^{123,123}$};
}
\]
\end{exmp}

\begin{rmk} Note that right-going zig-zag strands only tangle at $\tikz[baseline=-0.5ex]{\draw[ultra thick] (0,0.5) -- (0,-0.5); \draw[fill=black] (0,0.5) circle [radius=0.2]; \draw[fill=white] (0,-0.5) circle [radius=0.2];}$ and left-going zig-zag strands only tangle at $\tikz[baseline=-0.5ex]{\draw[ultra thick] (0,0.5) -- (0,-0.5); \draw[fill=white] (0,0.5) circle [radius=0.2]; \draw[fill=black] (0,-0.5) circle [radius=0.2];}$. By comparing the definition of dominating sets of a face $f$ and Equations (2.8) and (2.9) in \cite{BFZ}, we see that the face functions $A_f$ of the bipartite graph associated to the reduced word $\vec{i}$ of $(u,v)$ are precisely the Berenstein, Fomin, and Zelevinsky's cluster variables in the cluster of $\mathbb{C}[\GL_n^{u,v}]$ associated to $\vec{i}$. Further we will show in the next subsection that we fully recovers the upper cluster algebra structure on $\mathbb{C}[\GL_n^{u,v}]$ from these face variables. Thus Theorem 2.10 and Lemma 2.12 from \cite{BFZ} can be reformulated into the following statement.
\end{rmk}

\begin{thm}[Berenstein-Fomin-Zelevinsky]\label{bfz} The map $x\mapsto (A_f(x))$ gives a rational map
\[
\tilde{\psi}_\vec{i}:\GL_n^{u,v}\dashrightarrow (\mathbb{C}^*)^{\text{\# faces}}.
\]
which is a birational equivalence when we shrink the target space to the subtorus $A_0=1$.
\end{thm}

Besides the variables $X_f$ and functions $A_f$, we also associate two quivers, which (by an abuse of notation) will be denoted as $\tilde{\vec{i}}$ and $\vec{i}$, to each bipartite graph $\Gamma_\vec{i}$. The construction of $\tilde{\vec{i}}$ is given as follows.
\begin{enumerate}
\item Put a vertex for each face of $\Gamma_\vec{i}$.
\item Put a counterclockwisely oriented cycle for the faces next to a black vertex as below.
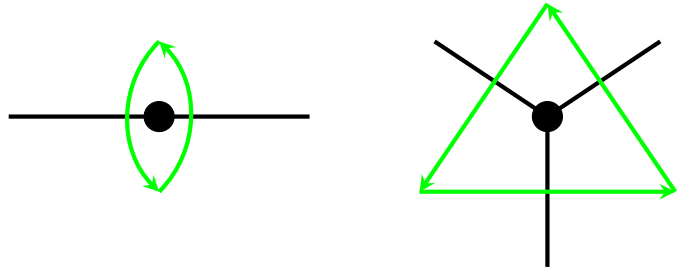
\[
\tikz[baseline=0.5ex]{
\draw[ultra thick] (-1,0) -- (3,0);
\draw[fill=black] (1,0) circle [radius=0.2];
\draw[ultra thick, green, ->] (1,1) to [out=-135, in=135] (1,-1);
\draw[ultra thick, green, ->] (1,-1) to [out=45, in=-45] (1,1);
}\quad \quad \quad \quad 
\tikz[baseline=0.5ex]{
\draw [ultra thick] (-1.5,1) -- (0,0) -- (1.5,1);
\draw [ultra thick] (0,0) -- (0,-2);
\draw [fill=black] (0,0) circle [radius=0.2];
\draw [ultra thick, green, ->] (0,1.5) -- (-1.7,-1);
\draw [ultra thick, green, ->] (-1.7,-1) -- (1.7,-1);
\draw [ultra thick, green, ->] (1.7,-1) -- (0,1.5);
}
\]
\item Remove any 2-cycles in the quiver (hence we don't need to add in counterclockwisely oriented 2-cycles for 2-valent black vertices in the first place).
\end{enumerate}
After constructing $\tilde{\vec{i}}$, the quiver $\vec{i}$ is then obtained from $\tilde{\vec{i}}$ by removing all the vertices corresponding to boundary faces (a total of $2n$ of them) and all the arrows involving those vertices. Because $\vec{i}$ does not have any vertex corresponding to boundary faces, we also call it the \textit{boundary-removed quiver} associated to $\Gamma_\vec{i}$.

\begin{exmp} The two quivers associated to the second picture of Example \ref{2.2} are the following.
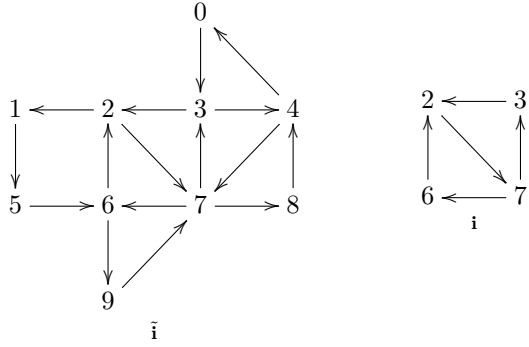
\[
\underset{\tilde{\vec{i}}}{\xymatrix{ & &  0 \ar[d] & \\
1 \ar[d] & 2 \ar[l] \ar[dr] & 3 \ar[l] \ar[r] & 4\ar[dl] \ar[ul] \\
5 \ar[r] & 6 \ar[u] \ar[d] & 7 \ar[l] \ar[u] \ar[r] & 8 \ar[u] \\
&   9 \ar[ur] & & }} \quad \quad \quad \quad 
\underset{\vec{i}}{\xymatrix{ &  \\  2 \ar[dr] & 3 \ar[l] \\ 6 \ar[u]  & 7 \ar[u] \ar[l] }}
\]
\end{exmp}

\begin{rmk} The quivers obtained from our construction is opposite to the quivers constructed in \cite{BFZ}. This is okay since the cluster $\mathcal{A}$-mutation formula is invariant under $\epsilon_{ij}\mapsto -\epsilon_{ij}$, and hence Fomin and Zelevinsky's upper cluster algebra structure on $\mathbb{C}[\GL_n^{u,v}]$ is still applicable to our story.
\end{rmk}

\subsection{Cluster Varieties}\label{cluster} We will give a brief review of Fock and Goncharov's theory of cluster ensemble, with specialization to our main geometric object, the double Bruhat cells in general linear groups. We will mainly follow the coordinate description presented in \cite{FG} with skewsymmetric exchange matrix and no frozen vertices.

\begin{defn} A \textit{seed} $\vec{i}$ is an ordered pair $(I,\epsilon)$ where $I$ is a finite set and $\epsilon$ is a skewsymmetric matrix whose rows and columns are indexed by $I$.
\end{defn}

The data of a seed defined as above is equivalent to the data of a quiver with vertex set $I$ and exchange matrix $\epsilon_{ij}$. We will make no distinctions between seeds and quivers in this paper, and will always use the seed notation to denote any quiver. 

The notion of quiver mutation can be described by a precise formula in the seed language: for an element $k\in I$, the \textit{seed mutation} $\mu_k$ gives a new seed $(I',\epsilon')$ where $I'=I$ and
\[
\epsilon'_{ij}=\left\{\begin{array}{ll}
-\epsilon_{ij} & \text{if $k\in \{i,j\}$;} \\
\epsilon_{ij} & \text{if $\epsilon_{ik}\epsilon_{kj}\leq 0$ and $k\notin \{i,j\}$;}\\
\epsilon_{ij}+|\epsilon_{ik}|\epsilon_{kj} & \text{if $\epsilon_{ik}\epsilon_{kj}>0$, $k\notin \{i,j\}$.}
\end{array}\right.
\]

Starting with an initial seed $\vec{i}_0$, we say that a seed $\vec{i}$ is \textit{mutation equivalent} to $\vec{i}_0$ if there is a sequence of seed mutations that turns $\vec{i}_0$ into $\vec{i}$; we denote the set of all seeds mutation equivalent to $\vec{i}_0$ by $|\vec{i}_0|$. To each seed $\vec{i}$ in $|\vec{i}_0|$ we associate two split algebraic tori $\mathcal{A}_\vec{i}=(\mathbb{C}^*)^{|I|}$ and $\mathcal{X}_\vec{i}=(\mathbb{C}^*)^{|I|}$, which are equipped with canonical coordinates $(A_i)$ and $(X_i)$ indexed by the set $I$ respectively. These two split algebraic tori are linked by a map $p_\vec{i}:\mathcal{A}_\vec{i}\rightarrow \mathcal{X}_\vec{i}$ given by
\[
p_\vec{i}^*(X_i)=\prod_{j\in I} A_j^{\epsilon_{ij}}.
\]
The split algebraic tori $\mathcal{A}_\vec{i}$ and $\mathcal{X}_\vec{i}$ are called a \textit{seed $\mathcal{A}$-torus} and a \textit{seed $\mathcal{X}$-torus} respectively.

A seed mutation $\mu_k:\vec{i}\rightarrow \vec{i}'$ gives rise to birational equivalences between the corresponding seed tori, which by an abuse of notation we also denote both as $\mu_k$; in terms of the canonical coordinates $(A'_i)$ and $(X'_i)$ they can be expressed as
\begin{align*}
\mu_k^*(A'_i)=&\left\{\begin{array}{ll} \displaystyle A_k^{-1}\left(\prod_{\epsilon_{kj}>0} A_j^{\epsilon_{kj}}+\prod_{\epsilon_{kj}<0} A_j^{-\epsilon_{kj}}\right) & \text{if $i=k$,}\\
A_i & \text{if $i\neq k$,}\end{array}\right. \\
\text{and} \quad \quad  \mu_k^*(X'_i)=&\left\{\begin{array}{l l} X_k^{-1} & \text{if $i=k$,} \\
X_i\left(1+X_k^{-\sgn (\epsilon_{ik})}\right)^{-\epsilon_{ik}}& \text{if $i\neq k$.}\end{array}\right.
\end{align*}
These two birational equivalences are called \textit{cluster $\mathcal{A}$-mutation} and \textit{cluster $\mathcal{X}$-mutation} respectively. One important feature about cluster mutations is that they commute with the respective $p$ maps.
\[
\xymatrix{\mathcal{A}_\vec{i} \ar[d]_{p_\vec{i}} \ar@{-->}[r]^{\mu_k} & \mathcal{A}_{\vec{i}'} \ar[d]^{p_{\vec{i}'}}\\
\mathcal{X}_\vec{i} \ar@{-->}[r]_{\mu_k} & \mathcal{X}_{\vec{i}'}} 
\]

Besides cluster mutations between seed tori we also care about cluster isomorphisms induced by seed (quiver) isomorphisms. A \textit{seed isomorphism} $\sigma:\vec{i}\rightarrow \vec{i}'$ is a bijection $\sigma:I\rightarrow I'$ such that $\epsilon'_{\sigma(i)\sigma(j)}=\epsilon_{ij}$. Given a seed isomorphism $\sigma:\vec{i}\rightarrow \vec{i}'$ between two seeds in $|\vec{i}_0|$, we obtain isomorphisms on the corresponding seed tori, which by an abuse of notation we also denote by $\sigma$:
\[
\sigma^*(A_{\sigma(i)})=A_i \quad \quad \text{and} \quad \quad \sigma^*(X_{\sigma(i)})=X_i.
\]
We call these isomorphisms \textit{cluster isomorphisms}. It is not hard to see that cluster isomorphisms also commute with the $p$ maps.
\[
\xymatrix{\mathcal{A}_\vec{i} \ar[d]_{p_\vec{i}} \ar[r]^\sigma & \mathcal{A}_{\vec{i}'} \ar[d]^{p_{\vec{i}'}}\\
\mathcal{X}_\vec{i} \ar[r]_\sigma & \mathcal{X}_{\vec{i}'}} 
\]

Compositions of seed mutations and seed isomorhpisms are called \textit{seed cluster transformations}, and compositions of cluster mutations and cluster isomorphisms are called \textit{cluster transformations}. A seed cluster transformation transformation $\vec{i}\rightarrow \vec{i}$ is called \textit{trivial} if it induces identity maps on the coresponding seed $\mathcal{A}$-torus $\mathcal{A}_\vec{i}$ and seed $\mathcal{X}$-torus $\mathcal{X}_\vec{i}$. 

By gluing the seed tori via cluster mutations we obtain the corresponding \textit{cluster varieties}, which will be denoted as $\mathcal{A}_{|\vec{i}_0|}$ and $\mathcal{X}_{|\vec{i}_0|}$ respectively. Then cluster transformations can be seen as automorphisms on these cluster varieties. Cluster varieties also connect the theory of cluster algebras with the Poisson geometry: on the one hand, the coordinate rings on cluster $\mathcal{A}$-varieties are examples of Fomin and Zelevinsky's upper cluster algebras, and on the other hand, cluster $\mathcal{X}$-varieties carry natural Poisson variety structures given by
\[
\{X_i,X_j\}=\epsilon_{ij}X_iX_j.
\]
Thus a cluster $\mathcal{X}$-variety is also known as a \textit{cluster Poisson variety}. More details are available in \cite{FG}.

Due to the positivity of the mutation formulas, we can tropicalize these cluster varieties, which we will describe in the next subsection.

Since the maps $p_\vec{i}$ commute with cluster mutations, they naturally glue into a map $p:\mathcal{A}_{|\vec{i}_0|}\rightarrow \mathcal{X}_{|\vec{i}_0|}$ of cluster varieties.

For the rest of this subsection we will focus on the double Bruhat cells $\GL_n^{u,v}$ defined by a pair of Weyl group elements $(u,v)$ and its related cluster varieties. Recall that each reduced word $\vec{i}$ of a pair of Weyl group elements $(u,v)$ gives rise to a bipartite graph $\Gamma_\vec{i}$, which in turn gives rise to two quivers (seeds) $\tilde{\vec{i}}$ and $\vec{i}$. Let's focus on the seed $\tilde{\vec{i}}$ first: the canonical coordinates of two seed tori $\mathcal{A}_{\tilde{\vec{i}}}$ and $\mathcal{X}_{\tilde{\vec{i}}}$ are naturally identified with the quantities $A_f$ and $X_f$ we have defined in last subsection respectively; then we naturally obtain the following two rational maps
\[
\tilde{\psi}_\vec{i}:\GL_n^{u,v}\dashrightarrow \mathcal{A}_{\tilde{\vec{i}}} \quad \quad \text{and} \quad \quad \tilde{\chi}_\vec{i}:\mathcal{X}_{\tilde{\vec{i}}}\longrightarrow \GL_n^{u,v}.
\]

To fully utilize the power of the cluster structure, we need to add in cluster mutations. Let's start with the following statement.

\begin{prop}\label{4.1} Let $\vec{i}$ and $\vec{i}'$ be two reduced words of the same pair of Weyl group elements $(u,v)$. Then $\vec{i}'$ can be obtained from $\vec{i}$ by applying a finite sequence of the following moves:
\begin{enumerate}
    \item Exchanging two neighboring entries $i(m)$ and $i(m+1)$ with opposite signs.
    \item Changing the pattern $(\cdots, \pm i, \pm (i+1), \pm i, \cdots)$ to $(\cdots, \pm (i+1), \pm i, \pm (i+1),\cdots)$ or vice versa. 
\end{enumerate}
\end{prop}
\begin{proof} The first move can be use to separate or mix letters of opposite signs in the reduced words of $(u,v)$, and the second move is simply the braid relation among the Coxeter generators.
\end{proof}

Note that between the two moves above, Move (1) does not change the face variables $X_f$ or the face functions $A_f$ unless $i(m)+i(m+1)=0$, and in that exceptional case, we have the following observation.

\begin{prop} Suppose $i(m)=-i(m+1)>0$ in the reduced word $\vec{i}$ and the reduced word $\vec{i}'$ is obtained from $\vec{i}$ by Move (1) applied to switching $i(m)$ and $i(m+1)$. Then in order for the diagrams
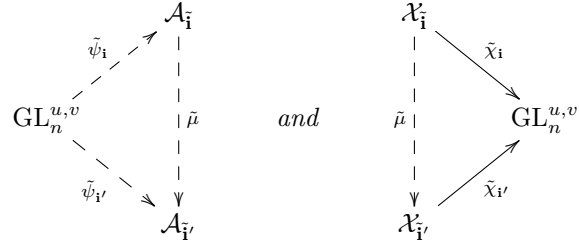
\[
\xymatrix{ & \mathcal{A}_{\tilde{\vec{i}}}\ar@{-->}[dd]^{\tilde{\mu}} & & \mathcal{X}_{\tilde{\vec{i}}}\ar[dr]^{\tilde{\chi}_\vec{i}} \ar@{-->}[dd]_{\tilde{\mu}} &  \\ \GL_n^{u,v} \ar@{-->}[ur]^{\tilde{\psi}_\vec{i}}  \ar@{-->}[dr]_{\tilde{\psi}_{\vec{i}'}} & &\text{and}& & \GL_n^{u,v} \\  & \mathcal{A}_{\tilde{\vec{i}}'}  & & \mathcal{X}_{\tilde{\vec{i}}'} \ar[ur]_{\tilde{\chi}_{\vec{i}'}}}
\]
to commute, we need 
\[
\tilde{\mu}^*(A'_i) = \left\{\begin{array}{l l}
     (A_1A_3+A_2A_4)/A_0 & \text{if $i=0$,} \\
    A_i & \text{if $i\neq 0$,}
\end{array}\right.\quad \quad \text{and} \quad \quad
\tilde{\mu}^*(X'_i)=\left\{\begin{array}{ll}
   1/X_0 & \text{if $i=0$,} \\
    X_i\left(1+X_k^{-\sgn \epsilon_{ik}}\right)^{-\epsilon_{ik}} & \text{if $i\neq 0$,}
\end{array}\right. 
\]
where $\epsilon_{ij}$ is the exchange matrix of the seed $\tilde{\vec{i}}$ and the indices are assigned as follows.
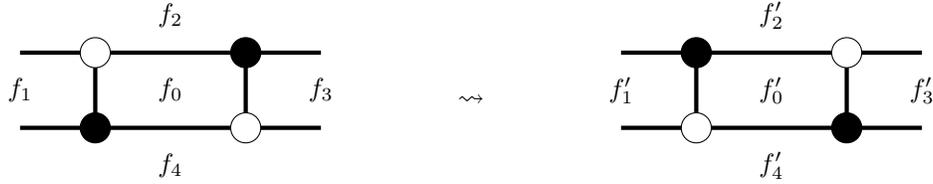
\[
\tikz[baseline=2ex]{
    \draw[ultra thick] (0,0) -- (4,0);
    \draw[ultra thick] (0,1) -- (4,1);
    \draw[ultra thick] (1,0) -- (1,1);
    \draw[ultra thick] (3,0) -- (3,1);
    \draw[fill=black] (1,0) circle [radius=0.2];
    \draw[fill=white] (1,1) circle [radius=0.2];
    \draw[fill=white] (3,0) circle [radius=0.2];
    \draw[fill=black] (3,1) circle [radius=0.2];
     \node at (0,0.5) [] {$f_1$};
    \node at (2,0.5) [] {$f_0$};
    \node at (4,0.5) [] {$f_3$};
    \node at (2,1.5) [] {$f_2$};
    \node at (2,-0.5) [] {$f_4$};
    } \quad \quad \quad \quad \rightsquigarrow \quad \quad \quad \quad
    \tikz[baseline=2ex]{
    \draw[ultra thick] (0,0) -- (4,0);
    \draw[ultra thick] (0,1) -- (4,1);
    \draw[ultra thick] (1,0) -- (1,1);
    \draw[ultra thick] (3,0) -- (3,1);
    \draw[fill=white] (1,0) circle [radius=0.2];
    \draw[fill=black] (1,1) circle [radius=0.2];
    \draw[fill=black] (3,0) circle [radius=0.2];
    \draw[fill=white] (3,1) circle [radius=0.2];
         \node at (0,0.5) [] {$f'_1$};
    \node at (2,0.5) [] {$f'_0$};
    \node at (4,0.5) [] {$f'_3$};
    \node at (2,1.5) [] {$f'_2$};
    \node at (2,-0.5) [] {$f'_4$};
    }
\]
Similar formulas can be derived analogously for the case $i(m)=-i(m+1)<0$. Note that these formulas are exactly the same as the cluster mutation formulas on respective cluster varieties.
\end{prop}
\begin{proof} For the $\mathcal{A}$-mutation formula, just consider the neighboring zig-zag strands as in the following configuration.
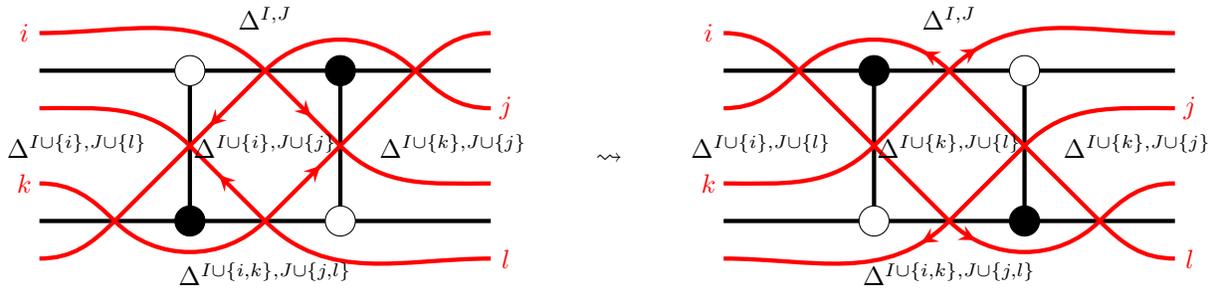
\[
\tikz[baseline=5ex]{
    \draw[ultra thick] (0,0) -- (6,0);
    \draw[ultra thick] (0,2) -- (6,2);
    \draw[ultra thick] (2,0) -- (2,2);
    \draw[ultra thick] (4,0) -- (4,2);
    \draw[fill=black] (2,0) circle [radius=0.2];
    \draw[fill=white] (2,2) circle [radius=0.2];
    \draw[fill=white] (4,0) circle [radius=0.2];
    \draw[fill=black] (4,2) circle [radius=0.2];
    \draw [
      decoration={markings, mark=at position 0.6 with {\arrow{stealth}}},
      postaction={decorate}, 
      ultra thick, red
      ] (0,2.5) to [out=0, in=135] (3,2) to [out=-45,in=135] (4,1) to [out=-45, in=180] (6,0.5);
     \draw [
      decoration={markings, mark=at position 0.6 with {\arrow{stealth}}},
      postaction={decorate}, 
      ultra thick, red
      ] (0,0.5) to [out=0, in=135] (1,0) to [out=-45,in=-135] (3,0) to [out=45, in=-135] (4,1) to [out=45, in=-135] (5,2) to [out=45,in=180] (6,2.5);
     \draw [
      decoration={markings, mark=at position 0.6 with {\arrow{stealth}}},
      postaction={decorate}, 
      ultra thick, red
      ] (6,1.5) to [out=180, in=-45] (5,2) to [out=135,in=45] (3,2) to [out=-135, in=45] (2,1) to [out=-135, in=45] (1,0) to [out=-135,in=0] (0,-0.5);
    \draw [
      decoration={markings, mark=at position 0.6 with {\arrow{stealth}}},
      postaction={decorate}, 
      ultra thick, red
      ] (6,-0.5) to [out=180, in=-45] (3,0) to [out=135,in=-45] (2,1) to [out=135, in=0] (0,1.5);
    \node [red] at (-0.2,2.5) [] {$i$};
    \node [red] at (-0.2,0.5) [] {$k$};
    \node [red] at (6.2,1.5) [] {$j$};
    \node [red] at (6.2,-0.5) [] {$l$};
    \node at (3,2.7) [] {$\Delta^{I,J}$};
    \node at (3,1) [] {$\Delta^{I\cup\{i\},J\cup \{j\}}$};
    \node at (0.5,1) [] {$\Delta^{I\cup\{i\},J\cup\{l\}}$};
    \node at (5.5,1) [] {$\Delta^{I\cup\{k\},J\cup\{j\}}$};
    \node at (3,-0.7) [] {$\Delta^{I\cup\{i,k\},J\cup\{j,l\}}$};
    } \quad \quad  \rightsquigarrow \quad \quad
    \tikz[baseline=5ex]{
        \draw[ultra thick] (0,0) -- (6,0);
    \draw[ultra thick] (0,2) -- (6,2);
    \draw[ultra thick] (2,0) -- (2,2);
    \draw[ultra thick] (4,0) -- (4,2);
    \draw[fill=white] (2,0) circle [radius=0.2];
    \draw[fill=black] (2,2) circle [radius=0.2];
    \draw[fill=black] (4,0) circle [radius=0.2];
    \draw[fill=white] (4,2) circle [radius=0.2];
    \draw [
      decoration={markings, mark=at position 0.6 with {\arrow{stealth}}},
      postaction={decorate}, 
      ultra thick, red
      ] (0,2.5) to [out=0, in=135] (1,2) to [out=-45,in=135] (2,1) to [out=-45, in=135] (3,0) to [out=-45,in=-135] (5,0) to [out=45,in=180] (6,0.5);
     \draw [
      decoration={markings, mark=at position 0.6 with {\arrow{stealth}}},
      postaction={decorate}, 
      ultra thick, red
      ] (0,0.5) to [out=0, in=-135] (2,1) to [out=45,in=-135] (3,2) to [out=45, in=180] (6,2.5);
     \draw [
      decoration={markings, mark=at position 0.6 with {\arrow{stealth}}},
      postaction={decorate}, 
      ultra thick, red
      ] (6,1.5) to [out=180, in=45] (4,1) to [out=-135,in=45] (3,0) to [out=-135, in=0] (0,-0.5);
    \draw [
      decoration={markings, mark=at position 0.6 with {\arrow{stealth}}},
      postaction={decorate}, 
      ultra thick, red
      ] (6,-0.5) to [out=180, in=-45] (5,0) to [out=135,in=-45] (4,1) to [out=135, in=-45] (3,2) to [out=135,in=45] (1,2) to [out=-135,in=0] (0,1.5);
    \node [red] at (-0.2,2.5) [] {$i$};
    \node [red] at (-0.2,0.5) [] {$k$};
    \node [red] at (6.2,1.5) [] {$j$};
    \node [red] at (6.2,-0.5) [] {$l$};
    \node at (3,2.7) [] {$\Delta^{I,J}$};
    \node at (3,1) [] {$\Delta^{I\cup\{k\},J\cup \{l\}}$};
    \node at (0.5,1) [] {$\Delta^{I\cup\{i\},J\cup\{l\}}$};
    \node at (5.5,1) [] {$\Delta^{I\cup\{k\},J\cup\{j\}}$};
    \node at (3,-0.7) [] {$\Delta^{I\cup\{i,k\},J\cup\{j,l\}}$};
    }
\]
Then the $\mathcal{A}$-mutation formula in this case follows from the following identity of matrix minors:
\[
\Delta^{I\cup\{i\},J\cup\{j\}}\Delta^{I\cup\{k\},J\cup\{l\}}=\Delta^{I\cup \{i\},J\cup\{l\}}\Delta^{I\cup\{k\},J\cup\{j\}}+\Delta^{I,J}\Delta^{I\cup\{i,k\},J\cup\{j,l\}}.
\]
A proof of this identity can be found in \cite{FZ} (Theorem 1.17).

The $\mathcal{X}$-mutation formula, on the other hand, reduces to the following identity
\[
e_i h^i(X)e_{-i}=h^i(1+X)e_{-i} h^{i-1}\left(1+X^{-1}\right)^{-1}h^i\left(X^{-1}\right)h^{i+1}\left(1+X^{-1}\right)^{-1}e_i h^i(1+X).
\]
A proof of this identity can be found in \cite{FGD} (Proposition 3.6).
\end{proof}

\begin{prop} Suppose the reduced word $\vec{i}'$ is obtained from $\vec{i}$ by Move (2) applied to changing the pattern from $(\dots, i, i+1, i,\dots)$ to $(\dots, i+1, i, i+1,\dots)$. Then in order for the diagrams 
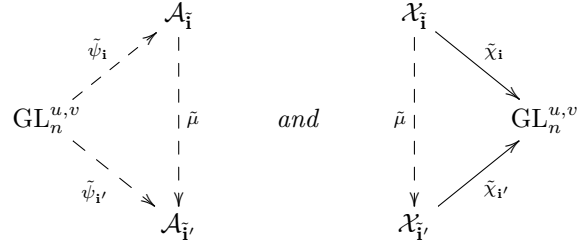
\[
\xymatrix{ & \mathcal{A}_{\tilde{\vec{i}}}\ar@{-->}[dd]^{\tilde{\mu}} & & \mathcal{X}_{\tilde{\vec{i}}}\ar[dr]^{\tilde{\chi}_\vec{i}} \ar@{-->}[dd]_{\tilde{\mu}} &  \\ \GL_n^{u,v} \ar@{-->}[ur]^{\tilde{\psi}_\vec{i}}  \ar@{-->}[dr]_{\tilde{\psi}_{\vec{i}'}} & &\text{and}& & \GL_n^{u,v} \\  & \mathcal{A}_{\tilde{\vec{i}}'}  & & \mathcal{X}_{\tilde{\vec{i}}'} \ar[ur]_{\tilde{\chi}_{\vec{i}'}}}
\]
to commute, we need 
\[
\tilde{\mu}^*(A'_i) = \left\{\begin{array}{l l}
     (A_1A_3+A_2A_4)/A_0 & \text{if $i=0$,} \\
    A_i & \text{if $i\neq 0$,}
\end{array}\right.\quad \quad \text{and} \quad \quad
\tilde{\mu}^*(X'_i)=\left\{\begin{array}{ll}
   1/X_0 & \text{if $i=0$,} \\
    X_i\left(1+X_k^{-\sgn \epsilon_{ik}}\right)^{-\epsilon_{ik}} & \text{if $i\neq 0$,}
\end{array}\right. 
\]
where $\epsilon_{ij}$ is the exchange matrix of the seed $\tilde{\vec{i}}$ and the indices are assigned as follows.
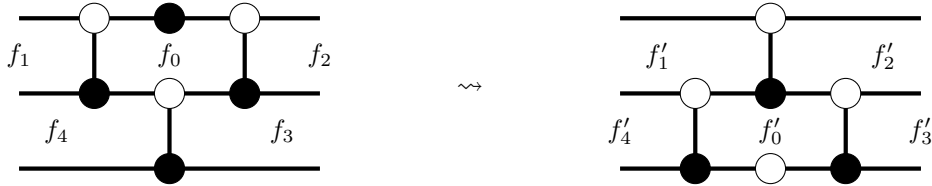
\[
\tikz[baseline=0ex]{
    \draw[ultra thick] (0,0) -- (4,0);
    \draw[ultra thick] (0,1) -- (4,1);
    \draw[ultra thick] (1,0) -- (1,1);
    \draw[ultra thick] (3,0) -- (3,1);
    \draw[ultra thick] (0,-1) -- (4,-1);
    \draw[ultra thick] (2,-1) -- (2,0);
    \draw[fill=black] (1,0) circle [radius=0.2];
    \draw[fill=white] (1,1) circle [radius=0.2];
    \draw[fill=black] (3,0) circle [radius=0.2];
    \draw[fill=white] (3,1) circle [radius=0.2];
    \draw[fill=black] (2,-1) circle [radius=0.2];
    \draw[fill=white] (2,0) circle [radius=0.2];
    \draw[fill=black] (2,1) circle [radius=0.2];
     \node at (0,0.5) [] {$f_1$};
    \node at (2,0.5) [] {$f_0$};
    \node at (4,0.5) [] {$f_2$};
    \node at (0.5,-0.5) [] {$f_4$};
    \node at (3.5,-0.5) [] {$f_3$};
    } \quad \quad \quad \quad \rightsquigarrow \quad \quad \quad \quad
   \tikz[baseline=0ex]{
    \draw[ultra thick] (0,0) -- (4,0);
    \draw[ultra thick] (0,1) -- (4,1);
    \draw[ultra thick] (1,0) -- (1,-1);
    \draw[ultra thick] (3,0) -- (3,-1);
    \draw[ultra thick] (0,-1) -- (4,-1);
    \draw[ultra thick] (2,1) -- (2,0);
    \draw[fill=white] (1,0) circle [radius=0.2];
    \draw[fill=black] (1,-1) circle [radius=0.2];
    \draw[fill=white] (3,0) circle [radius=0.2];
    \draw[fill=black] (3,-1) circle [radius=0.2];
    \draw[fill=white] (2,1) circle [radius=0.2];
    \draw[fill=black] (2,0) circle [radius=0.2];
    \draw[fill=white] (2,-1) circle [radius=0.2];
     \node at (0,-0.5) [] {$f'_4$};
    \node at (2,-0.5) [] {$f'_0$};
    \node at (4,-0.5) [] {$f'_3$};
    \node at (0.5,0.5) [] {$f'_1$};
    \node at (3.5,0.5) [] {$f'_2$};
    }
\]
Similar formulas can be derived analogously for the case of changing from $(\dots, -i,-(i+1),-i,\dots)$ to $(\dots, -(i+1),-i,-(i+1),\dots)$ as well as the other two cases going in the backward direction. Note that these formulas are exactly the same as the cluster mutation formulas on respective cluster varieties.
\end{prop}
\begin{proof} We will leave the zig-zag strands drawing to the readers. The $\mathcal{A}$-mutation formula reduces to the identity (Theorem 1.16 of \cite{FZ}):
\[
\Delta^{I,J\cup\{k\}}\Delta^{I\cup\{i\},J\cup\{j,l\}}=\Delta^{I,J\cup\{j\}}\Delta^{I,J\cup\{l\}}+\Delta^{I\cup\{i\},J\cup\{j,k\}}\Delta^{I\cup \{i\}, J\cup\{k,l\}}.
\]
The $\mathcal{X}$-mutation formula reduces to the identity (Proposition 3.6 of \cite{FGD}):
\[
e_ih^i(X)e_{i+1}e_i=h^{i+1}\left(1+X^{-1}\right)^{-1}h^i(1+X)e_{i+1}e_i h^{i+1}\left(X^{-1}\right)  e_{i+1} h^i\left(1+X^{-1}\right)^{-1}h^{i+1}(1+X). \qedhere
\]
\end{proof}

Combining our last two proposition, we can glue the maps $\tilde{\psi}_\vec{i}$ and $\tilde{\chi}_\vec{i}$ together respectively into two rational maps
\[
\tilde{\psi}:\GL_n^{u,v}\dashrightarrow \mathcal{A}_{|\tilde{\vec{i}}_0|} \quad \quad \text{and} \quad \quad \tilde{\chi}:\mathcal{X}_{|\tilde{\vec{i}}_0|}\longrightarrow \GL_n^{u,v},
\]
where $\vec{i}_0$ is a reduced word of the pair of Weyl group elements $(u,v)$.

Now let's switch our attention to the boundary-removed quiver $\vec{i}$ associated to a bipartite graph $\Gamma_\vec{i}$. First observe that there is a natural projection map $q_\vec{i}:\mathcal{X}_{\tilde{\vec{i}}}\rightarrow \mathcal{X}_\vec{i}$ which deletes the coordinates corresponding to the boundary faces. Combining the $p$ map $\tilde{p}_\vec{i}:\mathcal{A}_{\tilde{\vec{i}}}\rightarrow \mathcal{X}_{\tilde{\vec{i}}}$ and the natural projection $q_\vec{i}:\mathcal{X}_{\tilde{\vec{i}}}\rightarrow \mathcal{X}_\vec{i}$ we obtain a map 
\[
q_\vec{i}\circ \tilde{p}_\vec{i}: \mathcal{A}_{\tilde{\vec{i}}}\rightarrow \mathcal{X}_\vec{i};
\]
it is not hard to see that compositions $q_\vec{i}\circ \tilde{p}_\vec{i}$ commute with respective cluster mutations at non-boundary faces of $\Gamma_\vec{i}$, and hence we can glue them together into a map
\[
q\circ \tilde{p}:\mathcal{A}_{|\tilde{\vec{i}}_0|}\rightarrow \mathcal{X}_{|\vec{i}_0|}.
\]
On the other hand, for a pair of Weyl group elements $(u,v)$, there is a natural projection from the double Bruhat cell $\GL_n^{u,v}$ to its double quotient $H\backslash \GL_n^{u,v}/H$.

\begin{prop} The composition $q\circ \tilde{p}\circ \tilde{\psi}$ is constant along the fibers $\GL_n^{u,v}\rightarrow H\backslash \GL_n^{u,v}/H$. 
\end{prop}
\begin{proof} Since $q\circ \tilde{p}\circ \tilde{\psi}$ is glued from compositions of the form $q_\vec{i}\circ \tilde{p}_\vec{i}\circ \tilde{\psi}_\vec{i}$, it suffices to show that $q_\vec{i}\circ \tilde{p}_\vec{i}\circ \tilde{\psi}_\vec{i}$ is constant along the fibers $\GL_n^{u,v}\rightarrow H\backslash \GL_n^{u,v}/H$ for any reduced word $\vec{i}$ of $(u,v)$. Recall that multiplying a diagonal matrix on the left scales row vectors and multiplying a diagonal matrix on the right scales column vectors. Since the face functions $A_f(x)$ are defined as determinants of submatrices of $x$, we see that pictorially the row vectors of $x$ are represented by right-going zig-zag strands and the column vectors of $x$ are represented by left-going zig-zag strands. Thus we only need to verify that, under the map $q_\vec{i}\circ \tilde{p}_\vec{i}\circ \tilde{\psi}_\vec{i}$, the contribution of each row vector to the numerator of $X_f$ for any non-boundary face $f$ balances out that to the denominator of $X_f$. But this follows from the graph theoretical fact that given any zig-zag strand $\zeta$ and a non-boundary face $f$, the number of neighboring faces of $f$ dominated by $\zeta$ and pointing towards $f$ is equal to the number of neighboring faces of $f$ dominated by $\zeta$ and pointed by $f$.
\end{proof}

A direct consequence of this proposition is that we obtain a rational map
\[
\psi:H\backslash \GL_n^{u,v}/H\dashrightarrow \mathcal{X}_{|\vec{i}_0|}
\]
that fits in the commutative diagram
\[
\xymatrix{\GL_n^{u,v}\ar@{-->}[r]^{\tilde{p}\circ\tilde{\psi}} \ar[d] & \mathcal{X}_{|\tilde{\vec{i}}_0|} \ar[d]^q \\
H\backslash \GL_n^{u,v}/H \ar@{-->}[r]_(0.6)\psi & \mathcal{X}_{|\vec{i}_0|}}
\]

We can put in the map $\tilde{\chi}:\mathcal{X}_{|\tilde{\vec{i}}_0|}\rightarrow \GL_n^{u,v}$ into the picture as well, which can obviously be passed to the double quotient. We denote it as 
\[
\chi:\mathcal{X}_{|\vec{i}_0|}\rightarrow H\backslash \GL_n^{u,v}/H.
\]
As a result, we have the following chain of commutative diagrams
\begin{equation}\label{chain}
\xymatrix{ \cdots \ar@{-->}[r] & \mathcal{X}_{|\tilde{\vec{i}}_0|} \ar[r]^{\tilde{\chi}} \ar[d]^q & \GL_n^{u,v} \ar@{-->}[r]^{\tilde{p}\circ \tilde{\psi}}  \ar[d] & \mathcal{X}_{|\tilde{\vec{i}}_0|} \ar[d]^q \ar[r] & \cdots \\ \cdots \ar@{-->}[r] &
\mathcal{X}_{|\vec{i}_0|} \ar[r]_(0.4)\chi & H\backslash \GL_n^{u,v}/H \ar@{-->}[r]_(0.6)\psi & \mathcal{X}_{|\vec{i}_0|} \ar[r] & \cdots }
\end{equation}
We then can reformulate our main result (Theorem \ref{main}) into the following more precise proposition.

\begin{prop}\label{main'} Each map in the bottom row of the above chain of commutative diagram is a birational equivalence. The composition $\psi\circ \chi$ is a Donaldson-Thomas transformation of the cluster variety $\mathcal{X}_{|\vec{i}_0|}$ and the composition $\mathrm{DT}:=\chi\circ \psi$ takes the form
\[
H\backslash x/H \mapsto H\backslash\left(\left[\overline{u}^{-1}x\right]_-^{-1}\overline{u}^{-1} x \overline{v^{-1}}\left[x\overline{v^{-1}}\right]^{-1}_+\right)^t/H
\]
\end{prop}

\subsection{Tropicalization, Lamination, and Cluster Donaldson-Thomas Transformation}\label{1.4} In this subsection we will discuss tropicalization of cluster varieties and the laminations on them. Near the end we will use laminations to give a definition of Goncharov and Shen's cluster Donaldson-Thomas transformation.

Let's start with a split algebraic torus $\mathcal{X}$. The semiring of \textit{positive rational functions} on $\mathcal{X}$, which we denote as $P(\mathcal{X})$, is the semiring consisting of elements in the form $f/g$ where $f$ and $g$ are linear combinations of characters on $\mathcal{X}$ with positive integral coefficients. A rational map $\phi:\mathcal{X}\dashrightarrow \mathcal{Y}$ between two split algebraic tori is said to be \textit{positive} if it induces a semiring homomorphism $\phi^*:P(\mathcal{Y})\rightarrow P(\mathcal{X})$. It then follows that composition of positive rational maps is still a positive rational map.

One typical example of a positive rational map is a cocharacter $\chi$ of a split algebraic torus $\mathcal{X}$: the induced map $\chi^*$ pulls back an element $f/g\in P(\mathcal{X})$ to $\frac{\langle f, \chi\rangle}{\langle g,\chi\rangle}$ in $P(\mathbb{C}^*)$, where $\langle f, \chi\rangle$ and $\langle g,\chi\rangle$ are understood as linear extensions of the canonical pairing between characters and cocharacters with values in powers of $z$. We will denote the lattice of cocharacters of a split algebraic torus $\mathcal{X}$ by $\mathcal{X}^t$ for reasons that will become clear in a moment.

Note that $P(\mathbb{C}^*)$ is the semiring of rational functions in a single variable $z$ with positive integral coefficients. Thus if we let $\mathbb{Z}^t$ be the semiring $(\mathbb{Z}, \max, +)$, then there is a semiring homomorphism $\deg_z:P(\mathbb{C}^*)\rightarrow \mathbb{Z}^t$ defined by $f(z)/g(z)\mapsto \deg_zf-\deg_zg$. Therefore a cocharacter $\chi$ on $\mathcal{X}$ gives rise to a natural semiring homomorphism 
\[
\deg_z \langle \cdot, \chi\rangle:P(\mathcal{X})\rightarrow \mathbb{Z}^t
\]

\begin{prop} The map $\chi\mapsto \deg_z\langle \cdot, \chi\rangle$ is a bijection between the lattice of cocharacters and set of semiring homomorphisms from $P(\mathcal{X})$ to $\mathbb{Z}^t$.
\end{prop}
\begin{proof} Note that $P(\mathcal{X})$ is a free commutative semiring generated by any basis of the lattice of characters, and in paricular any choice of coordinates $(X_i)_{i=1}^r$. Therefore to define a semiring homomorphism from $P(\mathcal{X})$ to $\mathbb{Z}^t$ we just need to assign to each $X_i$ some integer $a_i$. But for any such $r$-tuple $(a_i)$ there exists a unique cocharacter $\chi$ such that $\langle X_i,\chi\rangle=z^{a_i}$. Therefore $\chi\mapsto \deg_z\langle \cdot, \chi\rangle$ is indeed a bijection.
\end{proof}

\begin{cor} A positive rational map $\phi:\mathcal{X}\dashrightarrow \mathcal{Y}$ between split algebraic tori gives rise to a natural map $\phi^t:\mathcal{X}^t\rightarrow \mathcal{Y}^t$ between the respective lattice of cocharacters.
\end{cor}
\begin{proof} Note that $\phi$ induces a semiring homomorphism $\phi^*:P(\mathcal{Y})\rightarrow P(\mathcal{X})$. Therefore for any cochcaracter $\chi$ of $\mathcal{X}$, the map $f\mapsto \deg_z\langle \phi^*f,\chi\rangle$ is a semiring homomorphism from $P(\mathcal{Y})\rightarrow \mathbb{Z}^t$. By the above proposition there is a unique cocharacter $\eta$ of $\mathcal{Y}$ representing this semiring homomorphism, and we assign $\phi^t(\chi)=\eta$.
\end{proof}

We also want to give an explicit way to compute the induced map $\phi^t$. Fix two coordinate charts $(X_i)$ on $\mathcal{X}$ and $(Y_j)$ on $\mathcal{Y}$. Then $(X_i)$ gives rise to a basis $\{e_i\}$ of the lattice of cocharacters $\mathcal{X}^t$, which is defined by
\[
\chi_i^*(X_k):=\left\{\begin{array}{ll}
z & \text{if $k=i$;} \\
1 & \text{if $k\neq i$.}
\end{array}\right.
\]
This basis allows us to write each cocharacter $\chi$ of $\mathcal{X}$ as a linear combination $\sum x_i\chi_i$. It is not hard to see that
\[
x_i=\deg_z\langle X_i, \chi\rangle.
\]
Similarly the coordinate chart $(Y_j)$ also gives rise to a basis $\{\eta_j\}$ of the lattice of cocharacters $\mathcal{Y}^t$, and we can write each cocharacter of $\mathcal{Y}$ as a linear combination $\sum y_j\eta_j$. 
On the other hand, for any positive rational function $q$ in $r$ varibles $X_1, \dots, X_r$ we have the so-called \textit{na\"{i}ve tropicalization}, which turns $q$ into a map from $\mathbb{Z}^r$ to $\mathbb{Z}$ via the following process:
\begin{enumerate}
\item replace addition in $q(X_1,\dots, X_r)$ by taking maximum;
\item replace multiplication in $q(X_1,\dots, X_r)$ by addition and replace division in $q(X_1,\dots, X_r)$ by subtraction;
\item replace every constant term by zero;
\item replace $X_i$ by $x_i$.
\end{enumerate}
It is not hard to see that, given a positive rational map $\phi:\mathcal{X}\dashrightarrow\mathcal{Y}$, the induced map $\phi^t$ maps $\sum x_i\chi_i$ to $\sum y_j\eta_j$ where
\begin{equation}\label{cocharacter}
y_j:=(\phi^*(Y_j))^t(x_i).
\end{equation}

Now we are ready to define tropicalization. 

\begin{defn} The \textit{tropicalization} of a split algebraic torus $\mathcal{X}$ is defined to be its lattice of cocharacters $\mathcal{X}^t$ (and hence the notation). For a positive rational map $\phi:\mathcal{X}\dashrightarrow \mathcal{Y}$ between split algebraic tori, the \textit{tropicalization} of $\phi$ is defined to be the map $\phi^t:\mathcal{X}^t\rightarrow \mathcal{Y}^t$. The basis $\{\chi_i\}$ of $\mathcal{X}^t$ corresponding to a coordinate system $(X_i)$ on $\mathcal{X}$ is called the \textit{basic laminations} associated to $(X_i)$.
\end{defn}

Now let's go back to the cluster varieties $\mathcal{A}_{|\vec{i}_0|}$ and $\mathcal{X}_{|\vec{i}_0|}$. Since both cluster varieties are obtained by gluing seed tori via positive birational equivalences, we can tropicalize everything and obtain two new glued objects which we call \textit{tropicalized cluster varieties} and denote as $\mathcal{A}_{|\vec{i}_0|}^t$ and $\mathcal{X}_{|\vec{i}_0|}^t$.

Since each seed $\mathcal{X}$-torus $\mathcal{X}_\vec{i}$ is given a split algebraic torus, it has a set of basic laminations associated to the canonical coordinates $(X_i)$; we will call this set of basic laminations the \textit{positive basic $\mathcal{X}$-laminations} and denote them as $l_i^+$. Note that $\{-l_i^+\}$ is also a set of basic laminations on $\mathcal{X}_\vec{i}$, which will be called the \textit{negative basic $\mathcal{X}$-laminations} and denote them as $l_i^-$.

With all the terminologies developed, we can now state the definition of Goncharov and Shen's cluster Donaldson-Thomas transformation as follows.

\begin{defn}[Definition 2.15 in \cite{GS}] A \textit{cluster Donaldson-Thomas transformation} (of a seed $\mathcal{X}$-torus $\mathcal{X}_\vec{i}$) is a cluster transformation $\mathrm{DT}:\mathcal{X}_\vec{i}\dashrightarrow \mathcal{X}_\vec{i}$ whose tropicalization $\mathrm{DT}^t:\mathcal{X}_\vec{i}^t\rightarrow \mathcal{X}_\vec{i}^t$ maps each positive basic $\mathcal{X}$-laminations $l_i^+$ to its corresponding negative basic $\mathcal{X}$-laminations $l_i^-$. \end{defn}

Goncharov and Shen proved that a cluster Donaldson-Thomas transformation enjoys the following properties.

\begin{thm} [Goncharov-Shen, Theorem 2.16 in \cite{GS}] \label{gs} A cluster Donaldson-Thomas transformation $\mathrm{DT}:\mathcal{X}_\vec{i}\rightarrow \mathcal{X}_\vec{i}$ is unique if it exists. If $\vec{i}'$ is another seed in $|\vec{i}|$ (the collection of seeds mutation equivalent to $\vec{i}$) and $\tau:\mathcal{X}\vec{i}\rightarrow \mathcal{X}_{\vec{i}'}$ is a cluster transformation, then the conjugate $\tau \mathrm{DT} \tau^{-1}$ is the cluster Donaldson-Thomas transformation of $\mathcal{X}_{\vec{i}'}$. Therefore it makes sense to say that the cluster Donaldson-Thomas transformation $\mathrm{DT}$ exists on a cluster $\mathcal{X}$-variety without referring to any one specific seed $\mathcal{X}$-torus.
\end{thm}

From our discussion on tropicalization above, we can translate the definition of a cluster Donaldson-Thomas transformation into the following equivalent one, which we will use to prove our main theorem.

\begin{prop}\label{lem0} A cluster transformation $\mathrm{DT}:\mathcal{X}_{|\vec{i}_0|}\rightarrow \mathcal{X}_{|\vec{i}_0|}$ is a cluster Donaldson-Thomas transformation if and only if on one (any hence any) seed $\mathcal{X}$-torus $\mathcal{X}_\vec{i}$ with cluster coordinates $(X_i)$, we have
\[
\deg_{X_i}\mathrm{DT}^*(X_j)=-\delta_{ij}
\]
where $\delta_{ij}$ denotes the Kronecker delta.
\end{prop}
\begin{proof} From Equation \eqref{cocharacter} we see that $\mathrm{DT}^t(l_i^+)= l_i^-$ for all $i$ if and only if $\deg_{X_i}\mathrm{DT}^*(X_j)=-\delta_{ij}$.
\end{proof}

\section{Proof of our Main Result}

In this section we will prove Proposition \ref{main'} which in turn proves our main result (Theorem \ref{main}). 

\subsection{A Formula for \texorpdfstring{$\chi\circ \psi$}{}}\label{formula} Let's start with the last part of Proposition \ref{main'}. By using the chain of commutative diagrams in \eqref{chain}, we can lift $\chi\circ \psi$ to the composition of maps $\tilde{\chi}\circ \tilde{p}\circ \tilde{\psi}$. Since $\tilde{\chi}\circ \tilde{p}\circ\tilde{\psi}$ is obtained by gluing $\tilde{\chi}_\vec{i}\circ \tilde{p}_\vec{i}\circ\tilde{\psi}_\vec{i}$, it suffices to show that for any reduced word $\vec{i}$ of the pair of Weyl group elements $(u,v)$, the rational map $\tilde{\chi}_\vec{i}\circ \tilde{p}_\vec{i}\circ\tilde{\psi}_\vec{i}$ can be expressed as
\[
x\mapsto D \left(\left[\overline{u}^{-1}x\right]_-^{-1}\overline{u}^{-1} x \overline{v^{-1}}\left[x\overline{v^{-1}}\right]^{-1}_+\right)^t D',
\]
where $D$ and $D'$ are two diagonal-matrices which may depend on $x$.

Fix a reduced word $\vec{i}$ of $(u,v)$. Let $\Gamma_\vec{i}$ be the bipartite graph corresponding to $\vec{i}$. Analogous to Definition \ref{dominate} we can define the notion of a right-going or left-going zig-zag strand dominating a white vertex $w$; in particular we can define dominating sets $I(w)$ (the set of indices of right-going zig-zag strands dominating $w$) and $J(w)$ (the set of indices of left-going zig-zag strands dominating $w$). It is not hard to see that for any white vertex $w$ of $\Gamma_\vec{i}$,
\[
|I(w)|=|J(w)|+1.
\]

Consider an $n$-dimensional complex vector space $V$ and its dual space $V^*$. There is a natural $\GL_n$ simply transitive right action on the space of bases of $V$ and $V^*$. Fix a basis $\{a_i\}$ of $V$ and let $\{\alpha_i\}$ be the corresponding dual basis of $V^*$. Let $x$ be an element in $\GL_n$ and let $\{b_i\}:=\{a_i\}.x$, i.e.,
\begin{equation}\label{b}
b_i:=\sum_{k=1}^na_kx_{ki}.
\end{equation}
Now for each white vertex $w$ in $\Gamma_\vec{i}$ we can define an element $\xi_w$ in $V^*$ by $\xi_w:=i_{\left(\overset{\rightarrow}{\bigwedge}_{j\in J(w)} b_j\right)} \left(\overset{\rightarrow}{\bigwedge}_{i\in I(w)} \alpha_i\right)$ where the notation $i$ denotes contraction and the notation $\overset{\rightarrow}{\bigwedge}$ means that we are taking the wedge product in the order of ascending indices; in other words, for any $v\in V$:
\[
\inprod{\xi_w}{v}:=\inprod{ \bigwedge^\rightarrow_{i\in I(w)} \alpha_i}{\left(\bigwedge^\rightarrow_{j\in J(w)} b_j\right)\wedge v}.
\]

\begin{prop}\label{Af} If either of the following three pictures is part of $\Gamma_\vec{i}$, then we have $A_f(x)=\pm \inprod{\xi_w}{b_l}$.
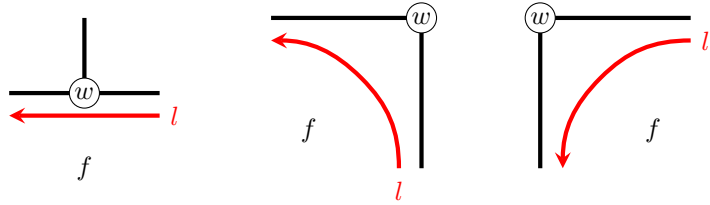
\[
\tikz[baseline=0.5ex]{
\draw[ultra thick] (0,0) -- (2,0);
\draw[ultra thick] (1,0) -- (1,1);
\draw[fill=white] (1,0) circle [radius=0.2];
\draw[ultra thick, red, ->] (2,-0.3) -- (0,-0.3);
\node [red] at (2.2,-0.3) [] {$l$};
\node at (1,0) [] {$w$};
\node at (1,-1) [] {$f$};
}\quad \quad \quad 
\tikz[baseline=0.5ex]{
\draw[ultra thick] (0,1) -- (2,1);
\draw[ultra thick] (2,1) -- (2,-1);
\draw[fill=white] (2,1) circle [radius=0.2];
\draw[ultra thick, red, ->] (1.7,-1) to [out=90, in=-45] (1.2,0.2) to [out=135,in=0] (0,0.7);
\node [red] at (1.7,-1.3) [] {$l$};
\node at (2,1) [] {$w$};
\node at (0.5,-0.5) [] {$f$};
}\quad \quad \quad
\tikz[baseline=0.5ex]{
\draw[ultra thick] (0,1) -- (2,1);
\draw[ultra thick] (0,1) -- (0,-1);
\draw[fill=white] (0,1) circle [radius=0.2];
\draw[ultra thick, red, ->] (2,0.7) to [out=180, in=45] (0.8,0.2) to [out=-135,in=90] (0.3,-1);
\node [red] at (2.2,0.7) [] {$l$};
\node at (0,1) [] {$w$};
\node at (1.5,-0.5) [] {$f$};
}
\]
As a corollary, for a generic element $x\in \GL_n^{u,v}$, $\xi_w\neq 0$.
\end{prop}
\begin{proof} Notice that up to a sign,
\begin{align*}
    \inprod{\xi_w}{b_l}=& \inprod{\bigwedge^\rightarrow_{i\in I(f)} \alpha_i}{\bigwedge^\rightarrow_{j\in J(f)} b_j}\\
    =& \inprod{\bigwedge^\rightarrow_{i\in I(f)} \alpha_i}{\bigwedge^\rightarrow_{j\in J(w)} \left(\sum_{k=1}^n a_kx_{kj}\right)}\\
    =&\Delta^{I(f),J(f)}(x)\\
    =&A_f(x).
\end{align*}
From Theorem \ref{bfz} we know that $A_f(x)\neq 0$ for a generic element $x\in \GL_n^{u,v}$, which implies that $\xi_w\neq 0$.
\end{proof}

Recall from our construction of the bipartite graph $\Gamma_\vec{i}$, black vertices are of two possible valencies: either 2-valent or 3-valent. In particular, any 2-valent black vertex connects two white vertices in the horizontal direction as shown below. 
\[
\tikz{
\draw[ultra thick] (0,0) -- (4,0);
\draw[fill=black] (2,0) circle [radius=0.2];
\draw[fill=white] (0,0) circle [radius=0.2];
\draw[fill=white] (4,0) circle [radius=0.2];
\node at (0,0) [] {$s$};
\node at (4,0) [] {$r$};
\node at (0,0.5) [] {};
\node at (0,-0.5) [] {};
}
\]
It is not hard to verify that in this case we have $\xi_r=\xi_s$.

On the other hand, 3-valent black vertices come in two types: they either form a ``T''-shape or an up-side-down ``T''-shape. The dual vectors associated to neighboring white vertices of a 3-valent black vertex then satisfy the following relation.

\begin{prop}\label{3-valent} At a 3-valent black vertex depicted as either picture below, the following identity holds for a generic element $x\in \GL_n^{u,v}$:
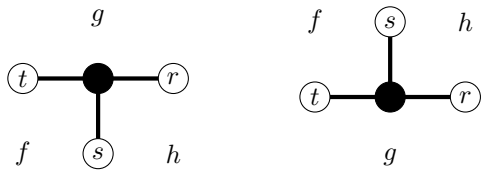
\[
A_f\xi_r=A_g\xi_s+A_h\xi_t.
\]
\[
\tikz{
\draw[ultra thick] (-1,0.5) -- (1,0.5);
\draw[ultra thick] (0,0.5) -- (0,-0.5);
\draw[fill=white] (-1,0.5) circle [radius=0.2];
\draw[fill=white] (1,0.5) circle [radius=0.2];
\draw[fill=white] (0,-0.5) circle [radius=0.2];
\draw[fill=black] (0,0.5) circle [radius=0.2];
\node at (1,0.5) [] {$r$};
\node at (-1,0.5) [] {$t$};
\node at (0,-0.5) [] {$s$};
\node at (0,1.3) [] {$g$};
\node at (-1,-0.5) [] {$f$};
\node at (1,-0.5) [] {$h$};
}
\quad \quad \quad \quad
\tikz{
\draw[ultra thick] (-1,-0.5) -- (1,-0.5);
\draw[ultra thick] (0,-0.5) -- (0,0.5);
\draw[fill=white] (-1,-0.5) circle [radius=0.2];
\draw[fill=white] (1,-0.5) circle [radius=0.2];
\draw[fill=white] (0,0.5) circle [radius=0.2];
\draw[fill=black] (0,-0.5) circle [radius=0.2];
\node at (1,-0.5) [] {$r$};
\node at (-1,-0.5) [] {$t$};
\node at (0,0.5) [] {$s$};
\node at (0,-1.3) [] {$g$};
\node at (-1,0.5) [] {$f$};
\node at (1,0.5) [] {$h$};
}
\]
\end{prop}
\begin{proof} Let's consider the picture on the left first. Let $\Phi:=\overset{\rightarrow}{\bigwedge}_{i\in I(g)}\alpha_i$ and $B:=\overset{\rightarrow}{\bigwedge}_{j\in J(g)} b_j$. Draw the three zig-zag strands involved in the picture as follows and label them by the (dual) vector they represent.
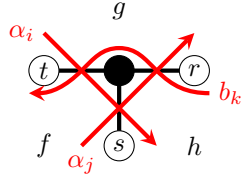
\[
\tikz{
\draw[ultra thick] (-1,0.5) -- (1,0.5);
\draw[ultra thick] (0,0.5) -- (0,-0.5);
\draw[fill=white] (-1,0.5) circle [radius=0.2];
\draw[fill=white] (1,0.5) circle [radius=0.2];
\draw[fill=white] (0,-0.5) circle [radius=0.2];
\draw[fill=black] (0,0.5) circle [radius=0.2];
\draw[ultra thick, red, <-] (-1.2,0.2) to [out=0,in=-135] (-0.5,0.5) to [out=45, in=180] (0,0.8) to [out=0,in=135] (0.5,0.5) to [out=-45,in=180] (1.2,0.2);
\draw[ultra thick, red, ->] (-0.5,-0.5) -- (1,1);
\draw[ultra thick, red, ->] (-1,1) -- (0.5,-0.5);
\node[red] at (-1,1) [left] {$\alpha_i$};
\node[red] at (-0.5,-0.5) [below] {$\alpha_j$};
\node[red] at (1.2,0.2) [right] {$b_k$};
\node at (1,0.5) [] {$r$};
\node at (-1,0.5) [] {$t$};
\node at (0,-0.5) [] {$s$};
\node at (0,1.3) [] {$g$};
\node at (-1,-0.5) [] {$f$};
\node at (1,-0.5) [] {$h$};
}
\]
Without loss of generality we may assume that the indices of all constituents of $\Phi$ are smaller than $i$, and the indices of all constituents of $B$ are smaller than $k$; due to the fact that $\vec{i}$ is a reduced word, we also know that $i<j$. Then from the definition of the dual vectors $\xi_w$ and Proposition \ref{Af}, we find that
\[
A_f=\inprod{\Phi\wedge \alpha_i}{B\wedge b_k}, \quad \quad A_g=\inprod{\Phi}{B}, \quad \quad A_h=\inprod{\Phi\wedge \alpha_j}{B\wedge b_k},
\]
and the pairing of $\xi_r$, $\xi_s$, and $\xi_t$ with any $v\in V$ are given by
\[
\inprod{\xi_r}{v}=\inprod{\Phi\wedge \alpha_j}{B\wedge v}, \quad \quad \inprod{\xi_s}{v}=\inprod{\Phi\wedge \alpha_i\wedge \alpha_j}{B\wedge b_k\wedge v}, \quad \quad \inprod{\xi_t}{v}=\inprod{\Phi\wedge \alpha_i}{B\wedge v}.
\]
Now consider the dual vector
\[
\eta:=A_g\xi_s+A_h\xi_t-A_f\xi_r.
\]
Obviously $\eta$ lies in the span of $\alpha_i$, $\alpha_j$, and the constituents of $\Phi$. Under the generic condition that both $A_f$ and $A_h$ are non-zero, it follows that $\eta=0$ if and only if $\eta$ vanishes on $b_k$ and all constituents of $B$. The latter is obviously satisfied due to the construction of $\eta$. So we only need to show that 
\[
\inprod{\eta}{b_k}=A_h\inprod{\xi_t}{b_k}-A_f\inprod{\xi_r}{b_k}=\inprod{\Phi\wedge \alpha_j}{B\wedge b_k}\inprod{\Phi\wedge \alpha_i}{B\wedge b_k}-\inprod{\Phi\wedge \alpha_i}{B\wedge b_k}\inprod{\Phi\wedge \alpha_j}{B\wedge b_k}
\]
vanishes, which is obviously true.

Next let's consider the picture on the right. Define $\Phi:=\overset{\rightarrow}{\bigwedge}_{i\in I(t)} \alpha_i$ and $B:=\overset{\rightarrow}{\bigwedge}_{j\in J(t)} b_j$ and draw the three zig-zag strands as before.
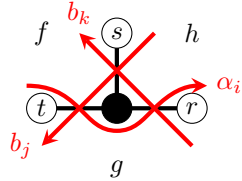
\[
\tikz{
\draw[ultra thick] (-1,-0.5) -- (1,-0.5);
\draw[ultra thick] (0,-0.5) -- (0,0.5);
\draw[fill=white] (-1,-0.5) circle [radius=0.2];
\draw[fill=white] (1,-0.5) circle [radius=0.2];
\draw[fill=white] (0,0.5) circle [radius=0.2];
\draw[fill=black] (0,-0.5) circle [radius=0.2];
\draw[ultra thick, red, ->] (-1.2,-0.2) to [out=0,in=135] (-0.5,-0.5) to [out=-45, in=180] (0,-0.8) to [out=0,in=-135] (0.5,-0.5) to [out=45,in=180] (1.2,-0.2);
\draw[ultra thick, red, <-] (-0.5,0.5) -- (1,-1);
\draw[ultra thick, red, <-] (-1,-1) -- (0.5,0.5);
\node[red] at (-1,-1) [left] {$b_j$};
\node[red] at (-0.5,0.5) [above] {$b_k$};
\node[red] at (1.2,-0.2) [right] {$\alpha_i$};
\node at (1,-0.5) [] {$r$};
\node at (-1,-0.5) [] {$t$};
\node at (0,0.5) [] {$s$};
\node at (0,-1.3) [] {$g$};
\node at (-1,0.5) [] {$f$};
\node at (1,0.5) [] {$h$};
}
\]
Without loss of generality we again assume that the indices of all constituents of $\Phi$ are smaller than $i$, and the indices of all constituents of $B$ are smaller than $j$; due to the fact that $\vec{i}$ is a reduced word, we also know that $j<k$. Similar to the case before, we have
\[
A_f=\inprod{\Phi}{B\wedge b_k}, \quad \quad A_g=\inprod{\Phi\wedge \alpha_i}{B\wedge b_j\wedge b_k}, \quad \quad A_h=\inprod{\Phi}{B\wedge b_j},
\]
and
\[
\inprod{\xi_r}{v}=\inprod{\Phi\wedge \alpha_i}{B\wedge b_j\wedge v}, \quad \quad \inprod{\xi_s}{v}=\inprod{\Phi}{B\wedge v}, \quad \quad \inprod{\xi_t}{v}=\inprod{\Phi\wedge \alpha_i}{B\wedge b_k\wedge v}.
\]
Again, consider the dual vector 
\[
\eta:=A_g\xi_s+A_h\xi_t-A_f\xi_r.
\]
Obviously $\eta$ lies in the span of $\alpha_i$ and the constituents of $\Phi$. Under the generic condition that $A_g$ is non-zero, it follows that $\eta=0$ if and only if $\eta$ vanishes on $b_j$ and $b_k$ and all constituents of $B$. Again it is obvious from the definition of $\eta$ that it vanishes on all constituents of $B$. So we need to just need to verify that $\inprod{\eta}{b_j}=\inprod{\eta}{b_k}=0$:
\begin{align*}
    \inprod{\eta}{b_j}=&\inprod{\Phi\wedge \alpha_i}{B\wedge b_j\wedge b_k}\inprod{\Phi}{B\wedge b_j}+\inprod{\Phi}{B\wedge b_j}\inprod{\Phi\wedge \alpha_i}{B\wedge b_k\wedge b_j}=0;\\
    \inprod{\eta}{b_k}=&\inprod{\Phi\wedge \alpha_i}{B\wedge b_j\wedge b_k}\inprod{\Phi}{B\wedge b_k} -\inprod{\Phi}{B\wedge b_k}\inprod{\Phi\wedge \alpha_i}{B\wedge b_j\wedge b_k}=0.\qedhere
\end{align*}
\end{proof}

Now it is time to introduce more notations. In our construction of the bipartite graph $\Gamma_\vec{i}$, we insisted on adding white vertices to the far left end and far right end of each of the horizontal parallel lines, and now we can make use of them. First for the white vertices on the far left of each horizontal line, we rename the associated dual vectors $\xi_w$ as $\xi_1,\dots, \xi_n$ from top to bottom; then for the white vertices on the far right of each horizontal line, we rename the associated dual vectors $\xi_w$ as $\eta_1,\dots, \eta_n$ from top to bottom as well. 

\begin{prop}\label{basis} For a generic $x\in \GL_n^{u,v}$, $\{\xi_i\}_{i=1}^n$ and $\{\eta_i\}_{i=1}^n$ are two bases for $V^*$.
\end{prop}
\begin{proof} It suffices to show that the pairings $\inprod{\overset{\rightarrow \quad}{\bigwedge_{i=1}^n} \xi_i}{\overset{\rightarrow \quad}{\bigwedge_{j=1}^n} b_j}$ and $\inprod{\overset{\rightarrow\quad }{\bigwedge_{i=1}^n} \eta_i}{\overset{\rightarrow\quad}{\bigwedge_{j=1}^n} b_j}$ do not vanish. Due to the similarity of the proofs, we will just do the first one here. By identity of pairing between wedge products, the pairing $\inprod{\overset{\rightarrow \quad}{\bigwedge_{i=1}^n} \xi_i}{\overset{\rightarrow \quad}{\bigwedge_{j=1}^n} b_j}$ is equal to a sum of products of pairings between constituents of $\overset{\rightarrow \quad}{\bigwedge_{i=1}^n} \xi_i$ and $\overset{\rightarrow \quad}{\bigwedge_{j=1}^n} b_j$ and a sign factor. Consider the constituent $\xi_n$ first. From the definition of $\xi_n$ and the way we constructed $\Gamma_\vec{i}$, we see that up to a sign,
\begin{equation}\label{xi}
\inprod{\xi_n}{v}=\inprod{\overset{\rightarrow \quad}{{\bigwedge}_{i=1}^n} \alpha_i}{\left(\overset{\rightarrow \quad}{{\bigwedge}_{j=1}^{n-1}} b_{v(j)}\right)\wedge v}
\end{equation}
(note that fixing a Coxeter generating set $S$ of $W$ gives an isomorphism $W\cong S_n$, and hence elements of $W$ can act on the set $\{1,\dots, n\}$). Thus the only non-vanishing pairing between $\xi_n$ and $b_j$ is $\inprod{\xi_n}{b_{v(n)}}$. Now go to the next one $\xi_{n-1}$; by a similar argument we can conclude that the only non-vanishing pairing between $\xi_{n-1}$ and $b_j$ is $\inprod{\xi_{n-1}}{b_{v(n)}}$ and $\inprod{\xi_{n-1}}{b_{v(n-1)}}$. Since $b_{v(n)}$ is already paired up with $\xi_n$, this leaves us with no choice but $\inprod{\xi_{n-1}}{b_{v(n-1)}}$... Continuing in this way and using Proposition \ref{Af}, we can conclude that up to a sign,
\[
\inprod{\overset{\rightarrow \quad}{{\bigwedge}_{i=1}^n} \xi_i}{\overset{\rightarrow \quad}{{\bigwedge}_{j=1}^n} b_j}=\prod_{i=1}^n \inprod{\xi_i}{b_{v(i)}}=\prod_{i=1}^n A_i(x).
\]
where $A_i$ denotes the $\mathcal{A}$-coordinate associated to the boundary face on the far left right below the $i$th horizonal line. Since $A_i(x)\neq 0$ for a generic $x\in \GL_n^{u,v}$ (Theorem \ref{bfz}), the proof is complete.
\end{proof}

Now by going backward across $\Gamma_\vec{i}$ (from right to left) and apply Proposition \ref{3-valent} to each black vertex, we see that generically each $\eta_i$ can be expressed as a linear combination of $\xi_i$. Proposition \ref{basis} tells us that for a generic $x\in \GL_n^{u,v}$, both $\{\xi_i\}$ and $\{\eta_i\}$ are bases of $V^*$; thus fromt he simply transitive right action of $\GL_n$ on the space of bases of $V^*$, there is a unique element of $\GL_n$ that takes $\{\xi_i\}$ to $\{\eta_i\}$; our next goal is to find a way to compute such element. 

One good way to do this is to introduce orientations on the edges of $\Gamma_\vec{i}$: we declare that all horizontal edges to point towards the right and all vertical edges to point from the white vertex towards the black vertex. This is in fact a special case of Postnikov's \textbf{perfect orientation}, introduced in his work on non-negative Grassmannian \cite{Pos}. 

\begin{exmp} The following picture is the bipartite graph associated to the reduced word $(1,-1)$ of the pair of Weyl group elements $(w_0,w_0)$ of the group $\GL_2$, with the perfect orientations drawn on top each edge.
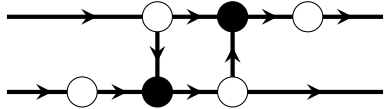
\[
\tikz{
\draw[ultra thick, decoration={markings, mark=at position 0.6 with {\arrow{>}}}, postaction={decorate}] (0,0) -- (1,0);
\draw[ultra thick, decoration={markings, mark=at position 0.6 with {\arrow{>}}}, postaction={decorate}] (1,0) -- (2,0);
\draw[ultra thick, decoration={markings, mark=at position 0.6 with {\arrow{>}}}, postaction={decorate}] (2,0) -- (3,0);
\draw[ultra thick, decoration={markings, mark=at position 0.6 with {\arrow{>}}}, postaction={decorate}] (3,0) -- (5,0);
\draw[ultra thick, decoration={markings, mark=at position 0.6 with {\arrow{>}}}, postaction={decorate}] (0,1) -- (2,1);
\draw[ultra thick, decoration={markings, mark=at position 0.6 with {\arrow{>}}}, postaction={decorate}] (2,1) -- (3,1);
\draw[ultra thick, decoration={markings, mark=at position 0.6 with {\arrow{>}}}, postaction={decorate}] (3,1) -- (4,1);
\draw[ultra thick, decoration={markings, mark=at position 0.6 with {\arrow{>}}}, postaction={decorate}] (4,1) -- (5,1);
\draw[ultra thick, decoration={markings, mark=at position 0.6 with {\arrow{>}}}, postaction={decorate}] (2,1) -- (2,0);
\draw[ultra thick, decoration={markings, mark=at position 0.6 with {\arrow{>}}}, postaction={decorate}] (3,0) -- (3,1);
\draw[fill=black] (2,0) circle [radius=0.2];
\draw[fill=black] (3,1) circle [radius=0.2];
\draw[fill=white] (1,0) circle [radius=0.2];
\draw[fill=white] (2,1) circle [radius=0.2];
\draw[fill=white] (3,0) circle [radius=0.2];
\draw[fill=white] (4,1) circle [radius=0.2];
}
\]
\end{exmp}

Under such a choice of orientations on edges, we see that all the boundary white vertices on the far left are sources and all the boundary white vertices on the far right are sinks. Let $i$ and $j$ be the labeling of a source and a sink respectively; if $\gamma$ is a path going from $i$ to $j$ compatible with the perfect orientation, then we will write $\gamma:i\rightarrow j$. We also denote the set of faces lying below $\gamma$ by $\hat{\gamma}$. Then we claim the following.

\begin{prop}\label{connection} Let $1\leq i,j\leq n$. Then for a generic $x\in \GL_n^{u,v}$, the coefficient of $\xi_i$ in the expansion of $\eta_j$ is given by
\[
\sum_{\gamma:i\rightarrow j}\prod_{f\in \hat{\gamma}} X_f
\]
where $(X_f):=\tilde{p}\circ \tilde{\psi}_\vec{i}(x)$.
\end{prop}
\begin{proof} Just apply Proposition \ref{3-valent} to every 3-valent black vertex, and recall that the quiver $\tilde{\vec{i}}$ is produced by putting a counterclockwisely oriented 3-cycle at each 3-valent black vertex.
\end{proof}

Now we are ready to rediscover the formula of the map $\tilde{\chi}_\vec{i}\circ \tilde{p}_\vec{i}\circ \tilde{\psi}_\vec{i}$.

\begin{cor}\label{3.5} For a generic $x\in \GL_n^{u,v}$, the image $\tilde{\chi}_\vec{i}\circ \tilde{p}_\vec{i}\circ \tilde{\psi}_\vec{i}(x)$ is the unique element that takes $\{\xi_i\}$ to $\{\eta_i\}$, i.e.,
\begin{equation}\label{dt}
(\xi_1,\dots, \xi_n) (\tilde{\chi}_\vec{i}\circ \tilde{p}_\vec{i}\circ \tilde{\psi}_\vec{i}(x))=(\eta_1,\dots, \eta_n),
\end{equation}
where $(\xi_1,\dots, \xi_n)$ and $(\eta_1,\dots, \eta_n)$ are regarded as a row vector with entries $\xi_i$ and $\eta_i$ respectively.
\end{cor}
\begin{proof} To prove this proposition, it suffices to consider how adding an extra vertical edge corresponding to $e_{\pm i}$ on the far right affects the whole picture. Let's consider $e_i$ first. The perfect orientation is going to look like the following.
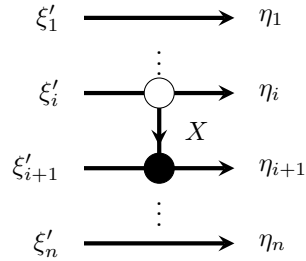
\[
\tikz{
\draw[ultra thick, ->] (0,3) -- (2,3);
\draw[ultra thick, ->] (0,2) -- (2,2);
\draw[ultra thick, ->] (0,1) -- (2,1);
\draw[ultra thick, ->] (0,0) -- (2,0);
\draw[decoration={markings, mark=at position 0.7 with {\arrow{>}}}, postaction={decorate}, ultra thick] (1,2) -- (1,1);
\draw[fill=white] (1,2) circle [radius=0.2];
\draw[fill=black] (1,1) circle [radius=0.2];
\node at (1,2.5) [] {$\vdots$};
\node at (1,0.5) [] {$\vdots$};
\node at (-0.2,0) [left] {$\xi'_n$};
\node at (-0.2,1) [left] {$\xi'_{i+1}$};
\node at (-0.2, 2) [left] {$\xi'_{i}$};
\node at (-0.2,3) [left] {$\xi'_1$};
\node at (2.2,0) [right] {$\eta_n$};
\node at (2.2,1) [right] {$\eta_{i+1}$};
\node at (2.2, 2) [right] {$\eta_{i}$};
\node at (2.2,3) [right] {$\eta_1$};
\node at (1.5,1.5) [] {$X$};
}
\]
It is not hard to see that, according to Proposition \ref{connection},
\[
(\eta_1, \dots, \eta_n)=(\xi'_1,\dots, \xi'_n)\begin{pmatrix}X &  \cdots & 0 & 0 & \cdots & 0 \\
\vdots & \ddots & \vdots & \vdots & \ddots & \vdots \\
0 & \cdots & X & 1 & \cdots & 0 \\
0 & \cdots & 0 & 1 & \cdots & 0 \\
\vdots & \ddots & \vdots & \vdots & \ddots & \vdots \\
0 & \cdots & 0 & 0 & \cdots & 1 \end{pmatrix}=(\xi'_1,\dots, \xi'_n) e^ih^i(X).
\]

Similarly for the case of $e_{-i}$, we have the following picture
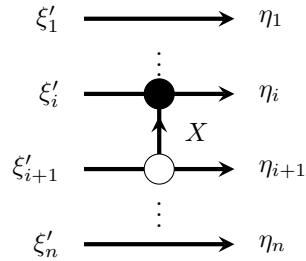
\[
\tikz{
\draw[ultra thick, ->] (0,3) -- (2,3);
\draw[ultra thick, ->] (0,2) -- (2,2);
\draw[ultra thick, ->] (0,1) -- (2,1);
\draw[ultra thick, ->] (0,0) -- (2,0);
\draw[decoration={markings, mark=at position 0.7 with {\arrow{>}}}, postaction={decorate}, ultra thick] (1,1) -- (1,2);
\draw[fill=white] (1,1) circle [radius=0.2];
\draw[fill=black] (1,2) circle [radius=0.2];
\node at (1,2.5) [] {$\vdots$};
\node at (1,0.5) [] {$\vdots$};
\node at (-0.2,0) [left] {$\xi'_n$};
\node at (-0.2,1) [left] {$\xi'_{i+1}$};
\node at (-0.2, 2) [left] {$\xi'_{i}$};
\node at (-0.2,3) [left] {$\xi'_1$};
\node at (2.2,0) [right] {$\eta_n$};
\node at (2.2,1) [right] {$\eta_{i+1}$};
\node at (2.2, 2) [right] {$\eta_{i}$};
\node at (2.2,3) [right] {$\eta_1$};
\node at (1.5,1.5) [] {$X$};
}
\]
and the following identity
\[
(\eta_1,\dots, \eta_n)=(\xi'_1,\dots, \xi'_n) \begin{pmatrix}X^{-1} &  \cdots & 0 & 0 & \cdots & 0 \\
\vdots & \ddots & \vdots & \vdots & \ddots & \vdots \\
0 & \cdots & X & 0 & \cdots & 0 \\
0 & \cdots & X & 1 & \cdots & 0 \\
\vdots & \ddots & \vdots & \vdots & \ddots & \vdots \\
0 & \cdots & 0 & 0 & \cdots & 1 \end{pmatrix} = (\xi'_1,\dots, \xi'_n) e^{-i}h^{i}(X).
\]

By comparing these two cases with the definition of the map $\tilde{\chi}_\vec{i}$, we see that $(\eta_1,\dots, \eta_n)=(\xi'_1,\dots, \xi'_n)(\tilde{\chi}_\vec{i}\circ \tilde{p}_\vec{i}\circ \tilde{\psi}_\vec{i}(x))$.
\end{proof}

Finally we are ready to prove the last part of Proposition \ref{main'}. As we have remarked at the beginning of this subsection, the key is to prove that for a fixed reduced word $\vec{i}$ of $(u,v)$, 
\[
\tilde{\chi}_\vec{i}\circ \tilde{p}_\vec{i}\circ \tilde{\psi}_\vec{i}(x)=D \left(\left[\overline{u}^{-1}x\right]_-^{-1}\overline{u}^{-1} x \overline{v^{-1}}\left[x\overline{v^{-1}}\right]^{-1}_+\right)^t D',
\]
where $D$ and $D'$ are two diagonal matrices which may depend on $x$. The key ingredient is Corollary \ref{3.5}, which is already done. The remaining part of the proof is just to relate $(\xi_1,\dots, \xi_n)$ and $(\eta_1,\dots, \eta_n)$ to Gaussian decompositions of $\overline{u}^{-1}x$ and $x\overline{v^{-1}}$.

Let's start with $(\xi_1,\dots, \xi_n)$. By combining Equations \eqref{b} and \eqref{xi}, we get
\begin{align*}
\xi_n=&\sum_{i=1}^n (-1)^{n-i} \inprod{\bigwedge^\rightarrow_{j=1,\dots, \hat{i},\dots, n} \alpha_j}{\bigwedge^\rightarrow_{k=1,\dots, n-1} b_{v(k)}}\alpha_i\\
=&\sum_{i=1}^n (-1)^{n-i} \inprod{\bigwedge^\rightarrow_{j=1,\dots, \hat{i},\dots, n} \alpha_j}{\bigwedge^\rightarrow_{k=1,\dots, n-1} \left(\sum_{l=1}^n a_lx_{l v(k)} \right)}\alpha_i\\
=&\sum_{i=1}^n (-1)^{n-i} \Delta^{\{1,\dots, n\}\setminus \{i\}, \{1,\dots, n-1\}}\left(x\overline{v^{-1}}\right) \alpha_i
\end{align*}
To better state similar results for all $\xi_i$, we introduce the notation 
\[
I_j:=\{1,\dots, j\}.
\]
Then by a similar computation, we find that
\[
\xi_i=\sum_{j=1}^i (-1)^{i-j}\Delta^{I_i\setminus\{j\},I_{i-1}}\left(x\overline{v^{-1}}\right)\alpha_j.
\]
We can make this expression more concise using the following lemma.

\begin{lem}\label{mij} If $x$ is a Gaussian decomposable element in $\GL_n$, then the matrix
\[
M_{ij}:=(-1)^{i-j}\Delta^{I_i\setminus\{j\},I_{i-1}}(x)
\]
satisfies the identity
\[
M_{ij}=D[x]_-^{-1}
\]
where $D$ is a diagonal matrix and $[x]_-$ is the Gaussian factor as in Equation \eqref{gauss}
\end{lem}
\begin{proof} From the convention $\Delta^{I,J}(x)=0$ if $|I|\neq |J|$ we know that $M_{ij}$ is a lower triangular matrix. By computation we also find that $M_{ij} x$ is an upper triangular matrix. Thus by the uniqueness of Gaussian decomposition, we know that $M_{ij}x$ only differs from $[x]_-^{-1}x=[x]_0[x]^+$ by a diagonal matrix factor on the left. This is equivalent to $M_{ij}=D[x]_-^{-1}$.
\end{proof}

Using this lemma, we now can write
\begin{equation}\label{v}
(\xi_1,\dots, \xi_n)=(\alpha_1,\dots, \alpha_n)\left(D\left[x\overline{v^{-1}}\right]_-^{-1}\right)^t 
\end{equation}
for some diagonal matrix $D$.

Let's now turn to $(\eta_1,\dots, \eta_n)$. By a similar method, we find that
\begin{align*}
\eta_i=&\sum_{j=1}^i (-1)^{i-j}\inprod{\bigwedge^\rightarrow_{k=1,\dots, \hat{j},\dots, i} \alpha_{u^{-1}(k)}}{\bigwedge^\rightarrow_{l=1,\dots, i-1} b_l}\alpha_{u^{-1}(j)}\\
=&\sum_{j=1}^i(-1)^{i-j}\inprod{\bigwedge^\rightarrow_{k=1,\dots, \hat{j},\dots, i} \alpha_{u^{-1}(k)}}{\bigwedge^\rightarrow_{l=1,\dots, i-1} \left(\sum_{m=1}^n a_mx_{ml}\right)}\alpha_{u^{-1}(j)}.
\end{align*}
By using Lemma \ref{mij} again we can write
\begin{equation}\label{u}
(\eta_1,\dots, \eta_n)=(\alpha_1,\dots, \alpha_n)\left(D'\left[\overline{u}^{-1}x\right]_-^{-1}\overline{u}^{-1}\right)^t
\end{equation}
for some diagonal matrix $D'$.

\vspace{12pt}

\noindent\textit{Proof of the last part of Proposition \ref{main'}.} By comparing Equations \eqref{dt}, \eqref{v}, and \eqref{u}, we can conclude that
\begin{align*}
\tilde{\chi}_\vec{i}\circ \tilde{p}_\vec{i}\circ \tilde{\psi}_\vec{i}(x)=&D\left(\left[x\overline{v^{-1}}\right]_-\right)^t\left(\left[\overline{u}^{-1}x\right]_-^{-1}\overline{u}^{-1}\right)^tD'\\
=& D\left(\left[\overline{u}^{-1}x\right]_-^{-1}\overline{u}^{-1}x\overline{v^{-1}}\left[x\overline{v^{-1}}\right]_+^{-1}\right)^tD',
\end{align*}
where we have replace $D$ and $D'$ by some other diagonal matrices; note that the last equality follows from the fact that
\[
\left[x\overline{v^{-1}}\right]_-=x\overline{v^{-1}}\left[x\overline{v^{-1}}\right]_+^{-1}\left[x\overline{v^{-1}}\right]_0^{-1},
\]
the last factor of which is absorbed into the diagonal matrix $D'$. \qed

\subsection{Twist Map and Birational Equivalence}\label{birational} In last subsection we computed a formula for the composition $\chi\circ \psi$, and found that it only differs from Fomin and Zelevinsky's twist map (see Definition 1.5 of \cite{FZ}) by an anti-automorphism $x\mapsto x^\iota$ of $\GL_n$, which is uniquely defined by (Equation (2.2) of \cite{FZ}):
\[
e_{\pm i}^\iota=e_{\pm i} \quad \quad \text{and} \quad \quad a^\iota=a^{-1} \quad \forall a\in H.
\]

On the other hand, in \cite{FG} Fock and Goncharov introduced a involution $i_\mathcal{X}$ of cluster varieties. Given any seed $\vec{i}=(I,\epsilon_{ij})$, we define a new seed $\vec{i}^\circ=(I^\circ,\epsilon_{ij}^\circ)$ by setting
\[
I^\circ=I \quad \quad \text{and} \quad \quad \epsilon_{ij}^\circ=-\epsilon_{ij}.
\]
Since the set of vertices $I^\circ=I$, there is a natural correspondence between coordinates $(X_i)$ of $\mathcal{X}_\vec{i}$ and $(X^\circ_i)$ of $\mathcal{X}_{\vec{i}^\circ}$. Then the involution $i_\mathcal{X}$ is the map from the seed torus $\mathcal{X}_\vec{i}$ to the seed torus $\mathcal{X}_{\vec{i}^\circ}$ defined by
\[
i_\mathcal{X}^*(X_i^\circ)=X_i^{-1}.
\]

It is not hard to see that the involutions $i_\mathcal{X}$ commute with cluster mutations as below.
\[
\xymatrix{ \mathcal{X}_\vec{i} \ar[r]^{\mu_k} \ar[d]_{i_\mathcal{X}} & \mathcal{X}_{\vec{i}'} \ar[d]^{i_\mathcal{X}} \\
\mathcal{X}_{\vec{i}^\circ}\ar[r]_{\mu_k} & \mathcal{X}_{\vec{i}'^\circ}}
\]
Thus the involutions $i_\mathcal{X}$ can be glued into a single involution $i_\mathcal{X}:\mathcal{X}_{|\vec{i}_0|}\rightarrow \mathcal{X}_{|\vec{i}_0^\circ|}$. 

The proposition below shows that in the case of double Bruhat cells in $\GL_n$, the involution $i_\mathcal{X}$ is essentially the same as the anti-automorphism $x\mapsto x^\iota$.

\begin{prop} Let $\vec{i}_0$ be any reduced word for a pair of Weyl group elements $(u,v)$. Then $(u^{-1},v^{-1})$ is also a pair of Weyl group elements and $\vec{i}_0^\circ$ is a reduced word for it. Further the following diagram commutes.
\[
\xymatrix{\mathcal{X}_{|\vec{i}_0|} \ar[r]^(0.4)\chi \ar[d]_{i_\mathcal{X}} & H\backslash \GL_n^{u,v}/H \ar[d]^{\iota} \\ \mathcal{X}_{|\vec{i}_0^\circ|}\ar[r]_(0.3)\chi & H\backslash\GL_n^{u^{-1},v^{-1}}/H}
\]
\end{prop}
\begin{proof} It is obvious that $(u^{-1},v^{-1})$ is also a pair of Weyl group elements. By careful examination we also see that the seed $\vec{i}_0^\circ$ can be produced from the bipartite graph $\Gamma_{\vec{i}_0}^\circ$, which is in turn obtained from $\Gamma_{\vec{i}_0}$, the bipartite graph associated to $\vec{i}_0$, by a flip over any vertical line. Thus the sequence of integers corresponding to the seed $\vec{i}_0^\circ$ is exactly the backward sequence of $\vec{i}_0$, and hence $\vec{i}_0^\circ$ is also a reduced word of $(u^{-1},v^{-1})$. The commutativity of the above diagram then follows from the definitions of $\iota$ and $\chi$.
\end{proof}

As stated in Conjecture 3.12 of \cite{GS}, Goncharov and Shen conjectured that the composition $D_\mathcal{X}:=i_\mathcal{X}\circ \DT$ is an involution. In the case of $H\backslash \GL_n^{u,v}/H$, assuming our main theorem, we see that $D_\mathcal{X}$ is precisely Fomin and Zelevinsky's twist map $t$ (passed to the double quotients), which is indeed a biregular involution (Theorem 1.6 of \cite{FZ}). To give the full picture, we put everything into the following commutative diagram.
\[
\xymatrix{\mathcal{X}_{|\vec{i}_0|}\ar[r]^(0.4)\chi \ar[drr]_{D_\mathcal{X}}& H\backslash \GL_n^{u,v}/H \ar@{-->}[r]^(0.6)\psi \ar[drr]_(0.3){t} & \mathcal{X}_{|\vec{i}_0|} \ar[r]^(0.4)\chi \ar[d]_(0.6){i_\mathcal{X}} & H \backslash \GL_n^{u,v}/H \ar[d]^\iota \\ & & \mathcal{X}_{|\vec{i}_0^\circ|}\ar[r]_(0.3)\chi & H\backslash \GL_n^{u^{-1},v^{-1}}/H}
\]

Using the fact that $t$ is a biregular involution we can prove the first part of our main theorem as an easy corollary.

\begin{cor} Both $\psi$ and $\chi$ are birational equivalences.
\end{cor}
\begin{proof} It suffices to show that $\psi_\vec{i}$ and $\chi_\vec{i}$ are birational equivalences for any reduced word $\vec{i}$ of $(u,v)$. Note that the twist map $t$ is a biregular isomorphic involution that factors through $\mathcal{X}_\vec{i}$ as $(\chi_{\vec{i}^\circ}\circ i_\mathcal{X})\circ \psi_\vec{i}$. Since both $H\backslash \GL_n^{u,v}/H$ and $\mathcal{X}_\vec{i}$ are algebraic varieties of the same dimension, it follows that both $\chi_{\vec{i}^\circ}\circ i_\mathcal{X}$ and $\psi_\vec{i}$ are birational equivalences. But obviously $i_\mathcal{X}$ is also a birational equivalence; thus $\chi_{\vec{i}^\circ}$ is also a birational equivalence, and so is $\chi_\vec{i}$.
\end{proof}

\subsection{Positivity of \texorpdfstring{$\psi\circ \chi$}{}}\label{positive} After showing that both $\chi$ and $\psi$ are birational equivalences, it now makes sense to ask the question that for a (equivalently any) reduced word $\vec{i}$ of a pair of Weyl group elements $(u,v)$, whether the rational map
\[
\psi_\vec{i}\circ \chi_\vec{i}:\mathcal{X}_\vec{i}\dashrightarrow \mathcal{X}_\vec{i}
\]
is positive (It had better be positive, for otherwise we cannot tropicalize it, not to mention proving that it is the Donaldson-Thomas transformation). In this subsection we will show that the answer to this question is positive.

The key to prove this is to lift everything to the top row in Diagram \eqref{chain}, and show that 
\[
\tilde{p}\circ \tilde{\psi}_\vec{i}\circ \tilde{\chi}_\vec{i}:\mathcal{X}_{\tilde{\vec{i}}}\dashrightarrow\mathcal{X}_{\tilde{\vec{i}}}
\]
is positive. Obviously $\tilde{p}$ is a positive map. Thus the problem reduces to showing that
\[
\tilde{\psi}_\vec{i}\circ \tilde{\chi}_\vec{i}:\mathcal{X}_{\tilde{\vec{i}}}\dashrightarrow \mathcal{A}_{\tilde{\vec{i}}}
\]
is positive. In fact, we can prove something even stronger.

\begin{prop} Let $\vec{i}$ be a reduced word of a pair of Weyl group elements $(u,v)$ and let $\Gamma$ be the bipartite graph associated to $\vec{i}$. Then
\[
\Delta^{I,J}\left(\tilde{\chi}_\vec{i}(X_f)\right)=\sum_{\{\gamma_i\}:I\rightarrow J}\prod_i\prod_{f\in \hat{\gamma_i}} X_f,
\]
where $\{\gamma_i\}:I\rightarrow J$ denotes the set of pairwise disjoint paths going from the labeling set $I$ of horizontal lines on the left to the labeling set $J$ of horizontal lines on the right in the perfect orientation of $\Gamma$ (both labeling sets count from top to bottom). In particular, this proposition shows that any minor $\Delta^{I,J}\left(\tilde{\chi}_\vec{i}(X_f)\right)$ is a polynomial with positive integral coefficients in terms of the face variables $X_f$.
\end{prop}
\begin{proof} This proof is credit to Postnikov \cite{Pos}. Let $x=\tilde{\chi}_\vec{i}(X_f)$. Again fix an $n$-dimensional complex vector space $V$ and a basis $\{a_i\}$; let $\{\alpha_i\}$ be dual basis of $\{a_i\}$ and let $\{b_i\}:=\{a_i\}.x$. Then by an argument similar to the proofs of Proposition \ref{connection} and Corollary \ref{3.5} we know that each $b_i$ can be written as the following linear combination.
\[
b_j=\sum_{i=1}^na_i\left(\sum_{\gamma:i\rightarrow j}\prod_{f\in \hat{\gamma}} X_f\right).
\]
Since $x_{ij}=\inprod{\alpha_i}{b_j}$, it follows that the minors of $x$ are in fact polynomials in the face variables $X_f$'s with integral coefficients.

Now it remains to show that the coefficients in such polynomials are positive. Let $\Phi(X_f)$ be one such polynomial. Note that each term in $\Phi(X_f)$ corresponds to a family of $|I|$ number of paths in the perfect orientation of $\Gamma$ going from $I$ to $J$. We claim that if the family of paths are not pairwise disjoint, then there is another term in $\Phi_I(X_f)$ to cancel it out. Suppose $\gamma$ and $\eta$ are two intersecting paths in one such family (the picture looks as if they are not intersecting but they in fact are sharing some edges together). Then we can just switch the beginning parts of $\gamma$ and $\eta$ to obtain $\gamma'$ and $\eta'$ and keep everything else in the family unchanged; the resulting family of paths will give another term in $\Phi(X_f)$
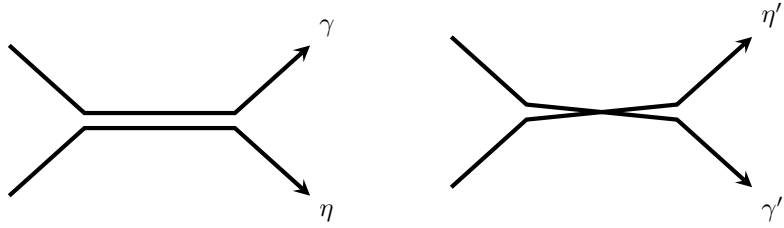
\[
\tikz{
\draw [ultra thick, ->] (0,2) -- (1,1.1) -- (3,1.1) -- (4,2);
\draw [ultra thick, ->] (0,0) -- (1,0.9) -- (3,0.9) --(4,0);
\node at (4,2) [above right] {$\gamma$};
\node at (4,0) [below right] {$\eta$};
}
\quad \quad \quad \quad
\tikz{
\draw [ultra thick, ->] (0,2) -- (1,1.1) -- (3,0.9) -- (4,0);
\draw [ultra thick, ->] (0,0) -- (1,0.9) -- (3,1.1) --(4,2);
\node at (4,2) [above right] {$\eta'$};
\node at (4,0) [below right] {$\gamma'$};
}
\]
Note that the faces that are dominated for the two families of paths are the same; but since we switches the sources of $\gamma$ and $\eta$, the corresponding term in $\Phi(X_f)$ will differ by a minus sign (due to a transposition of rows), and hence these two terms will cancel out. 

Thus the only non-vanishing terms in $\Phi(X_f)$ come from families of pairwise disjoint paths in the perfect orientation of $\Gamma$. Since any such family preserves the ordering on the labeling sets (by Jordan curve theorem), it follows that the coefficient in front of the product of dominated face variables is 1, which completes the proof.
\end{proof}

\subsection{Proof of \texorpdfstring{$\psi\circ\chi$}{} being a cluster transformation}\label{proof1} Now we are ready to prove the remaining part of Proposition \ref{main'}, namely to show that $\psi\circ\chi$ is a cluster Donaldson-Thomas transformation. By Theorem \ref{gs} and Proposition \ref{lem0}, it suffices to show that the following two things:
\begin{enumerate} 
\item[(i).] $\psi\circ \chi$ is a cluster transformation;
\item[(ii).] for some (equivalently any) reduced word $\vec{i}$ of $(u,v)$,
\[
\deg_{X_f}(\psi_\vec{i}\circ\chi_\vec{i})^*(X_g)=-\delta_{fg}.
\]
\end{enumerate}

We will hence break the proof into two parts, and this subsection will be devoted to prove (i). 

Recall from Subsection \ref{flag} that the map $\chi\circ \psi$ can be identified with the map $\eta:[B_1,B_2,B_3,B_4]\mapsto [B_3^*,B_4^*,B_5^*,B_6^*]$ on $\conf^{u,v}(\mathcal{B})$ where the six Borel subgroups can be fit in the following hexagon diagram.
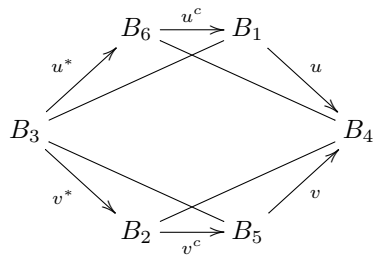
\[
\xymatrix{ & B_6\ar[r]^{u^c} \ar@{-}[drr] & B_1 \ar[dr]^u \ar@{-}[dll] & \\
B_3 \ar[ur]^{u^*} \ar[dr]_{v^*} \ar@{-}[drr] & & &  B_4 \ar@{-}[dll] \\
& B_2 \ar[r]_{v^c} & B_5 \ar[ur]_v &}
\]

The key to show that $\psi\circ \chi$ is a cluster transformation is to break $\eta$ down to a composition of a series of small ``clockwise tilting'' on the square diagram. To be more precise, fix a reduced word $(i(1),\dots, i(m)$ for $u$ and another reduced word $(j(1),\dots, j(l))$ for $v$. Then the juxtaposition 
\[
(-i(1),\dots, -i(m),j(1),\dots, j(l))
\]
is a reduced word for the pair of Weyl group elements $(u,v)$. From the way the reduced word $\vec{i}$ is structured, we see that if $x$ lies in the image of $\tilde{\chi}_\vec{i}$, then $x$ can be written as a product $x_-x_+$ where $x_\pm\in B_\pm$. Fix a choice of such factorization $x=x_-x_+$. Then we know that the point $[x_-^{-1}B_+,x_+B_-,B_-,B_+]$ in $\conf^{u,v}(\mathcal{B})$ corresponds to the equivalence class $H\backslash x/H$ in $H\backslash \GL_n^{u,v}/H$. We can represent such a point by the square diagram below.
\[
\xymatrix{x_-^{-1}B_+ \ar[r]^u \ar@{-}[d] & B_+ \ar@{-}[d] \\
B_- \ar[r]_{v^*} & x_+B_-}
\]

We initiate a sequence of tilting on the edge $\xymatrix{x_-^{-1}B_+\ar@{-}[r] & B_-}$ with respect to this reduced word of $u$. First set $B_u^{(0)}:=x_-^{-1}B_+$ and $B_u^{(m)}:=B_+$; by using Proposition \ref{2.8} we can find a sequence of Borel subgroups $\left(B_u^{(k)}\right)$ such that for $1\leq k\leq m$,
\[
\xymatrix{B_u^{(k-1)}\ar[r]^{s_{i(k)}} & B_u^{(k)}}.
\]
Next set $B_{u^*}^{(0)}:=B_-$ and again use Proposition \ref{2.8} to find a sequence of Borel subgroups $\left(B_{u^*}^{(k)}\right)$ such that for $1\leq k\leq m$,
\[
\xymatrix{B_{u^*}^{(k-1)}\ar[r]^{s_{i(k)}^*} & B_u^{(k)}},
\]
and 
\[
\xymatrix{B_{u^*}^{(m)}\ar[r]^{u^c} & x_-^{-1}B_+}.
\]
Since $s_{i(k+1)}^*\dots s_{i(m)}^*u^c s_{i(1)}\dots s_{i(k)}=w_0$ and $m+l(u^c)=l(w_0)$, it follows from Proposition \ref{2.8} again that for all $0\leq k\leq m$,
\[
\xymatrix{B_{u}^{(k)}\ar@{-}[r] & B_{u^*}^{(k)}}.
\]
Thus we can view these two sequences as a sequence of tilting of the edge $\xymatrix{x_-^{-1}B_+\ar@{-}[r] & B_-}$. This can be seen more intuitively on the original square diagram.
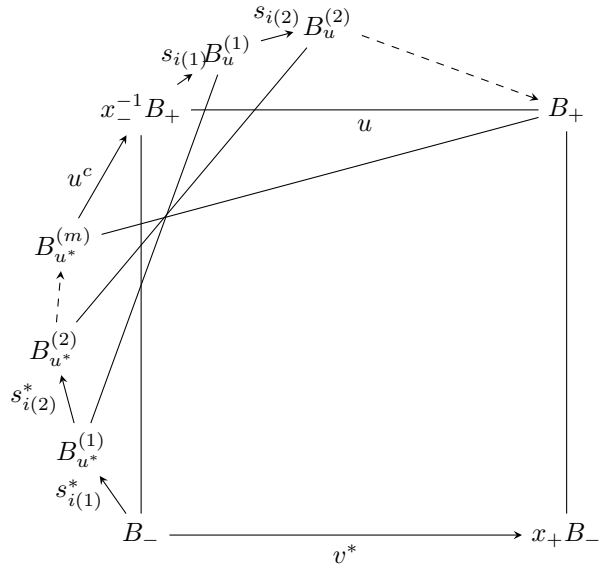
\[
\tikz{
\node (u0) at +(135:4) [] {$x_-^{-1}B_+$};
\node (u1) at +(115:4) [] {$B_u^{(1)}$};
\node (u2) at +(95:4) [] {$B_u^{(2)}$};
\node (um) at +(45:4) [] {$B_+$};
\node (u*0) at +(-135:4) [] {$B_-$};
\node (u*1) at +(-155:4) [] {$B_{u^*}^{(1)}$};
\node (u*2) at +(-175:4) [] {$B_{u^*}^{(2)}$};
\node (u*m) at +(165:4) [] {$B_{u^*}^{(m)}$};
\node (v) at +(-45:4) [] {$x_+B_-$};
\draw [->] (u0) -- (u1) node [midway, above] {$s_{i(1)}$};
\draw [->] (u1) -- (u2) node [midway, above] {$s_{i(2)}$};
\draw [->, dashed] (u2) -- (um);
\draw [->] (u*0) -- (u*1) node [midway, left] {$s_{i(1)}^*$};
\draw [->] (u*1) -- (u*2) node [midway, left] {$s_{i(2)}^*$};
\draw [->] (u*m) -- (u0) node [midway, left] {$u^c$};
\draw [->, dashed] (u*2) -- (u*m);
\draw [->] (u*0) -- (v) node [midway, below] {$v^*$};
\draw (um) -- (v);
\draw (u0) -- (u*0);
\draw (u0) -- (um) node [midway, below] {$u$};
\draw (u*1) -- (u1);
\draw (u*2) -- (u2);
\draw (u*m) -- (um);
}
\]
We can do the same construction for $v$, by finding two sequences $\left(B_{v^*}^{(k)}\right)_{k=1}^l$ and $\left(B_v^{(k)}\right)_{k=1}^l$ to fit into the square diagram in an analogous way. To save time, we will just draw the resulting tilting diagram.
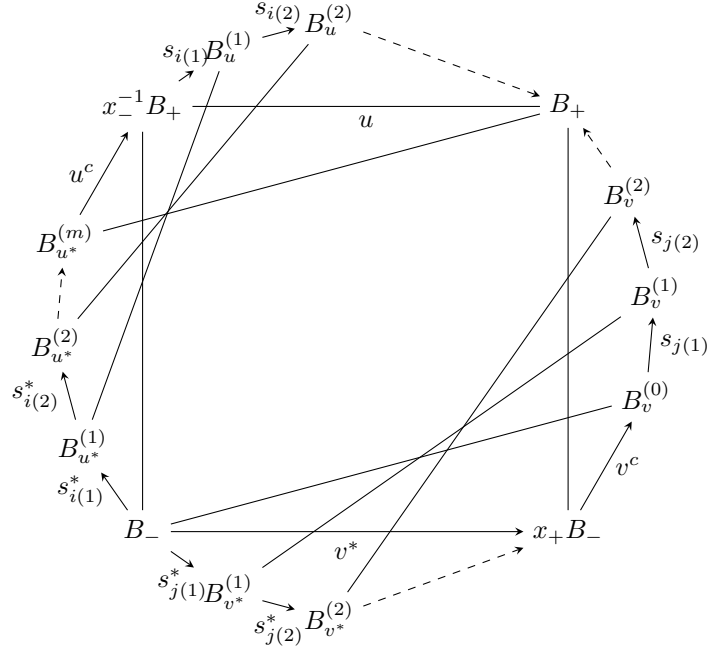
\[
\tikz{
\node (u0) at +(135:4) [] {$x_-^{-1}B_+$};
\node (u1) at +(115:4) [] {$B_u^{(1)}$};
\node (u2) at +(95:4) [] {$B_u^{(2)}$};
\node (um) at +(45:4) [] {$B_+$};
\node (u*0) at +(-135:4) [] {$B_-$};
\node (u*1) at +(-155:4) [] {$B_{u^*}^{(1)}$};
\node (u*2) at +(-175:4) [] {$B_{u^*}^{(2)}$};
\node (u*m) at +(165:4) [] {$B_{u^*}^{(m)}$};
\node (v*l) at +(-45:4) [] {$x_+B_-$};
\node (v*1) at (-115:4) [] {$B_{v^*}^{(1)}$};
\node (v*2) at (-95:4) [] {$B_{v^*}^{(2)}$};
\node (v0) at (-15:4) [] {$B_v^{(0)}$};
\node (v1) at (5:4) [] {$B_v^{(1)}$};
\node (v2) at (25:4) [] {$B_v^{(2)}$};
\draw [->] (u0) -- (u1) node [midway, above] {$s_{i(1)}$};
\draw [->] (u1) -- (u2) node [midway, above] {$s_{i(2)}$};
\draw [->, dashed] (u2) -- (um);
\draw [->] (u*0) -- (u*1) node [midway, left] {$s_{i(1)}^*$};
\draw [->] (u*1) -- (u*2) node [midway, left] {$s_{i(2)}^*$};
\draw [->] (u*m) -- (u0) node [midway, left] {$u^c$};
\draw [->, dashed] (u*2) -- (u*m);
\draw [->] (u*0) -- (v*l) node [midway, below] {$v^*$};
\draw (um) -- (v*l);
\draw (u0) -- (u*0);
\draw (u0) -- (um) node [midway, below] {$u$};
\draw (u*1) -- (u1);
\draw (u*2) -- (u2);
\draw (u*m) -- (um);
\draw [->] (u*0) -- (v*1) node [midway, below] {$s_{j(1)}^*$};
\draw [->] (v*1) -- (v*2) node [midway, below] {$s_{j(2)}^*$};
\draw [dashed, ->] (v*2) -- (v*l);
\draw [->] (v*l) -- (v0) node [midway, right] {$v^c$};
\draw [->] (v0) -- (v1) node [midway, right] {$s_{j(1)}$};
\draw [->] (v1) -- (v2) node [midway, right] {$s_{j(2)}$};
\draw [dashed, ->] (v2) -- (um);
\draw (u*0) -- (v0);
\draw (v*1) -- (v1);
\draw (v*2) -- (v2);
}
\]
One can see that we are tilting the right vertical edge $\xymatrix{B_+\ar@{-}[r] & x_+B_-}$ clockwisely a bit at a time, going from index $l$, $l-1$, and so on, till we get to $\xymatrix{B_v^{(0)}\ar@{-}[r]& B_-}$. Note that by the time we finish both tilting sequences, we can apply $*$ to the final square diagram and obtain $\eta[x_-^{-1}B_+,x_+B_-,B_-,B_+]=\left[B_-^*, B_+^*,\left(B_{v}^{(0)}\right)^*, \left(B_{u^*}^{(m)}\right)^*\right]$.
\[
\xymatrix{B_-^* \ar[r]^u \ar@{-}[d] & \left(B_{u^*}^{(m)}\right)^* \ar@{-}[d]\\
\left(B_v^{(0)}\right)^* \ar[r]_{v^*} &B_+^*}
\]

Our next mission is to figure out how to realize such a tilting in terms of operations on the element $x$ in $\GL_n^{u,v}$. By viewing $x$ as a transformation on pairs of opposite Borel subgroups, we can break down the square diagram of the quadruple $[x_-^{-1}B_+,x_+B_-,B_-,B_+]$ into a two-step process: first taking the opposite pair $\xymatrix{x_-^{-1}B_+\ar@{-}[r] & B_-}$ to the opposite pair $\xymatrix{B_+ \ar@{-}[r]& B_-}$ and then to the opposite pair $\xymatrix{B_+ \ar@{-}[r]& x_+B_-}$. We can pictorially represent it as below.
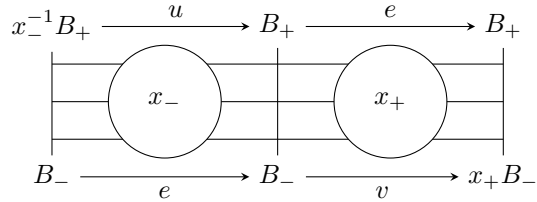
\[
\tikz{
\node (lt) at (0,1) [] {$x_-^{-1}B_+$};
\node (lb) at (0,-1) [] {$B_-$};
\node (mt) at (3,1) [] {$B_+$};
\node (mb) at (3,-1) [] {$B_-$};
\node (rt) at (6,1) [] {$B_+$};
\node (rb) at (6,-1) [] {$x_+B_-$};
\draw (lt) -- (lb);
\draw (mt) -- (mb);
\draw (rt) -- (rb);
\draw [->] (lt) -- (mt) node [midway, above] {$u$};
\draw [->] (mt) -- (rt) node [midway, above] {$e$};
\draw [->] (lb) -- (mb) node [midway, below] {$e$};
\draw [->] (mb) -- (rb) node [midway, below] {$v$};
\draw (0,0.5) -- (6,0.5);
\draw (0,-0.5) -- (6,-0.5);
\draw (0,0) -- (6,0);
\draw[fill=white] (1.5,0) circle [radius=0.75];
\draw[fill=white] (4.5,0) circle [radius=0.75];
\node at (1.5,0) [] {$x_-$};
\node at (4.5,0) [] {$x_+$};
}
\]
In the diagram above, note that the middle opposite pair $\xymatrix{B_+\ar@{-}[r] & B-}$ corresponds to the diagonal in the tilting diagram that separates the two tilting sequence, so we can expect that the tilting sequence for $u$ will take place on the left half of the diagram above, whereas the tilting sequence for $v$ will take place on the right half. Let's consider the tilting sequence for $u$ first.

\begin{prop}\label{u tilt} Based on the representative $(x_-^{-1}B_+,B_+,B_-,x_+B_-)$, the following identity holds for $0\leq k\leq m$:
\[
\tikz{
\node (a) at (0,1) [] {$B_u^{(k)}$};
\node (b) at (0,-1) [] {$B_{u^*}^{(k)}$};
\node (c) at (5,1) [] {$\left(e_{i(k)}e_{-i(k)}^{-1}\dots e_{i(1)}e_{-i(1)}^{-1}x_-\right)^{-1}B_+$};
\node (d) at (5,-1) [] {$\left(e_{i(k)}e_{-i(k)}^{-1}\dots e_{i(1)}e_{-i(1)}^{-1}x_-\right)^{-1}B_-$};
\draw (a) -- (b); 
\draw (c) -- (d); 
\node at (1,0) [] {$=$};
}.
\]
\end{prop}
\begin{proof} We just need to verify the defining conditions for $B_u^{(k)}$ and $B_{u^*}^{(k)}$, and the key facts are
\[
e_{i(k)}\in B_+\cap B_-s_{i(k)}B_- \quad \quad \text{and} \quad \quad e_{-i(k)}\in B_+s_{i(k)}B_+\cap B_-.
\]

Let's first look at $B_u^{(k)}$. Obviously $B_u^{(0)}=x_-^{-1}B_+$, and since
\[
\left(e_{i(k-1)}e_{-i(k-1)}^{-1}\dots e_{i(1)}e_{-i(1)}^{-1}x_-\right)\left(e_{i(k)}e_{-i(k)}^{-1}\dots e_{i(1)}e_{-i(1)}^{-1}x_-\right)^{-1}=e_{-i(k)}e_{i(k)}^{-1}\in B_+s_{i(k)}B_+,
\]
we know that $\xymatrix{\left(e_{i(k-1)}e_{-i(k-1)}^{-1}\dots e_{i(1)}e_{-i(1)}^{-1}x_-\right)^{-1}B_+ \ar[r]^{s_{i(k)}} & \left(e_{i(k)}e_{-i(k)}^{-1}\dots e_{i(1)}e_{-i(1)}^{-1}x_-\right)^{-1}B_+}$. Thus we only need to show that $\left(e_{i(m)}e_{-i(m)}^{-1}\dots e_{i(1)}e_{-i(1)}^{-1}x_-\right)^{-1}B_+=B_+$, which is equivalent to showing that $e_{i(m)}e_{-i(m)}^{-1}\dots e_{i(1)}e_{-i(1)}^{-1}x_-$ is an element of $B_+$. But if we look at the left half of the bipartite graph $\Gamma_\vec{i}$ (which is a bipartite graph for $x_-$), we see that multiplying $e_{i(1)}e_{-i(1)}^{-1}$ on the left of $x_-$ turns the left most vertical edge from $\tikz[baseline=-0.5ex]{\draw[ultra thick] (0,0.5) -- (0,-0.5); \draw[fill=black] (0,0.5) circle [radius=0.2]; \draw[fill=white] (0,-0.5) circle [radius=0.2];}$ to $\tikz[baseline=-0.5ex]{\draw[ultra thick] (0,0.5) -- (0,-0.5); \draw[fill=white] (0,0.5) circle [radius=0.2]; \draw[fill=black] (0,-0.5) circle [radius=0.2];}$; then we can move this edge to the right of $x_-$ using the two moves in Proposition \ref{4.1}; similar arguments apply to $e_{i(2)}e_{-i(2)}^{-1}$ and so on. Thus at then end we will obtain a bipartite graph with no vertical edge of the form $\tikz[baseline=-0.5ex]{\draw[ultra thick] (0,0.5) -- (0,-0.5); \draw[fill=black] (0,0.5) circle [radius=0.2]; \draw[fill=white] (0,-0.5) circle [radius=0.2];}$, and hence $e_{i(m)}e_{-i(m)}^{-1}\dots e_{i(1)}e_{-i(1)}^{-1}x_-$ is an element of $B_+$.

As for $B_{u^*}^{(k)}$, we see that $B_{u^*}^{(0)}=x_-^{-1}B_-=B_-$, and since
\[
\left(e_{i(k-1)}e_{-i(k-1)}^{-1}\dots e_{i(1)}e_{-i(1)}^{-1}x_-\right)\left(e_{i(k)}e_{-i(k)}^{-1}\dots e_{i(1)}e_{-i(1)}^{-1}x_-\right)^{-1}=e_{-i(k)}e_{i(k)}^{-1}\in B_-s_{i(k)}B_-,
\]
we know that $\xymatrix{\left(e_{i(k-1)}e_{-i(k-1)}^{-1}\dots e_{i(1)}e_{-i(1)}^{-1}x_-\right)^{-1}B_- \ar[r]^{s_{i(k)}^*} & \left(e_{i(k)}e_{-i(k)}^{-1}\dots e_{i(1)}e_{-i(1)}^{-1}x_-\right)^{-1}B_-}$. Thus we only need to show that
\[
\xymatrix{\left(e_{i(m)}e_{-i(m)}^{-1}\dots e_{i(1)}e_{-i(1)}^{-1}x_-\right)^{-1}B_-\ar[r]^(0.75){u^c} & x_-^{-1}B_+}.
\]
But this is equivalent to showing that
\[
\overline{w}_0e_{i(m)}e_{-i(m)}^{-1}\dots e_{i(1)}e_{-i(1)}^{-1}=\overline{u^c}\overline{u}e_{i(m)}e_{-i(m)}^{-1}\dots e_{i(1)}e_{-i(1)}^{-1}\in B_+u^cB_+,
\]
which is also equivalent to showing that
\[
\overline{u}e_{i(m)}e_{-i(m)}^{-1}\dots e_{i(1)}e_{-i(1)}^{-1}\in B_+.
\]
To show this, we recall that $\overline{s}_i:=e_i^{-1}e_{-i}e_i^{-1}$; thus
\[
\overline{u}e_{i(m)}e_{-i(m)}^{-1}\dots e_{i(1)}e_{-i(1)}^{-1}=\overline{s_{i(1)}\dots s_{i(m-1)}} e_{i(m)}^{-1}e_{i(m-1)}e_{-i(m-1)}^{-1}\dots e_{i(1)}e_{-i(1)}^{-1}.
\]
But then since $(i(1),\dots, i(m))$ is a reduced word of $u$, $s_{i(1)}\dots s_{i(m-1)}$ maps the simple root $\alpha_{i(m)}$ to a positive root, which implies that 
\[
\overline{s_{i(1)}\dots s_{i(m-1)}} e_{i(m)}^{-1}=b\overline{s_{i(1)}\dots s_{i(m-1)}}
\]
for some $b\in B_+$. The proof is then finished by induction on $m$.
\end{proof}

By a completely analogous proof one can also show the following proposition.

\begin{prop} Based on the representative $(x_-^{-1}B_+,B_+,B_-,x_+B_-)$, the following identity holds for $0\leq k\leq l$:
\[
\tikz{
\node (a) at (0,1) [] {$B_v^{(k)}$};
\node (b) at (0,-1) [] {$B_{v^*}^{(k)}$};
\node (c) at (5,1) [] {$x_+e_{j(l)}^{-1}e_{-j(l)}\dots e_{j(k+1)}^{-1}e_{-j(k+1)}B_+$};
\node (d) at (5,-1) [] {$x_+e_{j(l)}^{-1}e_{-j(l)}\dots e_{j(k+1)}^{-1}e_{-j(k+1)}B_-$};
\draw (a) -- (b); 
\draw (c) -- (d); 
\node at (1,0) [] {$=$};
}.
\]
\end{prop}

Our last two propositions show that in order to reflect the two tilting sequences in terms of $x$, all we need to do is to multiply $e_{i(m)}e_{-i(m)}^{-1}\dots e_{i(1)}e_{-i(1)}^{-1}$ on the left and $e_{j(l)}^{-1}e_{-j(l)}\dots e_{j(1)}^{-1}e_{-j(1)}$ on the right. With these two results in our pockets, we are ready to prove our last proposition.

\begin{prop}\label{3.18} $\psi\circ \chi$ is a cluster transformation.
\end{prop}
\begin{proof} Suppose $H\backslash x/H=\chi_\vec{i}(X_f)$ and consider the representative $[x_-^{-1}B_+,x_+B_-,B_-,B_+]$ corresponding to $H\backslash x/H$. We learned from Proposition \ref{eta} that $\chi\circ \psi$ on $H\backslash \GL_n^{u,v}/H$ is the same as the map $\eta$ on $\conf^{u,v}(\mathcal{B})$, and hence $\chi \circ \psi (H\backslash x/H)$ will correspond to the equivalence class of the configuration
\begin{align*}
& \eta\left(\tikz[baseline=-0.5ex]{
\node (a) at (0,1) [] {$x_-^{-1}B_+$};
\node (b) at (0,-1) [] {$B_-$};
\node (c) at (2,1) [] {$B_+$};
\node (d) at (2,-1) [] {$x_+B_-$};
\draw [->] (a) -- (c) node [midway, above] {$u$};
\draw [->] (b) -- (d) node [midway, below] {$v^*$};
\draw (a) -- (b);
\draw (c) -- (d);
}\right)\\
=& \tikz[baseline=-0.5ex]{
\node (a) at (0,1) [] {$B_-$};
\node (b) at (0,-1) [] {$\left(x_+e_{j(l)}^{-1}e_{-j(l)}\dots e_{j(1)}^{-1}e_{-j(1)}\right)^*B_+$};
\node (c) at (5,1) [] {$\left(\left(e_{i(m)}e_{-i(m)}^{-1}\dots e_{i(1)}e_{-i(1)}^{-1}x_-\right)^{-1}\right)^*B_-$};
\node (d) at (5,-1) [] {$B_+$};
\draw [->] (a) -- (c) node [midway, above] {$u$};
\draw [->] (b) -- (d) node [midway, below] {$v^*$};
\draw (a) -- (b);
\draw (c) -- (d);
}\\
=&\tikz[baseline=-0.5ex]{
\node (a) at (0,1) [] {$\left(x_+e_{j(l)}^{-1}e_{-j(l)}\dots e_{j(1)}^{-1}e_{-j(1)}\right)^*\overline{w}^{-1}_0B_+$};
\node (b) at (0,-1) [] {$\left(x_+e_{j(l)}^{-1}e_{-j(l)}\dots e_{j(1)}^{-1}e_{-j(1)}\right)^*\overline{w}^{-1}_0B_-$};
\node (c) at (7,1) [] {$\left(\left(e_{i(m)}e_{-i(m)}^{-1}\dots e_{i(1)}e_{-i(1)}^{-1}x_-\right)^{-1}\right)^*\overline{w}^{-1}_0B_+$};
\node (d) at (7,-1) [] {$\left(\left(e_{i(m)}e_{-i(m)}^{-1}\dots e_{i(1)}e_{-i(1)}^{-1}x_-\right)^{-1}\right)^*\overline{w}^{-1}_0B_-$};
\draw [->] (a) -- (c) node [midway, above] {$u$};
\draw [->] (b) -- (d) node [midway, below] {$v^*$};
\draw (a) -- (b);
\draw (c) -- (d);
}.
\end{align*}
Note that the last result corresponds to the element
\[
H\backslash\left(e_{i(m)}e_{-i(m)}^{-1}\dots e_{i(1)}e_{-i(1)}^{-1}xe_{j(l)}^{-1}e_{-j(l)}\dots e_{j(1)}^{-1}e_{-j(1)}\right)^t/H
\]
in $H\backslash \GL_n^{u,v}/H$. If we look at the bipartite graph $\Gamma_\vec{i}$, as we have argued in the proof of Proposition \ref{u tilt}, each time when we multiply $e_{i(k)}e_{-i(k)}^{-1}$ on the left, all we is doing is change the left most vertical edge from $\tikz[baseline=-0.5ex]{\draw[ultra thick] (0,0.5) -- (0,-0.5); \draw[fill=black] (0,0.5) circle [radius=0.2]; \draw[fill=white] (0,-0.5) circle [radius=0.2];}$ to $\tikz[baseline=-0.5ex]{\draw[ultra thick] (0,0.5) -- (0,-0.5); \draw[fill=white] (0,0.5) circle [radius=0.2]; \draw[fill=black] (0,-0.5) circle [radius=0.2];}$ and then move it to the middle (after all the vertical edges of them form $\tikz[baseline=-0.5ex]{\draw[ultra thick] (0,0.5) -- (0,-0.5); \draw[fill=black] (0,0.5) circle [radius=0.2]; \draw[fill=white] (0,-0.5) circle [radius=0.2];}$). Since $H\backslash \GL_n^{u,v}/H$ corresponds to the boundary-removed quiver $\vec{i}$, changing the leftmost vertical edge from $\tikz[baseline=-0.5ex]{\draw[ultra thick] (0,0.5) -- (0,-0.5); \draw[fill=black] (0,0.5) circle [radius=0.2]; \draw[fill=white] (0,-0.5) circle [radius=0.2];}$ to $\tikz[baseline=-0.5ex]{\draw[ultra thick] (0,0.5) -- (0,-0.5); \draw[fill=white] (0,0.5) circle [radius=0.2]; \draw[fill=black] (0,-0.5) circle [radius=0.2];}$ does nothing to the quiver, and moving it to the middle is just a sequence of quiver mutation, which gives rise to a corresponding sequence of cluster mutation on the cluster variety $\mathcal{X}_{|\vec{i}|}$. Similar argument also applies to each time we multiply $e_{j(k)}^{-1}e_{-j(k)}$ on the right. Lastly taking transposition flips the bipartite graph over a horizontal line and simultaneously switches the color of all vertices without changing the face variables. Thus transposition gives rise to an quiver isomorphism which in turn produces a cluster isomorphism on $\mathcal{X}_{|\vec{i}|}$. Combining these observations we see that $\psi\circ \chi$ is a composition of cluster mutations and cluster isomorphisms, which is by definition a cluster transformation. 
\end{proof}

\subsection{Proof of \texorpdfstring{$\psi\circ\chi$}{} being a cluster Donaldson-Thomas Transformation}\label{proof2}

In this subsection we will finish part (ii) of our proof, which is to show that for some reduced word $\vec{i}$ of $(u,v)$, 
\[
\deg_{X_f}\left(\psi_\vec{i}\circ \chi_\vec{i}\right)^*(X_g)=-\delta_{fg}.
\]
To do so, we need to choose a easy reduced word $\vec{i}$ of $(u,v)$ to work with; the one that we choose for the proof is called ``greedy reduced word'', which we introduce now.

Recall that the Coxeter generating set $S$ gives an identification between the Weyl group $W$ of $\GL_n$ and the symmetric group $S_n$, and the Coxeter generators $s_i$ are precisely the adjacent transpositions switching $i$ and $i+1$. 

\begin{defn} The \textit{greedy reduced word} $\vec{i}_w$ of $w$ is constructed via the following recursive procedure.
\begin{enumerate}
    \item The greedy reduced word of the identity element $e\in S_1$ is the empty word.
    \item Suppose $n>1$. Note that the sequence $s_{w^{-1}(n)} s_{w^{-1}(n)+1}\dots s_{n-1}$ moves $w^{-1}(n)$ to $n$. Thus $n$ is a fixed point of 
    \[
    w':=\left(s_{w^{-1}(n)}s_{w^{-1}(n)+1}\dots s_{n-1}\right)^{-1}w,
    \] 
    and hence we can view $w'$ as an element of $S_{n-1}$.
    \item Find the greedy reduced word $\vec{i}_{w'}$ of $w'$; then the greedy reduced word of $w$ is 
    \[
    \left(w^{-1}(n), w^{-1}(n)+1,\dots, n-1\right)\#\vec{i}_{w'},
    \]
    where the symbol $\#$ denotes juxtaposition of sequences.
\end{enumerate}
For a pair of Weyl group elements, let $\vec{i}_u$ and $\vec{i}_{v^{-1}}$ be the greedy reduced words of $u$ and $v^{-1}$ respectively; we then define the \textit{greedy reduced word} of $(u,v)$ to be
\[
\vec{i}_{(u,v)}=\left(-\vec{i}_u\right)\#\vec{i}^\circ_{v^{-1}},
\]
where the minus sign is due to the convention that the negative letters in a reduced word for $(u,v)$ should correspond a reduced word of $u$; notice that $\vec{i}^\circ_{v^{-1}}$ is the reverse sequence of $\vec{i}_{v^{-1}}$ and hence is a reduced word of $v$.
\end{defn}

Before we proceed, we should prove the following statement.

\begin{prop} The greedy reduced word $\vec{i}_w$ is indeed a reduced word of $w$.
\end{prop}
\begin{proof} Here we use the following fact about symmetric group: the length of an element $w\in S_n$ is equal to the $\sum_{i=1}^n \max\left(w(i)-i,0\right)$. Note that in $\left(w^{-1}(n),w^{-1}(n)+1,\dots, n-1\right)$ there are precisely $n-w^{-1}(n)$ letters. Thus by induction the sequence $\left(w^{-1}(n),w^{-1}(n)+1, \dots, n-1\right)\#\vec{i}_{w'}$ is of length $\sum_{i=1}^n\max\left(w(i)-i,0\right)=l(w)$. 
\end{proof}

\begin{exmp}\label{3.17} In the case of $\GL_3$, the greedy reduced word for the pair of Weyl group elements $(w_0,w_0)$ is $(-1,-2,-1,1,2,1)$, and the associated bipartite graph is the following.
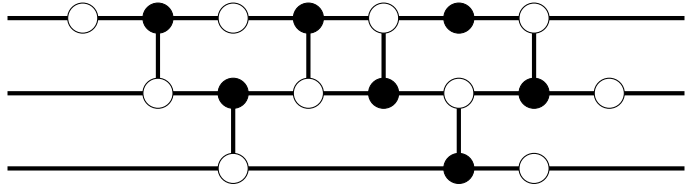
\[
\tikz{
\draw [ultra thick] (-1,0) -- (8,0);
\draw [ultra thick] (-1,1) -- (8,1);
\draw [ultra thick] (-1,2) -- (8,2);
\draw [ultra thick] (1,1) -- (1,2);
\draw [ultra thick] (2,0) -- (2,1);
\draw [ultra thick] (3,1) -- (3,2);
\draw [ultra thick] (4,1) -- (4,2);
\draw [ultra thick] (5,0) -- (5,1);
\draw [ultra thick] (6,1) -- (6,2);
\draw [fill=black] (1,2) circle [radius=0.2];
\draw [fill=black] (2,1) circle [radius=0.2];
\draw [fill=black] (3,2) circle [radius=0.2];
\draw [fill=black] (4,1) circle [radius=0.2];
\draw [fill=black] (5,0) circle [radius=0.2];
\draw [fill=black] (6,1) circle [radius=0.2];
\draw [fill=black] (5,2) circle [radius=0.2];
\draw [fill=white] (0,2) circle [radius=0.2];
\draw [fill=white] (1,1) circle [radius=0.2];
\draw [fill=white] (2,0) circle [radius=0.2];
\draw [fill=white] (3,1) circle [radius=0.2];
\draw [fill=white] (4,2) circle [radius=0.2];
\draw [fill=white] (5,1) circle [radius=0.2];
\draw [fill=white] (6,2) circle [radius=0.2];
\draw [fill=white] (2,2) circle [radius=0.2];
\draw [fill=white] (6,0) circle [radius=0.2];
\draw [fill=white] (7,1) circle [radius=0.2];
}
\]
\end{exmp}

The reason to introduce the greedy reduced word $\vec{i}_{(u,v)}$ is because of the following important property.

\begin{prop}\label{greedy} Let $\Gamma_{\vec{i}_{(u,v)}}$ be the bipartite graph associated to the greedy reduced word $\vec{i}_{(u,v)}$. If $A_f=\Delta^{I,J}$, then there is a unique family of pairwise disjoint paths going from the external edges on the left labeled by elements of $I$ to the external edges on the right labeled by elements of $J$ in the perfect orientation of $\Gamma_{\vec{i}_{(u,v)}}$, which corresponds to a term in the polynomial $\Delta^{I,J}\left(\tilde{\chi}_{\vec{i}_{(u,v)}}(X_g)\right)$ that maximizes the degree of every face variable $X_g$ simultaneously.
\end{prop}
\begin{proof} The statement is trivial for the top face and the bottom face of $\Gamma_{\vec{i}_{(u,v)}}$. Let $f$ be a face bound between the $k$th horizontal line and the $(k+1)$th horizontal line. Note that $f$ is bound between two vertical edges (or one if it is a boundary face). Let $\vec{i}_{<f}$ be the subsequence of $\vec{i}_{(u,v)}$ formed by the letters corresponding to the vertical edges of the form $\tikz[baseline=-0.5ex]{\draw[ultra thick] (0,0.5) -- (0,-0.5); \draw[fill=black] (0,0.5) circle [radius=0.2]; \draw[fill=white] (0,-0.5) circle [radius=0.2];}$ occurring to the left of $f$ (including its left edge) and let $\vec{i}_{>f}$ be the subsequence of $\vec{i}_{(u,v)}$ formed by the letters corresponding to the vertical edges of the form $\tikz[baseline=-0.5ex]{\draw[ultra thick] (0,0.5) -- (0,-0.5); \draw[fill=white] (0,0.5) circle [radius=0.2]; \draw[fill=black] (0,-0.5) circle [radius=0.2];}$ occurring to the right of $f$ (including its right edge). Let $u_{<f}$ be the Weyl group element associated to the reduced word $\vec{i}_{<f}$, and let $v_{>f}$ be the Weyl group element associated to the reduced word $\vec{i}_{>f}$. It is not hard to see from the definition of $A_f=\Delta^{I,J}$ that
\[
I=\left\{u_{<f}^{-1}(1),\dots, u_{<f}^{-1}(k)\right\} \quad \quad \text{and} \quad \quad J=\left\{v_{>f}(1),\dots, v_{>f}(k)\right\}.
\]

We now can describe the family of pairwise disjoint paths claimed in the statement. Index the elements in the sets $I$ and $J$ as $i_1,\dots, i_k$ and $j_1,\dots, j_k$ is ascending order. Since we need a family of pairwise disjoint paths going from $I$ to $J$, $i_1$ had better go to $j_1$, and $i_2$ had better go to $j_2$, and so on. Now starting with $i_1$, travel along the $i$th horizontal line until it hits the first vertical edge (must be of the form $\tikz[baseline=-0.5ex]{\draw[ultra thick] (0,0.5) -- (0,-0.5); \draw[fill=black] (0,0.5) circle [radius=0.2]; \draw[fill=white] (0,-0.5) circle [radius=0.2];}$ by the definition of greedy reduced word), and then go up, and then travel along the horizontal line until it hits another vertical edge (must of be of the form $\tikz[baseline=-0.5ex]{\draw[ultra thick] (0,0.5) -- (0,-0.5); \draw[fill=black] (0,0.5) circle [radius=0.2]; \draw[fill=white] (0,-0.5) circle [radius=0.2];}$ again), and then go up, and then repeat the same process until it arrives at the 1st horizontal line, and then travel along the 1st horizontal line until it goes over face $f$; repeat the same process for $i_2$ until it goes along the 2nd horizontal line over face $f$, and then repeat for $i_3$ and so on. The reason we can do this without the paths intersecting each other is because in the definition of the greedy reduced word, we have prioritized the process of moving $\{1,\dots, n\}\setminus I$ towards the end; up until all elements of $\{1,\dots, n\}\setminus I$ are moved to $\{k+1,\dots, n\}$, elements in $I$ always decreases (go up in the perfect orientation of $\Gamma_{\vec{i}_{(u,v)}}$) while keeping their ordering.

Analogously, since we have the greedy reduced word $\vec{i}_{v^{-1}}^\circ$ on the right half of $\Gamma_{\vec{i}_{(u,v)}}$, we can do the same process backward from right to left on $\Gamma_{\vec{i}_{(u,v)}}$ against the perfect orientation, starting from the set $J$ and end up at the horizontal lines from $1$ to $k$ over the face $f$. Now joint the paths from left and right correspondingly together and declare to go from left to right along the perfect orientation. Then at last we obtain our family of pairwise disjoint paths going from $I$ to $J$. Note that it is obvious by construction that no other family of pairwise disjoint paths going from $I$ to $J$ can possibly dominate more faces than the family constructed in this proof. Therefore this family of pairwise disjoint paths maximizes the degree of every face variable in the corresponding term in $\Delta^{I,J}\left(\tilde{\chi}_{\vec{i}_{(u,v)}}(X_g)\right)$ by construction.
\end{proof}

With the help of Proposition \ref{greedy} we can now compute 
\[
\deg_{X_f}\left(\psi_{\vec{i}_{(u,v)}}\circ \chi_{\vec{i}_{(u,v)}}\right)^*(X_g)=\deg_{X_f}\left(\tilde{p}_{\vec{i}_{(u,v)}}\circ \tilde{\psi}_{\vec{i}_{(u,v)}}\circ\tilde{\chi}_{\vec{i}_{(u,v)}}\right)^*(X_g)
\]
for non-boundary faces $f$ and $g$. 

First note from the definition of the greedy reduced word that, a non-boundary face $g$ can only be one of the following three cases (the dashed edges mean that they may or may not be present).
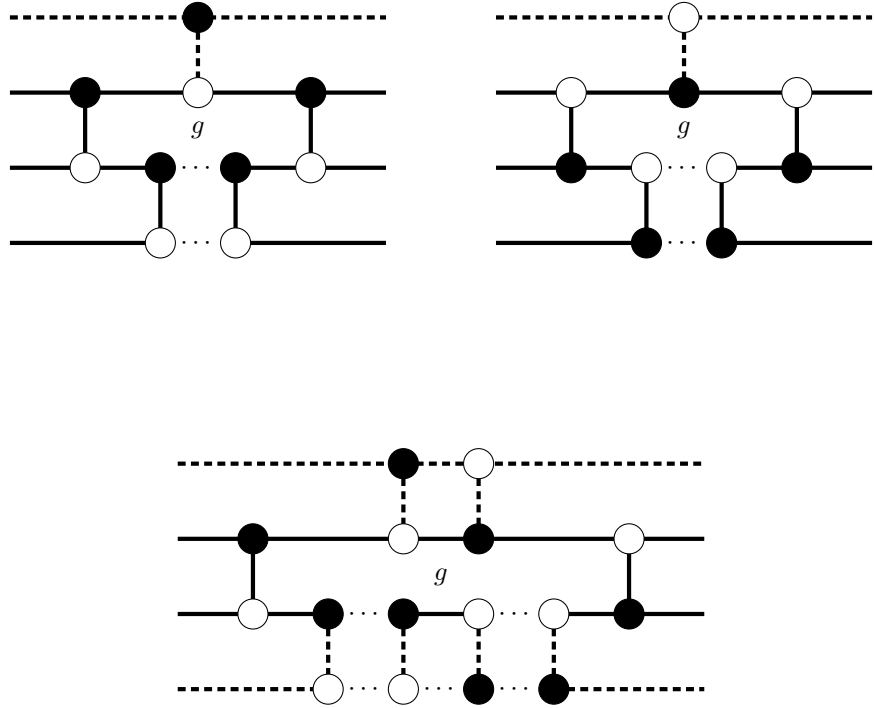
\[
\tikz[baseline=2ex]{
\draw[ultra thick] (0,0) -- (2,0);
\draw[ultra thick] (3,0) -- (5,0);
\draw[ultra thick] (0,1) -- (2,1);
\draw[ultra thick] (3,1) -- (5,1);
\draw[ultra thick] (0,2) -- (5,2);
\draw[ultra thick, densely dashed] (0,3) -- (5,3);
\draw[ultra thick] (1,1) -- (1,2);
\draw[ultra thick] (2,0) -- (2,1);
\draw[ultra thick] (3,0) -- (3,1);
\draw[ultra thick] (4,1) -- (4,2);
\draw[ultra thick, densely dashed] (2.5,2) -- (2.5,3);
\draw[fill=black] (1,2) circle [radius=0.2];
\draw[fill=black] (2,1) circle [radius=0.2];
\draw[fill=black] (3,1) circle [radius=0.2];
\draw[fill=black] (2.5,3) circle [radius=0.2];
\draw[fill=black] (4,2) circle [radius=0.2];
\draw[fill=white] (1,1) circle [radius=0.2];
\draw[fill=white] (2,0) circle [radius=0.2];
\draw[fill=white] (2.5,2) circle [radius=0.2];
\draw[fill=white] (4,1) circle [radius=0.2];
\draw[fill=white] (3,0) circle [radius=0.2];
\node at (2.5,1.5) [] {$g$};
\node at (2.5,1) [] {$\cdots$};
\node at (2.5,0) [] {$\cdots$};
} 
\quad \quad \quad \quad 
\tikz[baseline=2ex]{
\draw[ultra thick] (0,0) -- (2,0);
\draw[ultra thick] (3,0) -- (5,0);
\draw[ultra thick] (0,1) -- (2,1);
\draw[ultra thick] (3,1) -- (5,1);
\draw[ultra thick] (0,2) -- (5,2);
\draw[ultra thick, densely dashed] (0,3) -- (5,3);
\draw[ultra thick] (1,1) -- (1,2);
\draw[ultra thick] (2,0) -- (2,1);
\draw[ultra thick] (3,0) -- (3,1);
\draw[ultra thick] (4,1) -- (4,2);
\draw[ultra thick, densely dashed] (2.5,2) -- (2.5,3);
\draw[fill=white] (1,2) circle [radius=0.2];
\draw[fill=white] (2,1) circle [radius=0.2];
\draw[fill=white] (3,1) circle [radius=0.2];
\draw[fill=white] (2.5,3) circle [radius=0.2];
\draw[fill=white] (4,2) circle [radius=0.2];
\draw[fill=black] (1,1) circle [radius=0.2];
\draw[fill=black] (2,0) circle [radius=0.2];
\draw[fill=black] (2.5,2) circle [radius=0.2];
\draw[fill=black] (4,1) circle [radius=0.2];
\draw[fill=black] (3,0) circle [radius=0.2];
\node at (2.5,1.5) [] {$g$};
\node at (2.5,1) [] {$\cdots$};
\node at (2.5,0) [] {$\cdots$};
}
\]
\vspace{1cm}
\[
\tikz[baseline=2ex]{
\draw[ultra thick, densely dashed] (0,0) -- (2,0);
\draw[ultra thick, densely dashed] (5,0) -- (7,0);
\draw[ultra thick] (0,1) -- (2,1);
\draw[ultra thick] (3,1) -- (4,1);
\draw[ultra thick] (5,1) -- (7,1);
\draw[ultra thick] (0,2) -- (7,2);
\draw[ultra thick, densely dashed] (0,3) -- (7,3);
\draw[ultra thick] (1,1) -- (1,2);
\draw[ultra thick, densely dashed] (2,0) -- (2,1);
\draw[ultra thick, densely dashed] (3,0) -- (3,1);
\draw[ultra thick, densely dashed] (4,0) -- (4,1);
\draw[ultra thick, densely dashed] (5,0) -- (5,1);
\draw[ultra thick] (6,1) -- (6,2);
\draw[ultra thick, densely dashed] (3,2) -- (3,3);
\draw[ultra thick, densely dashed] (4,2) -- (4,3);
\draw[fill=black] (1,2) circle [radius=0.2];
\draw[fill=black] (2,1) circle [radius=0.2];
\draw[fill=black] (3,1) circle [radius=0.2];
\draw[fill=black] (4,0) circle [radius=0.2];
\draw[fill=black] (5,0) circle [radius=0.2];
\draw[fill=black] (6,1) circle [radius=0.2];
\draw[fill=black] (3,3) circle [radius=0.2];
\draw[fill=black] (4,2) circle [radius=0.2];
\draw[fill=white] (1,1) circle [radius=0.2];
\draw[fill=white] (2,0) circle [radius=0.2];
\draw[fill=white] (3,0) circle [radius=0.2];
\draw[fill=white] (4,1) circle [radius=0.2];
\draw[fill=white] (5,1) circle [radius=0.2];
\draw[fill=white] (6,2) circle [radius=0.2];
\draw[fill=white] (3,2) circle [radius=0.2];
\draw[fill=white] (4,3) circle [radius=0.2];
\node at (3.5,1.5) [] {$g$};
\node at (2.5,1) [] {$\cdots$};
\node at (4.5,1) [] {$\cdots$};
\node at (2.5,0) [] {$\cdots$};
\node at (3.5,0) [] {$\cdots$};
\node at (4.5,0) [] {$\cdots$};
}
\]

What remains is to compute $\deg_{X_f}\left(\tilde{p}_{\vec{i}_{(u,v)}}\circ \tilde{\psi}_{\vec{i}_{(u,v)}}\circ\tilde{\chi}_{\vec{i}_{(u,v)}}\right)^*(X_g)$ for each case. To demonstrate, we will do the first case (upper left) and the last case (lower) with all dashed edges present. The cases with some of the dashed edges absent will be left as an exercise for the readers. The proof of the second case (upper right) is completely analogous to the first case (upper left) due to symmetry, and we will omit it as well. 

\begin{itemize}
\item \textbf{The First Case.} Let's draw the right-going zig-zag strands locally to help determine the dominating sets of each neighboring face of $g$, with the assumption that $g$ lies between the $k$th and the $(k+1)$th horizontal lines. Note that the left-going zig-zag strands travel along each horizontal line in this case can we will simply index them as $j_1, \dots, j_4$ from top to bottom without drawing them. For the dominating sets of each neighoring face of $g$, $I$ and $J$ are two fixed $(k-2)$-element sets. Since the faces in the omitted region (if there are any) have no arrows pointing towards $g$ or from $g$, we do not need to compute their dominating sets. Also please keep in mind that $i_4=i_5$ if there is only one black vertex on the lower edge of $g$.
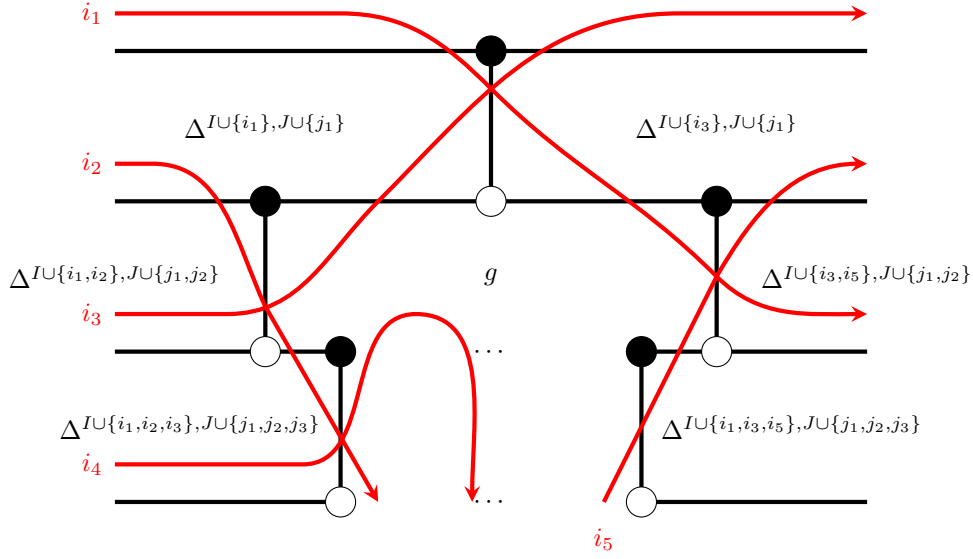
\[
\tikz[baseline=2ex]{
\draw[ultra thick] (0,0) -- (3,0);
\draw[ultra thick] (7,0) -- (10,0);
\draw[ultra thick] (0,2) -- (3,2);
\draw[ultra thick] (7,2) -- (10,2);
\draw[ultra thick] (0,4) -- (10,4);
\draw[ultra thick] (0,6) -- (10,6);
\draw[ultra thick] (2,2) -- (2,4);
\draw[ultra thick] (3,0) -- (3,2);
\draw[ultra thick] (7,0) -- (7,2);
\draw[ultra thick] (8,2) -- (8,4);
\draw[ultra thick] (5,4) -- (5,6);
\draw[fill=black] (2,4) circle [radius=0.2];
\draw[fill=black] (3,2) circle [radius=0.2];
\draw[fill=black] (7,2) circle [radius=0.2];
\draw[fill=black] (5,6) circle [radius=0.2];
\draw[fill=black] (8,4) circle [radius=0.2];
\draw[fill=white] (2,2) circle [radius=0.2];
\draw[fill=white] (3,0) circle [radius=0.2];
\draw[fill=white] (5,4) circle [radius=0.2];
\draw[fill=white] (8,2) circle [radius=0.2];
\draw[fill=white] (7,0) circle [radius=0.2];
\node at (5,3) [] {$g$};
\node at (5,2) [] {$\cdots$};
\node at (5,0) [] {$\cdots$};
\draw[ultra thick, red, ->] (0,0.5) to (2.5,0.5) to [out=0, in=180] (4,2.5) to [out=0, in=90] (4.75,0);
\draw[ultra thick, red, ->] (6.5,0) to (8,3) to [out=60, in=180] (10,4.5);
\draw[ultra thick, red, ->] (0,2.5) to (1.5,2.5) to [out=0, in=-135] (3.5,4) to [out=45, in=180] (7.5, 6.5) to (10,6.5);
\draw[ultra thick, red, ->] (0,4.5) to (0.5,4.5) to [out=0, in=115] (2,2.6) to (3.5,0);
\draw[ultra thick, red, ->] (0,6.5) to (3,6.5) to [out=0, in=135] (5,5.5) to [out=-45, in=135] (8,3) to [out=-45,in=180] (10,2.5);
\node [red] at (-0.3,6.5) [] {$i_1$};
\node [red] at (-0.3,4.5) [] {$i_2$};
\node [red] at (-0.3,2.5) [] {$i_3$};
\node [red] at (-0.3,0.5) [] {$i_4$};
\node [red] at (6.5,-0.5) [] {$i_5$};
\node at (2,5) [] {$\Delta^{I\cup\{i_1\},J\cup\{j_1\}}$};
\node at (8,5) [] {$\Delta^{I\cup\{i_3\},J\cup\{j_1\}}$};
\node at (0,3) [] {$\Delta^{I\cup\{i_1,i_2\},J\cup\{j_1,j_2\}}$};
\node at (10,3) [] {$\Delta^{I\cup\{i_3,i_5\},J\cup\{j_1,j_2\}}$};
\node at (1,1) [] {$\Delta^{I\cup\{i_1,i_2,i_3\},J\cup\{j_1,j_2,j_3\}}$};
\node at (9,1) [] {$\Delta^{I\cup\{i_1,i_3,i_5\},J\cup\{j_1,j_2,j_3\}}$};
} 
\]

Proposition \ref{greedy} tells us that, for each minor $\Delta$ there is a family of pairwise disjoint paths in the perfect orientation of $\Gamma_{\vec{i}_{(u,v)}}$ that gives the term with highest degree for each variable $X_f$; for notation convenience we will denote this term as $[\Delta]$. Now what we need to do is just to take the quotient of these $[\Delta]$ according to the quiver $\vec{i}_{(u,v)}$ near the face $g$. Note that for all the neighboring faces of $g$, their corresponding minors are all divisible by $\left[\Delta^{I,J}\right]$ by Proposition \ref{greedy}. Thus we can substitute $[\Delta]$ by $[\Delta]/\left[\Delta^{I,J}\right]$. To better convey our idea, we will use schematic pictures with integer labeling: each integer tells what the degree of the corresponding face variable in $[\Delta]/\left[\Delta^{I,J}\right]$ is for a face in that region. The shaded face is the face $g$. 
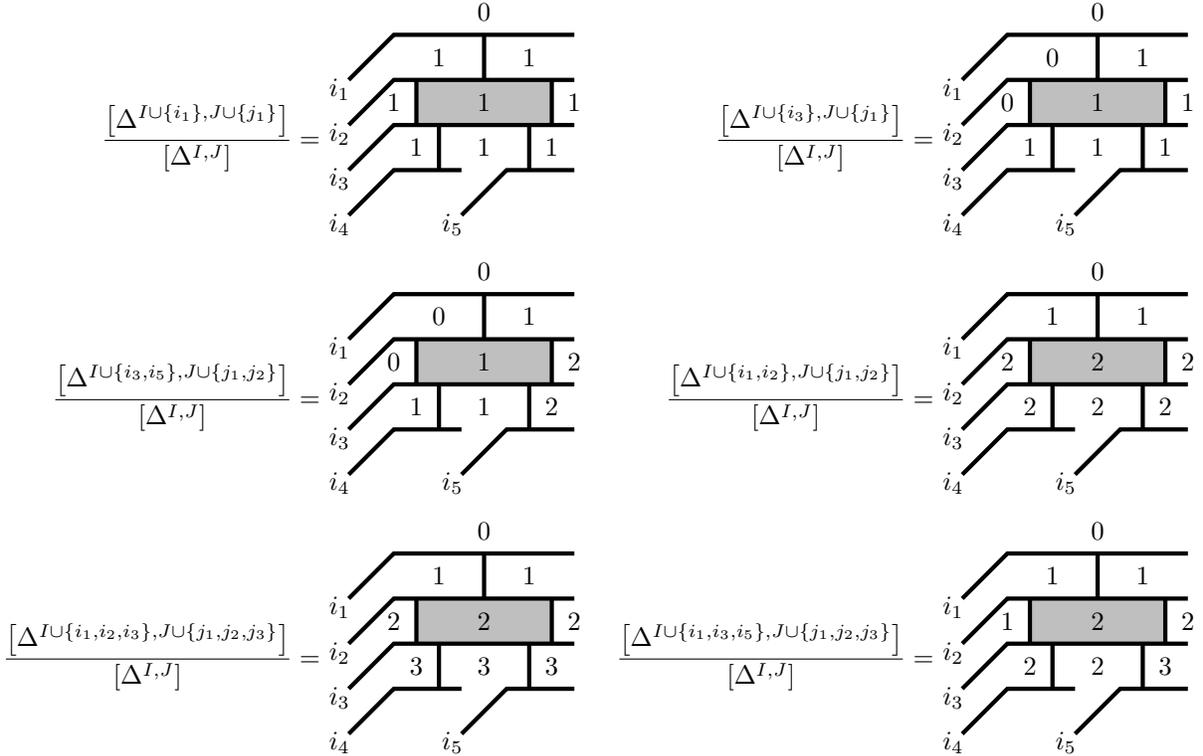
\begin{align*}
\frac{\left[\Delta^{I\cup\{i_1\},J\cup\{j_1\}}\right]}{\left[\Delta^{I,J}\right]}=&\tikz[baseline=6ex, scale=0.6]{
\fill[fill=lightgray] (1.5,2) -- (4.5,2) -- (4.5,3) -- (1.5,3);
\draw[ultra thick] (0,3) -- (1,4) -- (5,4);
\draw[ultra thick] (0,2) -- (1,3) -- (5,3);
\draw[ultra thick] (0,1) -- (1,2) -- (5,2);
\draw[ultra thick] (0,0) -- (1,1) -- (2.5,1);
\draw[ultra thick] (2.5,0) -- (3.5,1) -- (5,1);
\draw[ultra thick] (1.5,2) -- (1.5,3);
\draw[ultra thick] (4.5,2) -- (4.5,3);
\draw[ultra thick] (3,3) -- (3,4);
\draw[ultra thick] (2,1) -- (2,2);
\draw[ultra thick] (4,1) -- (4,2);
\node at (-0.2,-0.2) [] {$i_4$};
\node at (-0.2,0.8) [] {$i_3$};
\node at (-0.2,1.8) [] {$i_2$};
\node at (-0.2,2.8) [] {$i_1$};
\node at (2.3,-0.2) [] {$i_5$};
\node at (3,4.5) [] {$0$};
\node at (2,3.5) [] {$1$};
\node at (4,3.5) [] {$1$};
\node at (1,2.5) [] {$1$};
\node at (3,2.5) [] {$1$};
\node at (5,2.5) [] {$1$};
\node at (1.5,1.5) [] {$1$};
\node at (3,1.5) [] {$1$};
\node at (4.5, 1.5) [] {$1$};
}& 
\frac{\left[\Delta^{I\cup\{i_3\},J\cup\{j_1\}}\right]}{\left[\Delta^{I,J}\right]}=&\tikz[baseline=6ex, scale=0.6]{
\fill[fill=lightgray] (1.5,2) -- (4.5,2) -- (4.5,3) -- (1.5,3);
\draw[ultra thick] (0,3) -- (1,4) -- (5,4);
\draw[ultra thick] (0,2) -- (1,3) -- (5,3);
\draw[ultra thick] (0,1) -- (1,2) -- (5,2);
\draw[ultra thick] (0,0) -- (1,1) -- (2.5,1);
\draw[ultra thick] (2.5,0) -- (3.5,1) -- (5,1);
\draw[ultra thick] (1.5,2) -- (1.5,3);
\draw[ultra thick] (4.5,2) -- (4.5,3);
\draw[ultra thick] (3,3) -- (3,4);
\draw[ultra thick] (2,1) -- (2,2);
\draw[ultra thick] (4,1) -- (4,2);
\node at (-0.2,-0.2) [] {$i_4$};
\node at (-0.2,0.8) [] {$i_3$};
\node at (-0.2,1.8) [] {$i_2$};
\node at (-0.2,2.8) [] {$i_1$};
\node at (2.3,-0.2) [] {$i_5$};
\node at (3,4.5) [] {$0$};
\node at (2,3.5) [] {$0$};
\node at (4,3.5) [] {$1$};
\node at (1,2.5) [] {$0$};
\node at (3,2.5) [] {$1$};
\node at (5,2.5) [] {$1$};
\node at (1.5,1.5) [] {$1$};
\node at (3,1.5) [] {$1$};
\node at (4.5, 1.5) [] {$1$};
}
\\
\frac{\left[\Delta^{I\cup\{i_3,i_5\},J\cup\{j_1,j_2\}}\right]}{\left[\Delta^{I,J}\right]}=&\tikz[baseline=6ex, scale=0.6]{
\fill[fill=lightgray] (1.5,2) -- (4.5,2) -- (4.5,3) -- (1.5,3);
\draw[ultra thick] (0,3) -- (1,4) -- (5,4);
\draw[ultra thick] (0,2) -- (1,3) -- (5,3);
\draw[ultra thick] (0,1) -- (1,2) -- (5,2);
\draw[ultra thick] (0,0) -- (1,1) -- (2.5,1);
\draw[ultra thick] (2.5,0) -- (3.5,1) -- (5,1);
\draw[ultra thick] (1.5,2) -- (1.5,3);
\draw[ultra thick] (4.5,2) -- (4.5,3);
\draw[ultra thick] (3,3) -- (3,4);
\draw[ultra thick] (2,1) -- (2,2);
\draw[ultra thick] (4,1) -- (4,2);
\node at (-0.2,-0.2) [] {$i_4$};
\node at (-0.2,0.8) [] {$i_3$};
\node at (-0.2,1.8) [] {$i_2$};
\node at (-0.2,2.8) [] {$i_1$};
\node at (2.3,-0.2) [] {$i_5$};
\node at (3,4.5) [] {$0$};
\node at (2,3.5) [] {$0$};
\node at (4,3.5) [] {$1$};
\node at (1,2.5) [] {$0$};
\node at (3,2.5) [] {$1$};
\node at (5,2.5) [] {$2$};
\node at (1.5,1.5) [] {$1$};
\node at (3,1.5) [] {$1$};
\node at (4.5, 1.5) [] {$2$};
}& 
\frac{\left[\Delta^{I\cup\{i_1,i_2\},J\cup\{j_1,j_2\}}\right]}{\left[\Delta^{I,J}\right]}=&\tikz[baseline=6ex, scale=0.6]{
\fill[fill=lightgray] (1.5,2) -- (4.5,2) -- (4.5,3) -- (1.5,3);
\draw[ultra thick] (0,3) -- (1,4) -- (5,4);
\draw[ultra thick] (0,2) -- (1,3) -- (5,3);
\draw[ultra thick] (0,1) -- (1,2) -- (5,2);
\draw[ultra thick] (0,0) -- (1,1) -- (2.5,1);
\draw[ultra thick] (2.5,0) -- (3.5,1) -- (5,1);
\draw[ultra thick] (1.5,2) -- (1.5,3);
\draw[ultra thick] (4.5,2) -- (4.5,3);
\draw[ultra thick] (3,3) -- (3,4);
\draw[ultra thick] (2,1) -- (2,2);
\draw[ultra thick] (4,1) -- (4,2);
\node at (-0.2,-0.2) [] {$i_4$};
\node at (-0.2,0.8) [] {$i_3$};
\node at (-0.2,1.8) [] {$i_2$};
\node at (-0.2,2.8) [] {$i_1$};
\node at (2.3,-0.2) [] {$i_5$};
\node at (3,4.5) [] {$0$};
\node at (2,3.5) [] {$1$};
\node at (4,3.5) [] {$1$};
\node at (1,2.5) [] {$2$};
\node at (3,2.5) [] {$2$};
\node at (5,2.5) [] {$2$};
\node at (1.5,1.5) [] {$2$};
\node at (3,1.5) [] {$2$};
\node at (4.5, 1.5) [] {$2$};
}
\\
\frac{\left[\Delta^{I\cup\{i_1,i_2,i_3\},J\cup\{j_1,j_2,j_3\}}\right]}{\left[\Delta^{I,J}\right]}=&\tikz[baseline=6ex, scale=0.6]{
\fill[fill=lightgray] (1.5,2) -- (4.5,2) -- (4.5,3) -- (1.5,3);
\draw[ultra thick] (0,3) -- (1,4) -- (5,4);
\draw[ultra thick] (0,2) -- (1,3) -- (5,3);
\draw[ultra thick] (0,1) -- (1,2) -- (5,2);
\draw[ultra thick] (0,0) -- (1,1) -- (2.5,1);
\draw[ultra thick] (2.5,0) -- (3.5,1) -- (5,1);
\draw[ultra thick] (1.5,2) -- (1.5,3);
\draw[ultra thick] (4.5,2) -- (4.5,3);
\draw[ultra thick] (3,3) -- (3,4);
\draw[ultra thick] (2,1) -- (2,2);
\draw[ultra thick] (4,1) -- (4,2);
\node at (-0.2,-0.2) [] {$i_4$};
\node at (-0.2,0.8) [] {$i_3$};
\node at (-0.2,1.8) [] {$i_2$};
\node at (-0.2,2.8) [] {$i_1$};
\node at (2.3,-0.2) [] {$i_5$};
\node at (3,4.5) [] {$0$};
\node at (2,3.5) [] {$1$};
\node at (4,3.5) [] {$1$};
\node at (1,2.5) [] {$2$};
\node at (3,2.5) [] {$2$};
\node at (5,2.5) [] {$2$};
\node at (1.5,1.5) [] {$3$};
\node at (3,1.5) [] {$3$};
\node at (4.5, 1.5) [] {$3$};
}& 
\frac{\left[\Delta^{I\cup\{i_1,i_3,i_5\},J\cup\{j_1,j_2,j_3\}}\right]}{\left[\Delta^{I,J}\right]}=&\tikz[baseline=6ex, scale=0.6]{
\fill[fill=lightgray] (1.5,2) -- (4.5,2) -- (4.5,3) -- (1.5,3);
\draw[ultra thick] (0,3) -- (1,4) -- (5,4);
\draw[ultra thick] (0,2) -- (1,3) -- (5,3);
\draw[ultra thick] (0,1) -- (1,2) -- (5,2);
\draw[ultra thick] (0,0) -- (1,1) -- (2.5,1);
\draw[ultra thick] (2.5,0) -- (3.5,1) -- (5,1);
\draw[ultra thick] (1.5,2) -- (1.5,3);
\draw[ultra thick] (4.5,2) -- (4.5,3);
\draw[ultra thick] (3,3) -- (3,4);
\draw[ultra thick] (2,1) -- (2,2);
\draw[ultra thick] (4,1) -- (4,2);
\node at (-0.2,-0.2) [] {$i_4$};
\node at (-0.2,0.8) [] {$i_3$};
\node at (-0.2,1.8) [] {$i_2$};
\node at (-0.2,2.8) [] {$i_1$};
\node at (2.3,-0.2) [] {$i_5$};
\node at (3,4.5) [] {$0$};
\node at (2,3.5) [] {$1$};
\node at (4,3.5) [] {$1$};
\node at (1,2.5) [] {$1$};
\node at (3,2.5) [] {$2$};
\node at (5,2.5) [] {$2$};
\node at (1.5,1.5) [] {$2$};
\node at (3,1.5) [] {$2$};
\node at (4.5, 1.5) [] {$3$};
}
\end{align*}
Now with these six schematic pictures, it is not hard to compute
\[
\frac{\left[\Delta^{I\cup\{i_1\},J\cup\{j_1\}}\right]\left[\Delta^{I\cup\{i_3,i_5\},J\cup\{j_1,j_2\}}\right]\left[\Delta^{I\cup\{i_1,i_2,i_3\},J\cup\{j_1,j_2,j_3\}}\right]}{\left[\Delta^{I\cup\{i_3\},J\cup\{j_1\}}\right]\left[\Delta^{I\cup\{i_1,i_2\},J\cup\{j_1,j_2\}}\right]\left[\Delta^{I\cup\{i_1,i_3,i_5\},J\cup\{j_1,j_2,j_3\}}\right]}=\tikz[baseline=6ex, scale=0.6]{
\fill[fill=lightgray] (1.5,2) -- (4.5,2) -- (4.5,3) -- (1.5,3);
\draw[ultra thick] (0,3) -- (1,4) -- (5,4);
\draw[ultra thick] (0,2) -- (1,3) -- (5,3);
\draw[ultra thick] (0,1) -- (1,2) -- (5,2);
\draw[ultra thick] (0,0) -- (1,1) -- (2.5,1);
\draw[ultra thick] (2.5,0) -- (3.5,1) -- (5,1);
\draw[ultra thick] (1.5,2) -- (1.5,3);
\draw[ultra thick] (4.5,2) -- (4.5,3);
\draw[ultra thick] (3,3) -- (3,4);
\draw[ultra thick] (2,1) -- (2,2);
\draw[ultra thick] (4,1) -- (4,2);
\node at (-0.2,-0.2) [] {$i_4$};
\node at (-0.2,0.8) [] {$i_3$};
\node at (-0.2,1.8) [] {$i_2$};
\node at (-0.2,2.8) [] {$i_1$};
\node at (2.3,-0.2) [] {$i_5$};
\node at (3,4.5) [] {$0$};
\node at (2,3.5) [] {$0$};
\node at (4,3.5) [] {$0$};
\node at (1,2.5) [] {$0$};
\node at (3,2.5) [] {$-1$};
\node at (5,2.5) [] {$0$};
\node at (1.5,1.5) [] {$0$};
\node at (3,1.5) [] {$0$};
\node at (4.5, 1.5) [] {$0$};
}=X_g^{-1},
\]
from which we can conclude that
\[
\deg_{X_f}\left(\tilde{p}_{\vec{i}_{(u,v)}}\circ \tilde{\psi}_{\vec{i}_{(u,v)}}\circ\tilde{\chi}_{\vec{i}_{(u,v)}}\right)^*(X_g)=-\delta_{fg}.
\]

\item \textbf{The Last Case.} The last case is slightly harder computation-wise; we will just follow the same guidelines we had in the first case. Let's first draw the zig-zag strands (now we have both left-going and right-going ones.
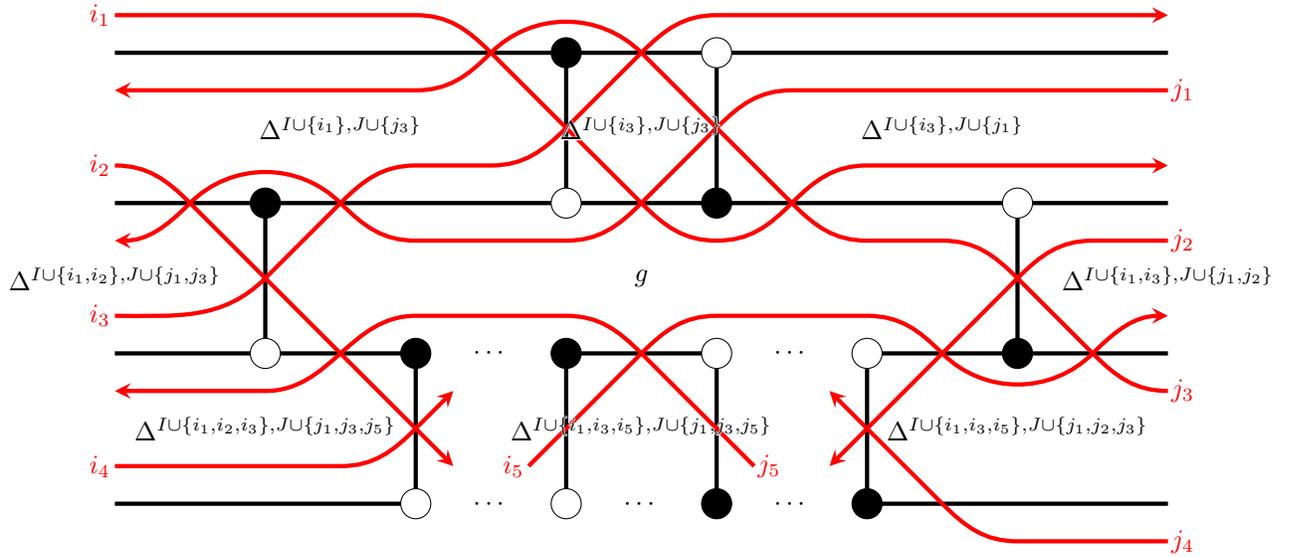
\[
\tikz[baseline=2ex]{
\draw[ultra thick] (0,0) -- (4,0);
\draw[ultra thick] (10,0) -- (14,0);
\draw[ultra thick] (0,2) -- (4,2);
\draw[ultra thick] (6,2) -- (8,2);
\draw[ultra thick] (10,2) -- (14,2);
\draw[ultra thick] (0,4) -- (14,4);
\draw[ultra thick] (0,6) -- (14,6);
\draw[ultra thick] (2,2) -- (2,4);
\draw[ultra thick] (4,0) -- (4,2);
\draw[ultra thick] (6,0) -- (6,2);
\draw[ultra thick] (8,0) -- (8,2);
\draw[ultra thick] (10,0) -- (10,2);
\draw[ultra thick] (12,2) -- (12,4);
\draw[ultra thick] (6,4) -- (6,6);
\draw[ultra thick] (8,4) -- (8,6);
\draw[fill=black] (2,4) circle [radius=0.2];
\draw[fill=black] (4,2) circle [radius=0.2];
\draw[fill=black] (6,2) circle [radius=0.2];
\draw[fill=black] (8,0) circle [radius=0.2];
\draw[fill=black] (10,0) circle [radius=0.2];
\draw[fill=black] (12,2) circle [radius=0.2];
\draw[fill=black] (6,6) circle [radius=0.2];
\draw[fill=black] (8,4) circle [radius=0.2];
\draw[fill=white] (2,2) circle [radius=0.2];
\draw[fill=white] (4,0) circle [radius=0.2];
\draw[fill=white] (6,0) circle [radius=0.2];
\draw[fill=white] (8,2) circle [radius=0.2];
\draw[fill=white] (10,2) circle [radius=0.2];
\draw[fill=white] (12,4) circle [radius=0.2];
\draw[fill=white] (6,4) circle [radius=0.2];
\draw[fill=white] (8,6) circle [radius=0.2];
\node at (7,3) [] {$g$};
\node at (5,2) [] {$\cdots$};
\node at (9,2) [] {$\cdots$};
\node at (5,0) [] {$\cdots$};
\node at (7,0) [] {$\cdots$};
\node at (9,0) [] {$\cdots$};
\draw[ultra thick, ->, red] (0,6.5) to (4,6.5) to [out=0, in=135] (5,6) to [out=-45,in=135] (7,4) to [out=-45,in=180] (8,3.5) to [out=0,in=-135] (9,4) to [out=45, in=180] (10,4.5) to (14,4.5);
\draw[ultra thick, ->, red] (0,4.5) to [out=0,in=135] (1,4) to [out=-45,in=135] (2,3) to (4.5,0.5);
\draw[ultra thick, ->, red] (0,2.5) to [out=0,in=-135] (2,3) to [out=45,in=-135] (3,4) to [out=45, in=180] (4,4.5) to (5,4.5) to [out=0,in=-135] (6,5) to [out=45,in=-135] (7,6) to [out=45,in=180] (8,6.5) to (14,6.5);
\draw[ultra thick, ->, red] (0,0.5) to (3,0.5) to [out=0,in=-135] (4,1) to (4.5,1.5);
\draw[ultra thick, ->, red] (5.5,0.5) to (7,2) to [out=45,in=180] (8,2.5) to (10,2.5) to [out=0, in=135] (11,2) to [out=-45,in=-135] (13,2) to [out=45,in=180] (14,2.5);
\draw[ultra thick, ->, red] (14,5.5) to (9,5.5) to [out=180,in=45] (8,5) to [out=-135,in=45] (7,4) to [out=-135,in=0] (6,3.5) to (4,3.5) to [out=180,in=-45] (3,4) to [out=135,in=45] (1,4) to [out=-135,in=0] (0,3.5);
\draw[ultra thick, ->, red] (14,3.5) to (13,3.5) to [out=180,in=45] (12,3) to (9.5,0.5);
\draw[ultra thick, ->, red] (14,1.5) to [out=180,in=-45] (13,2) to (12,3) to [out=135,in=0] (11,3.5) to (10,3.5) to [out=180, in=-45] (9,4) to (7,6) to [out=135,in=45] (5,6) to [out=-135,in=0] (4,5.5) to (0,5.5);
\draw[ultra thick, ->, red] (14,-0.5) to (12,-0.5) to [out=180, in=-45] (11,0) to (9.5,1.5);
\draw[ultra thick, ->, red] (8.5,0.5) to (7,2) to [out=135,in=0] (6,2.5) to (4,2.5) to [out=180,in=0] (2,1.5) to (0,1.5);
\node [red] at (-0.2,6.5) [] {$i_1$};
\node [red] at (-0.2,4.5) [] {$i_2$};
\node [red] at (-0.2,2.5) [] {$i_3$};
\node [red] at (-0.2,0.5) [] {$i_4$};
\node [red] at (5.3,0.5) [] {$i_5$};
\node [red] at (14.2,5.5) [] {$j_1$};
\node [red] at (14.2,3.5) [] {$j_2$};
\node [red] at (14.2,1.5) [] {$j_3$};
\node [red] at (14.2,-0.5) [] {$j_4$};
\node [red] at (8.7,0.5) [] {$j_5$};
\node at (3,5) [] {$\Delta^{I\cup\{i_1\},J\cup\{j_3\}}$};
\node at (7,5) [] {\contour{white}{$\Delta^{I\cup\{i_3\},J\cup\{j_3\}}$}};
\node at (11,5) [] {$\Delta^{I\cup\{i_3\},J\cup\{j_1\}}$};
\node at (0,3) [] {$\Delta^{I\cup\{i_1,i_2\},J\cup\{j_1,j_3\}}$};
\node at (14,3) [] {$\Delta^{I\cup\{i_1,i_3\},J\cup\{j_1,j_2\}}$};
\node at (2,1) [] {$\Delta^{I\cup\{i_1,i_2,i_3\},J\cup\{j_1,j_3,j_5\}}$};
\node at (7,1) [] {\contour{white}{$\Delta^{I\cup\{i_1,i_3,i_5\},J\cup\{j_1,j_3,j_5\}}$}};
\node at (12,1) [] {$\Delta^{I\cup\{i_1,i_3,i_5\},J\cup\{j_1,j_2,j_3\}}$};
}
\]
Then we need to compute $[\Delta]/\left[\Delta^{I,J}\right]$ for each neighboring face of $g$.
\begin{align*}
\frac{\left[\Delta^{I\cup\{i_1\},J\cup\{j_3\}}\right]}{\left[\Delta^{I,J}\right]}=&\tikz[baseline=5ex,scale=0.4]{
\fill[fill=lightgray] (1.5,2) -- (6.5,2) -- (6.5,3) -- (1.5,3);
\draw[ultra thick] (0,3) -- (1,4) -- (7,4) -- (8,3);
\draw[ultra thick] (0,2) -- (1,3) -- (7,3) -- (8,2);
\draw[ultra thick] (0,1) -- (1,2) -- (7,2) -- (8,1);
\draw[ultra thick] (0,0) -- (1,1) -- (2.5,1);
\draw[ultra thick] (2,0) -- (3,1) -- (3.75,1);
\draw[ultra thick] (4.25,1) -- (5,1) -- (6,0);
\draw[ultra thick] (5.5,1) -- (7,1) -- (8,0);
\draw[ultra thick] (3,3) -- (3,4);
\draw[ultra thick] (5,3) -- (5,4);
\draw[ultra thick] (1.5,2) -- (1.5,3);
\draw[ultra thick] (6.5,2) -- (6.5,3);
\draw[ultra thick] (2,1) -- (2,2);
\draw[ultra thick] (3.5,1) -- (3.5,2);
\draw[ultra thick] (4.5,1) -- (4.5,2);
\draw[ultra thick] (6,1) -- (6,2);
\node at (-0.2,2.8) [] {$i_1$};
\node at (-0.2,1.8) [] {$i_2$};
\node at (-0.2,0.8) [] {$i_3$};
\node at (-0.2,-0.2) [] {$i_4$};
\node at (1.8,-0.2) [] {$i_5$};
\node at (8.2,2.8) [] {$j_1$};
\node at (8.2,1.8) [] {$j_2$};
\node at (8.2,0.8) [] {$j_3$};
\node at (8.2,-0.2) [] {$j_4$};
\node at (6.2,-0.2) [] {$j_5$};
\node at (4,4.5) [] {$0$};
\node at (2,3.5) [] {$1$};
\node at (4,3.5) [] {$1$};
\node at (6,3.5) [] {$0$};
\node at (1,2.5) [] {$1$};
\node at (4,2.5) [] {$1$};
\node at (7,2.5) [] {$0$};
\node at (1.5,1.5) [] {$1$};
\node at (2.75,1.5) [] {$1$};
\node at (4,1.5) [] {$1$};
\node at (5.25,1.5) [] {$1$};
\node at (6.5,1.5) [] {$1$};
} & \frac{\left[\Delta^{I\cup\{i_3\},J\cup\{j_3\}}\right]}{\left[\Delta^{I,J}\right]}=&\tikz[baseline=5ex,scale=0.4]{
\fill[fill=lightgray] (1.5,2) -- (6.5,2) -- (6.5,3) -- (1.5,3);
\draw[ultra thick] (0,3) -- (1,4) -- (7,4) -- (8,3);
\draw[ultra thick] (0,2) -- (1,3) -- (7,3) -- (8,2);
\draw[ultra thick] (0,1) -- (1,2) -- (7,2) -- (8,1);
\draw[ultra thick] (0,0) -- (1,1) -- (2.5,1);
\draw[ultra thick] (2,0) -- (3,1) -- (3.75,1);
\draw[ultra thick] (4.25,1) -- (5,1) -- (6,0);
\draw[ultra thick] (5.5,1) -- (7,1) -- (8,0);
\draw[ultra thick] (3,3) -- (3,4);
\draw[ultra thick] (5,3) -- (5,4);
\draw[ultra thick] (1.5,2) -- (1.5,3);
\draw[ultra thick] (6.5,2) -- (6.5,3);
\draw[ultra thick] (2,1) -- (2,2);
\draw[ultra thick] (3.5,1) -- (3.5,2);
\draw[ultra thick] (4.5,1) -- (4.5,2);
\draw[ultra thick] (6,1) -- (6,2);
\node at (-0.2,2.8) [] {$i_1$};
\node at (-0.2,1.8) [] {$i_2$};
\node at (-0.2,0.8) [] {$i_3$};
\node at (-0.2,-0.2) [] {$i_4$};
\node at (1.8,-0.2) [] {$i_5$};
\node at (8.2,2.8) [] {$j_1$};
\node at (8.2,1.8) [] {$j_2$};
\node at (8.2,0.8) [] {$j_3$};
\node at (8.2,-0.2) [] {$j_4$};
\node at (6.2,-0.2) [] {$j_5$};
\node at (4,4.5) [] {$0$};
\node at (2,3.5) [] {$0$};
\node at (4,3.5) [] {$1$};
\node at (6,3.5) [] {$0$};
\node at (1,2.5) [] {$0$};
\node at (4,2.5) [] {$1$};
\node at (7,2.5) [] {$0$};
\node at (1.5,1.5) [] {$1$};
\node at (2.75,1.5) [] {$1$};
\node at (4,1.5) [] {$1$};
\node at (5.25,1.5) [] {$1$};
\node at (6.5,1.5) [] {$1$};
}
\\
\frac{\left[\Delta^{I\cup\{i_3\},J\cup\{j_1\}}\right]}{\left[\Delta^{I,J}\right]}=&\tikz[baseline=5ex,scale=0.4]{
\fill[fill=lightgray] (1.5,2) -- (6.5,2) -- (6.5,3) -- (1.5,3);
\draw[ultra thick] (0,3) -- (1,4) -- (7,4) -- (8,3);
\draw[ultra thick] (0,2) -- (1,3) -- (7,3) -- (8,2);
\draw[ultra thick] (0,1) -- (1,2) -- (7,2) -- (8,1);
\draw[ultra thick] (0,0) -- (1,1) -- (2.5,1);
\draw[ultra thick] (2,0) -- (3,1) -- (3.75,1);
\draw[ultra thick] (4.25,1) -- (5,1) -- (6,0);
\draw[ultra thick] (5.5,1) -- (7,1) -- (8,0);
\draw[ultra thick] (3,3) -- (3,4);
\draw[ultra thick] (5,3) -- (5,4);
\draw[ultra thick] (1.5,2) -- (1.5,3);
\draw[ultra thick] (6.5,2) -- (6.5,3);
\draw[ultra thick] (2,1) -- (2,2);
\draw[ultra thick] (3.5,1) -- (3.5,2);
\draw[ultra thick] (4.5,1) -- (4.5,2);
\draw[ultra thick] (6,1) -- (6,2);
\node at (-0.2,2.8) [] {$i_1$};
\node at (-0.2,1.8) [] {$i_2$};
\node at (-0.2,0.8) [] {$i_3$};
\node at (-0.2,-0.2) [] {$i_4$};
\node at (1.8,-0.2) [] {$i_5$};
\node at (8.2,2.8) [] {$j_1$};
\node at (8.2,1.8) [] {$j_2$};
\node at (8.2,0.8) [] {$j_3$};
\node at (8.2,-0.2) [] {$j_4$};
\node at (6.2,-0.2) [] {$j_5$};
\node at (4,4.5) [] {$0$};
\node at (2,3.5) [] {$0$};
\node at (4,3.5) [] {$1$};
\node at (6,3.5) [] {$1$};
\node at (1,2.5) [] {$0$};
\node at (4,2.5) [] {$1$};
\node at (7,2.5) [] {$1$};
\node at (1.5,1.5) [] {$1$};
\node at (2.75,1.5) [] {$1$};
\node at (4,1.5) [] {$1$};
\node at (5.25,1.5) [] {$1$};
\node at (6.5,1.5) [] {$1$};
} & \frac{\left[\Delta^{I\cup\{i_1,i_2\},J\cup\{j_1,j_3\}}\right]}{\left[\Delta^{I,J}\right]}=&\tikz[baseline=5ex,scale=0.4]{
\fill[fill=lightgray] (1.5,2) -- (6.5,2) -- (6.5,3) -- (1.5,3);
\draw[ultra thick] (0,3) -- (1,4) -- (7,4) -- (8,3);
\draw[ultra thick] (0,2) -- (1,3) -- (7,3) -- (8,2);
\draw[ultra thick] (0,1) -- (1,2) -- (7,2) -- (8,1);
\draw[ultra thick] (0,0) -- (1,1) -- (2.5,1);
\draw[ultra thick] (2,0) -- (3,1) -- (3.75,1);
\draw[ultra thick] (4.25,1) -- (5,1) -- (6,0);
\draw[ultra thick] (5.5,1) -- (7,1) -- (8,0);
\draw[ultra thick] (3,3) -- (3,4);
\draw[ultra thick] (5,3) -- (5,4);
\draw[ultra thick] (1.5,2) -- (1.5,3);
\draw[ultra thick] (6.5,2) -- (6.5,3);
\draw[ultra thick] (2,1) -- (2,2);
\draw[ultra thick] (3.5,1) -- (3.5,2);
\draw[ultra thick] (4.5,1) -- (4.5,2);
\draw[ultra thick] (6,1) -- (6,2);
\node at (-0.2,2.8) [] {$i_1$};
\node at (-0.2,1.8) [] {$i_2$};
\node at (-0.2,0.8) [] {$i_3$};
\node at (-0.2,-0.2) [] {$i_4$};
\node at (1.8,-0.2) [] {$i_5$};
\node at (8.2,2.8) [] {$j_1$};
\node at (8.2,1.8) [] {$j_2$};
\node at (8.2,0.8) [] {$j_3$};
\node at (8.2,-0.2) [] {$j_4$};
\node at (6.2,-0.2) [] {$j_5$};
\node at (4,4.5) [] {$0$};
\node at (2,3.5) [] {$1$};
\node at (4,3.5) [] {$1$};
\node at (6,3.5) [] {$1$};
\node at (1,2.5) [] {$2$};
\node at (4,2.5) [] {$2$};
\node at (7,2.5) [] {$1$};
\node at (1.5,1.5) [] {$2$};
\node at (2.75,1.5) [] {$2$};
\node at (4,1.5) [] {$2$};
\node at (5.25,1.5) [] {$2$};
\node at (6.5,1.5) [] {$2$};
}
\\
\frac{\left[\Delta^{I\cup\{i_1,i_2,i_3\},J\cup\{j_1,j_3,j_5\}}\right]}{\left[\Delta^{I,J}\right]}=&\tikz[baseline=5ex,scale=0.4]{
\fill[fill=lightgray] (1.5,2) -- (6.5,2) -- (6.5,3) -- (1.5,3);
\draw[ultra thick] (0,3) -- (1,4) -- (7,4) -- (8,3);
\draw[ultra thick] (0,2) -- (1,3) -- (7,3) -- (8,2);
\draw[ultra thick] (0,1) -- (1,2) -- (7,2) -- (8,1);
\draw[ultra thick] (0,0) -- (1,1) -- (2.5,1);
\draw[ultra thick] (2,0) -- (3,1) -- (3.75,1);
\draw[ultra thick] (4.25,1) -- (5,1) -- (6,0);
\draw[ultra thick] (5.5,1) -- (7,1) -- (8,0);
\draw[ultra thick] (3,3) -- (3,4);
\draw[ultra thick] (5,3) -- (5,4);
\draw[ultra thick] (1.5,2) -- (1.5,3);
\draw[ultra thick] (6.5,2) -- (6.5,3);
\draw[ultra thick] (2,1) -- (2,2);
\draw[ultra thick] (3.5,1) -- (3.5,2);
\draw[ultra thick] (4.5,1) -- (4.5,2);
\draw[ultra thick] (6,1) -- (6,2);
\node at (-0.2,2.8) [] {$i_1$};
\node at (-0.2,1.8) [] {$i_2$};
\node at (-0.2,0.8) [] {$i_3$};
\node at (-0.2,-0.2) [] {$i_4$};
\node at (1.8,-0.2) [] {$i_5$};
\node at (8.2,2.8) [] {$j_1$};
\node at (8.2,1.8) [] {$j_2$};
\node at (8.2,0.8) [] {$j_3$};
\node at (8.2,-0.2) [] {$j_4$};
\node at (6.2,-0.2) [] {$j_5$};
\node at (4,4.5) [] {$0$};
\node at (2,3.5) [] {$1$};
\node at (4,3.5) [] {$1$};
\node at (6,3.5) [] {$1$};
\node at (1,2.5) [] {$2$};
\node at (4,2.5) [] {$2$};
\node at (7,2.5) [] {$1$};
\node at (1.5,1.5) [] {$3$};
\node at (2.75,1.5) [] {$3$};
\node at (4,1.5) [] {$3$};
\node at (5.25,1.5) [] {$2$};
\node at (6.5,1.5) [] {$2$};
} & \frac{\left[\Delta^{I\cup\{i_1,i_3\},J\cup\{j_1,j_2\}}\right]}{\left[\Delta^{I,J}\right]}=&\tikz[baseline=5ex,scale=0.4]{
\fill[fill=lightgray] (1.5,2) -- (6.5,2) -- (6.5,3) -- (1.5,3);
\draw[ultra thick] (0,3) -- (1,4) -- (7,4) -- (8,3);
\draw[ultra thick] (0,2) -- (1,3) -- (7,3) -- (8,2);
\draw[ultra thick] (0,1) -- (1,2) -- (7,2) -- (8,1);
\draw[ultra thick] (0,0) -- (1,1) -- (2.5,1);
\draw[ultra thick] (2,0) -- (3,1) -- (3.75,1);
\draw[ultra thick] (4.25,1) -- (5,1) -- (6,0);
\draw[ultra thick] (5.5,1) -- (7,1) -- (8,0);
\draw[ultra thick] (3,3) -- (3,4);
\draw[ultra thick] (5,3) -- (5,4);
\draw[ultra thick] (1.5,2) -- (1.5,3);
\draw[ultra thick] (6.5,2) -- (6.5,3);
\draw[ultra thick] (2,1) -- (2,2);
\draw[ultra thick] (3.5,1) -- (3.5,2);
\draw[ultra thick] (4.5,1) -- (4.5,2);
\draw[ultra thick] (6,1) -- (6,2);
\node at (-0.2,2.8) [] {$i_1$};
\node at (-0.2,1.8) [] {$i_2$};
\node at (-0.2,0.8) [] {$i_3$};
\node at (-0.2,-0.2) [] {$i_4$};
\node at (1.8,-0.2) [] {$i_5$};
\node at (8.2,2.8) [] {$j_1$};
\node at (8.2,1.8) [] {$j_2$};
\node at (8.2,0.8) [] {$j_3$};
\node at (8.2,-0.2) [] {$j_4$};
\node at (6.2,-0.2) [] {$j_5$};
\node at (4,4.5) [] {$0$};
\node at (2,3.5) [] {$1$};
\node at (4,3.5) [] {$1$};
\node at (6,3.5) [] {$1$};
\node at (1,2.5) [] {$1$};
\node at (4,2.5) [] {$2$};
\node at (7,2.5) [] {$2$};
\node at (1.5,1.5) [] {$2$};
\node at (2.75,1.5) [] {$2$};
\node at (4,1.5) [] {$2$};
\node at (5.25,1.5) [] {$2$};
\node at (6.5,1.5) [] {$2$};
}
\\
\frac{\left[\Delta^{I\cup\{i_1,i_3,i_5\},J\cup\{j_1,j_2,j_3\}}\right]}{\left[\Delta^{I,J}\right]}=&\tikz[baseline=5ex,scale=0.4]{
\fill[fill=lightgray] (1.5,2) -- (6.5,2) -- (6.5,3) -- (1.5,3);
\draw[ultra thick] (0,3) -- (1,4) -- (7,4) -- (8,3);
\draw[ultra thick] (0,2) -- (1,3) -- (7,3) -- (8,2);
\draw[ultra thick] (0,1) -- (1,2) -- (7,2) -- (8,1);
\draw[ultra thick] (0,0) -- (1,1) -- (2.5,1);
\draw[ultra thick] (2,0) -- (3,1) -- (3.75,1);
\draw[ultra thick] (4.25,1) -- (5,1) -- (6,0);
\draw[ultra thick] (5.5,1) -- (7,1) -- (8,0);
\draw[ultra thick] (3,3) -- (3,4);
\draw[ultra thick] (5,3) -- (5,4);
\draw[ultra thick] (1.5,2) -- (1.5,3);
\draw[ultra thick] (6.5,2) -- (6.5,3);
\draw[ultra thick] (2,1) -- (2,2);
\draw[ultra thick] (3.5,1) -- (3.5,2);
\draw[ultra thick] (4.5,1) -- (4.5,2);
\draw[ultra thick] (6,1) -- (6,2);
\node at (-0.2,2.8) [] {$i_1$};
\node at (-0.2,1.8) [] {$i_2$};
\node at (-0.2,0.8) [] {$i_3$};
\node at (-0.2,-0.2) [] {$i_4$};
\node at (1.8,-0.2) [] {$i_5$};
\node at (8.2,2.8) [] {$j_1$};
\node at (8.2,1.8) [] {$j_2$};
\node at (8.2,0.8) [] {$j_3$};
\node at (8.2,-0.2) [] {$j_4$};
\node at (6.2,-0.2) [] {$j_5$};
\node at (4,4.5) [] {$0$};
\node at (2,3.5) [] {$1$};
\node at (4,3.5) [] {$1$};
\node at (6,3.5) [] {$1$};
\node at (1,2.5) [] {$1$};
\node at (4,2.5) [] {$2$};
\node at (7,2.5) [] {$2$};
\node at (1.5,1.5) [] {$2$};
\node at (2.75,1.5) [] {$2$};
\node at (4,1.5) [] {$3$};
\node at (5.25,1.5) [] {$3$};
\node at (6.5,1.5) [] {$3$};
} & \frac{\left[\Delta^{I\cup\{i_1,i_3,i_5\},J\cup\{j_1,j_3,j_5\}}\right]}{\left[\Delta^{I,J}\right]}=&\tikz[baseline=5ex,scale=0.4]{
\fill[fill=lightgray] (1.5,2) -- (6.5,2) -- (6.5,3) -- (1.5,3);
\draw[ultra thick] (0,3) -- (1,4) -- (7,4) -- (8,3);
\draw[ultra thick] (0,2) -- (1,3) -- (7,3) -- (8,2);
\draw[ultra thick] (0,1) -- (1,2) -- (7,2) -- (8,1);
\draw[ultra thick] (0,0) -- (1,1) -- (2.5,1);
\draw[ultra thick] (2,0) -- (3,1) -- (3.75,1);
\draw[ultra thick] (4.25,1) -- (5,1) -- (6,0);
\draw[ultra thick] (5.5,1) -- (7,1) -- (8,0);
\draw[ultra thick] (3,3) -- (3,4);
\draw[ultra thick] (5,3) -- (5,4);
\draw[ultra thick] (1.5,2) -- (1.5,3);
\draw[ultra thick] (6.5,2) -- (6.5,3);
\draw[ultra thick] (2,1) -- (2,2);
\draw[ultra thick] (3.5,1) -- (3.5,2);
\draw[ultra thick] (4.5,1) -- (4.5,2);
\draw[ultra thick] (6,1) -- (6,2);
\node at (-0.2,2.8) [] {$i_1$};
\node at (-0.2,1.8) [] {$i_2$};
\node at (-0.2,0.8) [] {$i_3$};
\node at (-0.2,-0.2) [] {$i_4$};
\node at (1.8,-0.2) [] {$i_5$};
\node at (8.2,2.8) [] {$j_1$};
\node at (8.2,1.8) [] {$j_2$};
\node at (8.2,0.8) [] {$j_3$};
\node at (8.2,-0.2) [] {$j_4$};
\node at (6.2,-0.2) [] {$j_5$};
\node at (4,4.5) [] {$0$};
\node at (2,3.5) [] {$1$};
\node at (4,3.5) [] {$1$};
\node at (6,3.5) [] {$1$};
\node at (1,2.5) [] {$1$};
\node at (4,2.5) [] {$2$};
\node at (7,2.5) [] {$1$};
\node at (1.5,1.5) [] {$2$};
\node at (2.75,1.5) [] {$2$};
\node at (4,1.5) [] {$3$};
\node at (5.25,1.5) [] {$2$};
\node at (6.5,1.5) [] {$2$};
}
\end{align*}
Combining all eight of them we conclude that
\begin{align*}
&\frac{\left[\Delta^{I\cup\{i_1\},J\cup\{j_3\}}\right]\left[\Delta^{I\cup\{i_3\},J\cup\{j_1\}}\right]\left[\Delta^{I\cup\{i_1,i_2,i_3\},J\cup\{j_1,j_3,j_5\}}\right]\left[\Delta^{I\cup\{i_1,i_3,i_5\},J\cup\{j_1,j_2,j_3\}}\right]}{\left[\Delta^{I\cup\{i_3\},J\cup\{j_3\}}\right]\left[\Delta^{I\cup\{i_1,i_2\},J\cup\{j_1,j_3\}}\right]\left[\Delta^{I\cup\{i_1,i_3\},J\cup\{j_1,j_2\}}\right]\left[\Delta^{I\cup\{i_1,i_3,i_5\},J\cup\{j_1,j_3,j_5\}}\right]}\\
=&\tikz[baseline=5ex,scale=0.4]{
\fill[fill=lightgray] (1.5,2) -- (6.5,2) -- (6.5,3) -- (1.5,3);
\draw[ultra thick] (0,3) -- (1,4) -- (7,4) -- (8,3);
\draw[ultra thick] (0,2) -- (1,3) -- (7,3) -- (8,2);
\draw[ultra thick] (0,1) -- (1,2) -- (7,2) -- (8,1);
\draw[ultra thick] (0,0) -- (1,1) -- (2.5,1);
\draw[ultra thick] (2,0) -- (3,1) -- (3.75,1);
\draw[ultra thick] (4.25,1) -- (5,1) -- (6,0);
\draw[ultra thick] (5.5,1) -- (7,1) -- (8,0);
\draw[ultra thick] (3,3) -- (3,4);
\draw[ultra thick] (5,3) -- (5,4);
\draw[ultra thick] (1.5,2) -- (1.5,3);
\draw[ultra thick] (6.5,2) -- (6.5,3);
\draw[ultra thick] (2,1) -- (2,2);
\draw[ultra thick] (3.5,1) -- (3.5,2);
\draw[ultra thick] (4.5,1) -- (4.5,2);
\draw[ultra thick] (6,1) -- (6,2);
\node at (-0.2,2.8) [] {$i_1$};
\node at (-0.2,1.8) [] {$i_2$};
\node at (-0.2,0.8) [] {$i_3$};
\node at (-0.2,-0.2) [] {$i_4$};
\node at (1.8,-0.2) [] {$i_5$};
\node at (8.2,2.8) [] {$j_1$};
\node at (8.2,1.8) [] {$j_2$};
\node at (8.2,0.8) [] {$j_3$};
\node at (8.2,-0.2) [] {$j_4$};
\node at (6.2,-0.2) [] {$j_5$};
\node at (4,4.5) [] {$0$};
\node at (2,3.5) [] {$0$};
\node at (4,3.5) [] {$0$};
\node at (6,3.5) [] {$0$};
\node at (1,2.5) [] {$0$};
\node at (4,2.5) [] {$-1$};
\node at (7,2.5) [] {$0$};
\node at (1.5,1.5) [] {$0$};
\node at (2.75,1.5) [] {$0$};
\node at (4,1.5) [] {$0$};
\node at (5.25,1.5) [] {$0$};
\node at (6.5,1.5) [] {$0$};
}\\
=&X_g^{-1},
\end{align*}
and hence
\[
\deg_{X_f}\left(\tilde{p}_{\vec{i}_{(u,v)}}\circ \tilde{\psi}_{\vec{i}_{(u,v)}}\circ\tilde{\chi}_{\vec{i}_{(u,v)}}\right)^*(X_g)=-\delta_{fg}.
\]
\end{itemize}

\vspace{12pt}

\noindent \textit{Proof of the remaining part of Proposition \ref{main'}.} We know from Proposition \ref{3.18} that $\psi\circ \chi$ is a cluster transformation, and we know from the computation above that $\deg_{X_f}\left(\psi_{\vec{i}_{(u,v)}}\circ \chi_{\vec{i}_{(u,v)}}\right)^*(X_g)=-\delta_{fg}$. Therefore by Proposition \ref{lem0} we can conclude that $\psi\circ \chi$ is a cluster Donaldson-Thomas transformation. \qed

\bibliographystyle{amsalpha-a}

\bibliography{biblio}

\end{document}